\documentclass{article}

\usepackage{amsthm, amsmath, amssymb}
\usepackage{thmtools}
\usepackage{soul}

\usepackage{mathrsfs}
\usepackage{hyperref}
\hypersetup{
    colorlinks,
    linkcolor={red!50!black},
    citecolor={blue!50!black},
    urlcolor={blue!80!black}
}

\usepackage[dvipsnames]{xcolor}
\usepackage{graphicx, multicol}
\usepackage{marvosym, wasysym}
\usepackage{fancyhdr}
\usepackage{stackengine}
\usepackage{algorithm}
\usepackage{bm}
\usepackage[noend]{algpseudocode}
\usepackage{bbm}
\usepackage{verbatim}
\makeatletter
\renewcommand{\ALG@beginalgorithmic}{\scriptsize}
\makeatother
\usepackage{xcolor}
\allowdisplaybreaks

\DeclareFontFamily{U}{mathx}{}
\DeclareFontShape{U}{mathx}{m}{n}{ <-> mathx10 }{}
\DeclareSymbolFont{mathx}{U}{mathx}{m}{n}
\DeclareFontSubstitution{U}{mathx}{m}{n}

\newif\ifextradetail
\extradetailtrue

\newif\ifshownavigationpage
\newif\ifshowreminders
\newif\ifshownotationindex
\newif\ifshowtheoremlinks
\newif\ifshowtheoremtree
\newif\ifshowtheoremlist
\newif\ifshowequationlist

\newif\ifshowcomments
\newif\ifhighlight 
\newif\ifelaborate

\newif\ifshowaddressedcomments
\newif\ifshowrvin
\newif\ifshowrvout

\ifhighlight
    
\else
    
\fi

\newcommand{\rvoutopacity}{20}
\ifshowrvin

\else

\fi

\ifshowrvout
    \newcommand{\rvout}[1]{{\color{red!\rvoutopacity}{#1}} }
    \newcommand{\chrout}[1]{{\color{blue!\rvoutopacity}{#1}} }    
    \newcommand{\rvoutm}[1]{{\color{black!\rvoutopacity}{\ifmmode\text{\sout{\ensuremath{\displaystyle#1}}}\else\sout{#1}\fi}} } 
\else
    \newcommand{\rvout}[1]{}
    \newcommand{\chrout}[1]{}
\fi

\ifshowcomments
    \newcommand{\summ}[1]{{\color{blue}[summary: #1]} } 
    \newcommand{\chr}[1]{{\color{PineGreen}[CR: #1]} } 
    \newcommand{\xw}[1]{{\color{RoyalBlue}[XW: #1]} } 
    \ifshowaddressedcomments
        \newcommand{\chra}[1]{{\color{PineGreen}\sout{[CR: #1]}} } 
        \newcommand{\xwa}[1]{{\color{RoyalBlue}\sout{[XW: #1]}} } 
    \else
        \newcommand{\chra}[1]{} 
        \newcommand{\xwa}[1]{} 
    \fi
\else
    \newcommand{\summ}[1]{} 
    \newcommand{\chr}[1]{} 
    \newcommand{\chra}[1]{} 
    \newcommand{\xw}[1]{} 
    \newcommand{\xwa}[1]{} 
\fi

\usepackage[shortlabels]{enumitem}

\newlist{thmdependence}{itemize}{10}
\setlist[thmdependence]{nosep,label=-}

\newcommand{\thmtreenode}[5]{\item[#1] \linkdest{location, thm tree #3} {#2}~\ref{#3} \linktopf{#3} \thmsum{#4}{#5}}

\newcommand{\thmtreeref}[2]{\item[\elsewhere] {{\hyperlink{location, thm tree #2}{\color{gray}#1}}}~\ref{#2}\thmsum{0.5}{}}

\ifshowtheoremlinks
    \newcommand{\linksinthm}[1]{\emph{\linkdest{location, #1}\linktopf{#1} \linktothmtree{location, thm tree #1} }}
    \newcommand{\linksinthmwopf}[1]{\emph{\linkdest{location, #1} \linktothmtree{location, thm tree #1} }}
    \newcommand{\linksinpf}[1]{\linkdest{location, proof of #1}\linktothm{#1} \linktothmtree{location, thm tree #1} }
\else
    \newcommand{\linksinthm}[1]{}
    \newcommand{\linksinthmwopf}[1]{}
    \newcommand{\linksinpf}[1]{}
\fi

\ifshownotationindex
    \newcommand{\notationdef}[2]{\linkdest{location, notation definition of #1}\hyperlink{location, notation index of #1}{#2}}
    \newcommand{\notationidx}[2]{\linkdest{location, notation index of #1}\hyperlink{location, notation definition of #1}{#2}}
\else
    \newcommand{\notationdef}[2]{#2}
\fi
    
\newcommand{\linktopf}[1]{\hyperlink{location, proof of #1}{\pflinksymbol}}
\newcommand{\linktothm}[1]{\hyperlink{location, #1}{\thmlinksymbol}}
\newcommand{\linktothmtree}[1]{\hyperlink{#1}{\thmtreelinksymbol}}
\newcommand{\thmlinksymbol}{{\tiny [Theorem]}}
\newcommand{\pflinksymbol}{{\tiny [Proof]}}
\newcommand{\thmtreelinksymbol}{{\tiny [ThmTree]}}

\newcommand{\elsewhere}{}

\newcommand{\thmsum}[2]{\quad{\color{gray}\begin{minipage}[t]{#1\linewidth}{#2}\vspace{0.5\baselineskip}\end{minipage}}}

\makeatletter
\newcommand{\linkdest}[1]{\Hy@raisedlink{\hypertarget{#1}{}}}
\makeatother

\newcommand{\elaborateopacity}{50}
\newcommand{\elaboratecolor}{RawSienna}
\ifelaborate
    \newcommand{\elaborate}[1]{{\color{\elaboratecolor!\elaborateopacity}{
    \begin{framed}
    \noindent {\footnotesize[Elaboration]}
    #1 
    \end{framed}
    }}\noindent}
\else
    \newcommand{\elaborate}[1]{}
\fi

\usepackage[margin=1.2in]{geometry}

\newtheorem{theorem}{Theorem}
\newtheorem{lemma}[theorem]{Lemma}

\newtheorem{proposition}[theorem]{Proposition}

\newtheorem{definition}{Definition}
\newtheorem{result}{Result}
\newtheorem{assumption}{Assumption}
\newtheorem{remark}{Remark}

\usepackage{mathtools}
\DeclarePairedDelimiter{\ceil}{\lceil}{\rceil}
\DeclarePairedDelimiter\floor{\lfloor}{\rfloor}

\newcommand{\D}{\mathbb D}

\newcommand{\cmt}[1]{#1} 
\renewcommand{\cmt}[1]{} 
\renewcommand{\P}{\mathbf{P}}
\newcommand{\Q}{\mathbf{Q}}
\newcommand{\E}{\mathbf{E}}
\newcommand{\RV}{\mathcal{RV}}
\newcommand{\R}{\mathbb{R}}

\renewcommand{\complement}{c}

\usepackage[normalem]{ulem}
\renewcommand{\rvout}[1]{{\color{red!\rvoutopacity}{#1}} }
\newcommand{\rvouta}[1]{}

\usepackage{multirow}
\usepackage{longtable}

\usepackage{array}
\def\delequal{\mathrel{\ensurestackMath{\stackon[1pt]{=}{\scriptscriptstyle\Delta}}}}
\def\distequal{\mathrel{\ensurestackMath{\stackon[1pt]{=}{\scriptstyle\mathcal{D}}}}}
\def\delequal{\mathrel{\ensurestackMath{\stackon[1pt]{=}{\scriptscriptstyle\Delta}}}}
\def\distequal{\mathrel{\ensurestackMath{\stackon[1pt]{=}{\scriptstyle d}}}}
\newcommand{\norm}[1]{\left\lVert#1\right\rVert}
\usepackage{mathtools}

\algrenewcommand\algorithmicrequire{\textbf{Require:}}
\algrenewcommand\algorithmicensure{\textbf{Postcondition:}}

\usepackage{subfigure}
\usepackage{float}
\usepackage{subfiles}
\title{Strongly Efficient Rare-Event Simulation for Regularly Varying L\'evy Processes with Infinite Activities}

\DeclareMathAccent{\widecheck}{0}{mathx}{"71}

\author{Xingyu Wang, Chang-Han Rhee}

\begin{document}


    
    



    




\maketitle

\begin{abstract}
\noindent
In this paper, we address rare-event simulation for heavy-tailed L\'evy processes with infinite activities.
The presence of infinite activities poses a critical challenge, making it impractical to simulate or store the precise sample path of the Lévy process.
We present a rare-event simulation algorithm that incorporates an importance sampling strategy based on heavy-tailed large deviations, the stick-breaking approximation for the extrema of L\'evy processes, 
the Asmussen-Rosiński approximation,
and the randomized debiasing technique. 
By establishing a novel characterization for the Lipschitz continuity of the law of L\'evy processes, we show that the proposed algorithm is unbiased and strongly efficient under mild conditions, and hence applicable to a broad class of L\'evy processes.
In numerical experiments, our algorithm demonstrates significant improvements in efficiency compared to the crude Monte-Carlo approach.
\end{abstract}

\counterwithin{equation}{section}
\counterwithin{lemma}{section}
\counterwithin{theorem}{section}
\counterwithin{proposition}{section}
\counterwithin{figure}{section}
\counterwithin{table}{section}



\section{Introduction}

In this paper, we propose a strongly efficient rare-event simulation algorithm for heavy-tailed L\'evy processes with infinite activities.
Specifically, the goal is to estimate probabilities of the form $\P(X \in A)$,
where $X = \{ X(t):\ t \in [0,1] \}$ is a L\'evy process in $\R$,
$A$ is a subset of the càdlàg space that doesn't include the typical path of $X$ so that $\P(X \in A)$ is close to $0$,
and the event $\{X \in A\}$ is ``unsimulatable'' due to the infinite number of activities within any finite time interval.
The defining features of the problem are as follows.
\begin{itemize}
    \item 
        The increments of the L\'evy process $X(t)$ are heavy-tailed. Throughout this paper, we characterize the heavy-tailed phenomenon through the notion of regular variation and assume that the tail cdf $\P(\pm X(t) > x)$ decays roughly at a power-law rate of $1/x^\alpha$;
        see Definition~\ref{def: RV} for details.
        The notion of heavy tails provides the mathematical formulation for the extreme uncertainty that manifests in a wide range of real-world dynamics and systems,
        including 
        the spread of COVID-19 (see, e.g., \cite{cohen2022covid}), 
        traffic in computer and communication networks (see, e.g., \cite{li:hal-01891760}),
        financial assets (see, e.g., \cite{embrechts2013modelling,Borak2011}),
        and the training of deep neural networks (see, e.g., \cite{pmlr-v139-gurbuzbalaban21a, hodgkinson2021multiplicative}).

    \item 
        $A$ is a general subset of $\D$ (i.e., the space of the real-valued càdlàg functions over $[0,1]$) that involves the supremum of the path.
        For concreteness in our presentation, the majority of the paper focuses on
        \begin{align}
            A = 
            \Big\{\xi \in \mathbb{D}: \sup_{t \in [0,1]}\xi(t)\geq {a}; \sup_{t \in (0,1] }\xi(t) - \xi(t-) < {b} \Big\}.
            \label{def set A, intro}
        \end{align}
        Intuitively speaking, this is closely related to ruin probabilities under reinsurance mechanisms, as $\{X \in A\}$ requires the supremum of the process $X(t)$ over $[0,1]$ to exceed some threshold $a$ even though all upward jumps in $X(t)$ are bounded by $b$.
        Nevertheless, we stress that the algorithmic framework proposed in this paper is flexible enough to address more general form of events $\{X \in A\}$ that are of practical interest.
        For instance, we demonstrate in Section~\ref{sec: barrier option pricing} of the Appendix that the framework can also address rare event simulation in the context of barrier option pricing.

    \item 
        $X(t)$ possesses infinite activities; see Section~\ref{subsec: SBA, preliminary} for the precise definition.
        Consequently, it is
        computationally infeasible to simulate or store the entire sample path of such processes.
        In other words, we focus on a computationally challenging case where $\mathbf{I}\{X \in A\}$ cannot be exactly simulated or evaluated.
        Addressing such ``unsimulatable'' cases is crucial 
        due to the increasing popularity of L\'evy models with infinite activities in
        risk management and mathematical finance (see, e.g., \cite{4fc30f5f-894d-3c6c-9b15-01f8f8d74820, https://doi.org/10.1111/1467-9965.00020, ROSINSKI2007677, bianchi2011tempered, risks10080148}), 
        as they offer more accurate and flexible descriptions for the price and volatility of financial assets compared to the classical jump-diffusion models (see, e.g., \cite{https://doi.org/10.1111/mafi.12055}).
\end{itemize}
In summary, our goal is to tackle a practically significant yet computationally challenging task, where the nature of the rare events renders crude Monte Carlo methods highly inefficient, if not entirely infeasible, due to the infinite activities in $X(t)$.
To address these challenges, we integrate several mathematical machinery:
a design of importance sampling based on sample-path large deviations for heavy-tailed L\'evy processes in \cite{rhee2019sample},
the stick-breaking approximation in \cite{cazares2018geometrically} for L\'evy processes with infinite activities, 
and the randomized multilevel Monte Carlo debiasing technique in \cite{rhee2015unbiased}.
By combining these tools, we propose a rare event simulation algorithm for heavy-tailed L\'evy processes with infinite activities that attains strong efficiency
(see Definition~\ref{def: strong efficiency} for details).

As mentioned above, the first challenge is rooted in the nature of rare events as the crude Monte Carlo methods can be prohibitively expensive when estimating a small $p = \P(X \in A)$.
Instead, variance reduction techniques are often employed for efficient rare event simulation.
When the underlying uncertainties are light-tailed, the exponential tilting strategy guided by large deviation theories has been successfully applied in a variety of contexts; see, e.g., \cite{bucklew_ney_sadowsky_1990,boxma2018linear,TORRISI2004225,dupuis2007dynamic}.
However, the exponential tilting approach falls short in providing a principled and provably efficient design of the importance sampling estimators (see, for example, \cite{BASSAMBOO2007251})
due to fundamentally different mechanisms through which the rare events occur.
Instead, different importance sampling strategies (e.g., \cite{10.1214/07-AAP485,10.1145/1243991.1243995,blanchet2008efficient,blanchet_liu_2008,doi:10.1287/13-SSY114,10.1145/2517451})
and other variance reduction techniques such as conditional Monte Carlo (e.g., \cite{asmussen_kroese_2006,hult2016exact}) 
and Markov Chain Monte Carlo (e.g., \cite{gudmundsson_hult_2014}) have been proposed to address problems associated with specific types of processes or events.

Recent developments in heavy-tailed large deviations, such as those by \cite{rhee2019sample} and \cite{wang2023large},
offer critical insights into the design of efficient and universal importance sampling schemes for heavy-tailed systems.
Central to this development is the discrete hierarchy of heavy-tailed rare events, known as the {catastrophe principle}. The principle dictates
that rare events in heavy-tailed systems arise due to catastrophic failures of a small number of system components, and the number of such components governs the asymptotic rate at which the associated rare events occur.
This creates a discrete hierarchy among heavy-tailed rare events. 
By combining the defensive importance sampling design with such hierarchy, strongly efficient importance sampling algorithms have been proposed for a variety of rare events associated with random walks and compound Poisson processes in \cite{chen2019efficient}.
See also \cite{10.5555/3643142.3643148} for a tutorial on this topic.
In this paper, we adopt and extend this framework to encompass L\'evy processes with infinite activities.
The specifics of the importance sampling distribution are detailed in Section~\ref{subsec: importance sampling dist, algo}.

Another challenge arises from the simulation of L\'evy processes with infinite activities.
While the design of importance sampling algorithm in \cite{chen2019efficient} has been successfully applied to a wide range of stochastic systems that are exactly simulatable (including random walks, compound Poisson processes, 
iterates of stochastic gradient descent,
and several classes of queueing systems),
it cannot be implemented for L\'evy processes with infinite activities.
More specifically,
the simulation of the random vector $(X(t),\ M(t))$, where $M(t) = \sup_{s \leq t}X(t)$, poses a significant challenge in the case with infinite activities.
As of now, exact simulation of the extrema of L\'evy processes (excluding the compound Poisson case) is only available for 
specific cases (see, for instance, \cite{gonzalez2019exact,cázares2023fast,10.1214/20-EJP503}),
let alone the exact simulation of the joint law of $(X(t),\ M(t))$.
We therefore approach the challenge by considering the following questions:
$(i)$ Does there exist a provably efficient approximation algorithm for $(X(t),\ M(t))$,
and 
$(ii)$ Are we able to remove the approximation bias while still attaining strong efficiency in our rare-event simulation algorithm?

Regarding the first question, several classes of algorithms have been proposed for the approximate simulation of the extrema of L\'evy processes.
This includes the random walk approximations based on Euler-type discretization of the process (see e.g., \cite{10.1214/aoap/1177004597,Dia_Lamberton_2011,giles2017multilevel}),
the Wiener-Hopf approximation methods (see e.g. \cite{10.1214/10-AAP746,FERREIROCASTILLA2014985}) based on the fluctuation theory of L\'evy processes,
the jump-adapted Gaussian approximations (see e.g. \cite{10.1214/10-AAP695,DEREICH20111565}),
and the characteristic function approach in \cite{boyarchenko2023efficientevaluationjointpdf,boyarchenko2023simulationlevyprocessextremum}
based on efficient evaluation of joint cdf.
Nevertheless, the approximation errors in the aforementioned methods are either unavailable or exhibit a polynomial rate of decay.
Thankfully, the recently developed stick-breaking approximation (SBA) algorithm in \cite{cazares2018geometrically} provides a novel approach to the simulation of the joint law of $X(t)$ and $M(t)$.
The theoretical foundation of SBA
is the following
description for the concave majorants of L\'evy processes with infinite activities in \cite{pitman2012convex}:
$$
\big(X(t),\ M(t)\big) \distequal \Big( \sum_{j \geq 1}\xi_j,\ \sum_{j \geq 1}\max\{\xi_j,0\}\Big).
$$
Here, $(l_j)_{j \geq 1}$ is a sequence of iteratively generated non-negative RVs satisfying $\sum_{j \geq 1}l_j = t$ and $\E l_j = t/2^j\ \forall j \geq 1$;
conditioned on the values of $(l_j)_{j \geq 1}$, $\xi_j$'s are independently generated such that $\xi_j \distequal X(l_j)$.
While it is computationally infeasible to generate the entirety of the infinite sequences $(l_j)_{j \geq 1}$ and $(\xi_j)_{j \geq 1}$,
by terminating the procedure at the $m$-th step we yield
approximations of the form
\begin{align}
    \big( \hat X_m(t),\ \hat M_m(t) \big)
    \delequal
    \Big(
        \sum_{j = 1}^m \xi_j,\ \sum_{j = 1}^m \max\{\xi_j,0\}
    \Big).
    \label{def: SBA, intro}
\end{align}
We provide a review in Section~\ref{subsec: SBA, preliminary}.
In particular, due to $\E \big[\sum_{j > m}l_j\big] = t/2^m$, 
with each extra step in \eqref{def: SBA, intro} we expect to reduce the approximation error by half, thus leading to the geometric convergence rate of errors.
See \cite{cazares2018geometrically} for analyses of the approximation errors for different types of functionals.

Additionally, while SBA can be considered sufficiently accurate for a wide range of tasks,
eliminating the approximation errors is crucial in the context of 
rare-event simulation.
Otherwise, 
any effort to efficiently estimate a small probability might be fruitless and could be overwhelmed by potentially large errors in the algorithm.
In order to remove 
the approximation errors of SBA in \eqref{def: SBA, intro},
we employ the construction of unbiased estimators proposed in \cite{rhee2015unbiased}.
This can be interpreted as a randomized version of the multilevel Monte Carlo scheme \cite{hei01,gil08} when a sequence of biased yet increasingly more accurate approximations is available.
It allows us to construct an unbiased estimation algorithm that terminates within a finite number of steps.
By combining SBA, the randomized debiasing technique, and the design of importance sampling distributions based on heavy-tailed large deviations,
we propose Algorithm~\ref{algoISnoARA} for rare-event simulation of L\'evy processes with infinite activities.
In case that 
the exact sampling of $X(t)$, and hence the increments $\xi_j$'s in \eqref{def: SBA, intro}, is not available, 
we further incorporate the Asmussen-Rosiński approximation (ARA) in \cite{asmussen2001approximations}.
This approximation replaces the small-jump martingale in the L\'evy process $X(t)$ with a Brownian motion of the same variance, thus leading to Algorithm~\ref{algoIS}.
We note that the combination of SBA and the randomized debiasing technique has been explored in \cite{cazares2018geometrically},
and an ARA-incorporated version of SBA has been proposed in \cite{gonzalez2022simulation}.
However, the goal of proposing strongly efficient rare-event simulation algorithm 
adds another layer of difficulty
and sets our work apart from the existing literature.
In particular, 
the notion of strong efficiency demands that
the proposed estimator remains efficient under the importance sampling algorithm w.r.t.\ not just a given task, but throughout a sequence of increasingly more challenging rare-event simulation tasks as $\P(X \in A)$ tends to $0$.
This introduces a new dimension into the theoretical analysis that is not presented in 
\cite{cazares2018geometrically,gonzalez2022simulation}
and necessitates the development of new technical tools to characterize the performance of the algorithm when all these components (importance sampling, SBA, debiasing technique, and ARA) are in effect.

An important technical question in our analysis concerns the continuity of the law of the running supremum $M(t)$.
To provide high-level descriptions,
let us consider estimators for $\P(X \in A) = \E\big[\mathbf{I}\{X \in A\}\big]$
that admit the form $f(\hat X)$ where $\hat X$ is some approximation to the L\'evy process $X$
and 
$
f(\xi) = \mathbf{I}\{\xi \in A\}.
$
SBA and the debiasing technique allow us to construct $\hat X$ such that the deviation $\hat X - X$ has a small variance.
Nevertheless, the estimation can be fallible if $X$ concentrates on the boundary cases,
i.e., $X$ falls into a neighborhood of $\partial A$ fairly often.
Specializing to the case in \eqref{def set A, intro},
this requires obtaining sufficiently tight bounds for 
probabilities of form $\P(M(t) \in [x,x+\delta])$.
Nevertheless, the continuity of the law of the supremum $M(t)$ remains an active area of study, with many essential questions left open.
Recent developments regarding the law of $M(t)$ are mostly qualitative or focus on the cumulative distribution function (cdf);
see, e.g., \cite{10.1214/11-AOP708,coutinpontierngom2018,10.1214/11-AOP719,michna2012formula,michna2014distribution,10.1214/ECP.v18-2236}.
In short, addressing this aspect of the challenge requires us to establish novel and useful quantitative characterizations of the law of supremum $M(t)$. 

For our purpose of efficient rare event simulation, particularly under the importance sampling scheme detailed in Section~\ref{subsec: importance sampling dist, algo},
the following condition proves to be sufficient: 
\begin{align}
    \P\Big(X^{<z}(t) \in [x, x + \delta]\Big) \leq \frac{C}{t^\lambda\wedge 1}\delta\qquad \forall z \geq z_0,\ t > 0,\ x \in \R,\ \delta \in [0,1].
    \label{assumption, unif holder cont, intro}
\end{align}
Here, $X^{<z}(t)$ is a modulated version of the process $X(t)$ where all the upward jumps with sizes larger than $z$ are removed;
see Section~\ref{sec: algorithm} for the rigorous definition.
First, we establish in Theorem~\ref{theorem: strong efficiency without ARA} (resp. Theorem~\ref{theorem: strong efficiency}) that
Algorithm~\ref{algoISnoARA} (resp. Algorithm~\ref{algoIS}) does attain strong efficiency under condition~\eqref{assumption, unif holder cont, intro}.
More importantly, we demonstrate in Section~\ref{sec: lipschitz cont, A2} that
condition \eqref{assumption, unif holder cont, intro} is mild for L\'evy processes with infinitive activities, as it only requires the intensity of jumps to approach $\infty$ (hence attaining infinite activities in $X$) at a rate that is not too slow.
In particular, in Theorems~\ref{ CorollaryRVlevyMeasureAtOrigin } and \ref{ CorollarySemiStableLevyMeasureAtOrigin } we provide two sets of sufficient conditions for \eqref{assumption, unif holder cont, intro} that are easy to verify.
We note that
the representation of concave majorants for L\'evy processes developed in \cite{pitman2012convex}
proves to be a valuable tool for studying the law of $X(t)$ and $M(t)$.
As will be elaborated in the proofs in Section~\ref{sec: proof},
the key technical tool that allows us to connect condition~\ref{assumption, unif holder cont, intro} with the law of the supremum $M(t)$ is, again, the 
representation in \eqref{def: SBA, intro}.
See also
\cite{10.3150/23-BEJ1590}
for its application in studying the joint density of $X(t)$ and $M(t)$ of stable processes.

Some algorithmic contributions of this paper were presented in a preliminary form at a conference in \cite{10.5555/3466184.3466229} without rigorous proofs.
The current paper presents several significant extensions:
$(i)$ In addition to Algorithm~\ref{algoISnoARA}, we also propose an ARA-incorporated version of the importance sampling algorithm (see Algorithm~\ref{algoIS}) to address the case where $X(t)$ cannot be exactly simulated;
$(ii)$ Rigorous proofs of strong efficiency are provided in Section~\ref{sec: proof} in this paper; 
$(iii)$ We establish two sets of sufficient conditions for \eqref{assumption, unif holder cont, intro} in Section~\ref{sec: lipschitz cont, A2}, leveraging the properties of regularly varying or semi-stable processes.

The rest of the paper is structured as follows.
Section~\ref{sec: notations, preliminaries} reviews the theoretical foundations of our algorithms, including the heavy-tailed large deviation theories (Section \ref{subsec: review, LD of levy}), 
the stick-breaking approximations (Section \ref{subsec: SBA, preliminary}),
and the debiasing technique (Section \ref{subsec: unbiased estimation}).
Section~\ref{sec: algorithm} presents the importance sampling algorithms and establishes their strong efficiency.
Section~\ref{sec: lipschitz cont, A2} investigates the continuity of the law of $X(t)$ and provides sufficient conditions for \eqref{assumption, unif holder cont, intro},
a critical condition to ensure the strong efficiency of our importance sampling scheme.
Numerical experiments are reported in Section~\ref{sec: experiment}.
The proofs of all technical results are collected in
Section~\ref{sec: proof}.
In the Appendix,
Section~\ref{sec: barrier option pricing} extends the algorithmic framework to the context of barrier option pricing.

\section{Preliminaries}
\label{sec: notations, preliminaries}

In this section, we introduce some notations and results that will be frequently used when developing the strongly efficient rare-event simulation algorithm.

\subsection{Notations}
\label{subsec: notations}
Let $\mathbb N = \{0,1,2,\ldots\}$ be the set of non-negative integers.
For any positive integer $k$, let $\notationdef{set-for-integers-below-n}{[k]} = \{1,2,\ldots,k\}$.
For any $x,y \in \R$, let $\notationdef{wedge-min-operator}{x \wedge y} \delequal \min\{x,y\}$  and $\notationdef{vee-max-operator}{x\vee y} \delequal \max\{x,y\}$.
For any $x\in\mathbb{R}$, we define $(x)^+ \delequal x \vee 0$ as the positive part of $x$, and
$$
\notationdef{floor-operator}{\floor{x}} \delequal \max\{ n \in \mathbb{Z}:\ n \leq x \},\qquad 
    \notationdef{ceil-operator}{\ceil{x}} \delequal  \min\{n \in \mathbb{Z}:\ n \geq x\}
$$
as the floor and ceiling function.
Given a measure space $(\mathcal{X},\mathcal{F},\mu)$ and any set $A \in \mathcal{F}$, we use $\notationdef{measure-restricted-on-A}{\mu|_A( \cdot )} \delequal \mu( A \cap \cdot )$ to denote restriction of the measure $\mu$ on $A$.
For any random variable $X$ and any Borel measureable set $A$,
let $\notationdef{law-of-X}{\mathscr{L}(X)}$ be the law of $X$, and $\notationdef{law-of-X-on-A}{\mathscr{L}(X|A)}$ be the law of $X$ conditioned on event $A$.
Let $(\mathbb{D}_{[0,1],\R},\bm{d})$ be the metric space of $\notationdef{cadlag-space-D-01}{\mathbb{D}} = \mathbb{D}_{[0,1],\R}$ (i.e., the space of all real-valued càdlàg functions with domain $[0,1]$) equipped with Skorokhod $J_1$ metric $\bm{d}$.
Here, the metric $\bm d$ is defined by
\begin{align}
    \notationdef{skorokhod-j1-metric-d}{\bm{d}(x,y)} \delequal \inf_{\lambda \in \Lambda} \sup_{t \in [0,1]}|\lambda(t) - t| \vee |x(\lambda(t)) - y(t)|
    \label{def: J 1 metric}
\end{align}
with $\Lambda$ being the set of all increasing homeomorphisms from $[0,1]$ to itself. 

Henceforth in this paper, the heavy-tailedness of any random element will be captured by the notion of
regular variation.
\begin{definition}\label{def: RV}
    For any measurable function $\phi:(0,\infty) \to (0,\infty)$, we say that $\phi$ is \textbf{regularly varying} as $x \rightarrow\infty$ with index $\beta$ (denoted as $\phi(x) \in  \notationdef{notation-RV-LDP}{\RV_\beta}(x)$ as $x\to \infty$) if $\lim_{x \rightarrow \infty}\phi(tx)/\phi(x) = t^\beta$ for all $t>0$. 
     We also say that a measurable function $\phi(\eta)$
is {regularly varying} as $\eta \downarrow 0$ with index $\beta$ 
if $\lim_{\eta \downarrow 0} \phi(t\eta)/\phi(\eta) = t^\beta$ for any $t > 0$.
We denote this as $\phi(\eta) \in {\RV_{\beta}}(\eta)$ as $\eta \downarrow 0$.
\end{definition}
\noindent
For properties of regularly varying functions, see, for example, Chapter 2 of \cite{resnick2007heavy}.

Next, we discuss the L\'evy-Ito decomposition of one-dimensional L\'evy processes, i.e., $X(t) \in \R$.
The law of a one-dimensional L\'evy process $\{X(t):t \geq 0\}$ is completely characterized by its generating triplet $(c,\sigma,\nu)$ where $c\in \R$ represents the constant drift, $\sigma \geq 0$ is the magnitude of the Brownian motion term,
and the L\'evy measure $\nu$ characterizes the intensity of the jumps. 
More precisely,
\begin{align}
    X(t) \distequal ct + \sigma B(t)+ \int_{|x| \leq 1}x[ N([0,t]\times dx) - t\nu(dx) ] + \int_{|x| > 1}xN( [0,t]\times dx ) 
    \label{prelim: levy ito decomp}
\end{align}
where $\text{Leb}(\cdot)$ is the Lebesgue measure on $\R$, $B$ is a standard Brownian motion, the measure $\nu$ satisfies $\int (|x|^2\wedge 1) \nu(dx) < \infty$,
and $N$ is a Poisson random measure over $(0,\infty) \times \R$ with intensity measure $\text{Leb}((0,\infty))\times \nu$ and is independent of $B$.
For standard references on this topic, see Chapter 4 of \cite{sato1999levy}.

Given two sequences of non-negative real numbers $(x_n)_{n \geq 1}$ and $(y_n)_{n \geq 1}$, 
we say that $x_n = \bm{O}(y_n)$ (as $n \to \infty$) if there exists some $C \in [0,\infty)$ such that $x_n \leq C y_n\ \forall n\geq 1$.
Besides, we say that $x_n = \bm{o}(y_n)$ if $\lim_{n \rightarrow \infty} x_n/y_n = 0$.
The goal of this paper is described in the following definition of strong efficiency.
\begin{definition}\label{def: strong efficiency}
Let $(L_n)_{n \geq 1}$ be a sequence of random variables supported on a probability space $(\Omega,\mathcal{F},\P)$ and $(A_n)_{n \geq 1}$ be a sequence of events (i.e., $A_n \in \mathcal F\ \forall n$). 
We say that $(L_n)_{n \geq 1}$ are \notationdef{def-strong-efficiency}{unbiased and strongly efficient} estimators of $(A_n)_{n \geq 1}$ if
$$
\E L_n = \P(A_n) \ \forall n \geq 1;\qquad 
    \E L^2_n = \bm{O}\big(\P^2(A_n)\big)\ \text{ as }n \rightarrow \infty.
$$
\end{definition}
\noindent
We stress again that 
strongly efficient estimators $(L_n)_{n \geq 1}$ achieve uniformly bounded relative errors (i.e., the ratio between standard error and mean) for all $n \geq 1$.

\subsection{Sample-Path Large Deviations for Regularly Varying L\'evy Processes}
\label{subsec: review, LD of levy}

The key ingredient of our importance sampling algorithm is the recent development of the sample-path large deviations for  L\'evy processes with regularly varying increments; see \cite{rhee2019sample}.
To familiarize the readers with this mathematical machinery,
we start by reviewing the results in the one-sided cases, and then move onto the more general two-sided results.

Let $X(t) \in \R$ be a centered L\'evy process (i.e., $\E X(t) = 0\ \forall t > 0$) with generating triplet $(c,\sigma,\nu)$ such that the L\'evy measure $\nu$ is supported on $(0,\infty)$.
In other words, all the discontinuities in $X$ will be positive, hence one-sided.
Moreover, we are interested in the heavy-tailed setting where
the function $H_+(x) = \nu[x,\infty)$ is regularly varying as $x \rightarrow \infty$ with index $-\alpha$ where $\alpha > 1$.
Define a scaled version of the process as $\bar X_n(t) \delequal \frac{1}{n} X(nt)$,
and let $\bar X_n \delequal \{\bar X_n(t):\ t \in [0,1]\}$.
Note that $\bar X_n$ is a random element taking values in $\mathbb D$.

For all $l \geq 1$, let $\notationdef{set-D-l}{\mathbb{D}_l}$ be the subset of $\mathbb D$ containing all the non-decreasing step functions that has $l$ jumps (i.e., discontinuities) and vanishes at the origin.
Let $\mathbb{D}_0 = \{\bm{0}\}$ be the set that only contains the zero function $\bm{0}(t) \equiv 0$.
Let $\notationdef{set-D-<-l}{\mathbb D_{<l}} = \cup_{j = 0,1,\cdots,l-1}\mathbb{D}_l$.
For any $\beta > 0$, let $\notationdef{def-measure-nu-beta}{\nu_\beta}$ be the measure supported on $(0,\infty)$ with $\nu_\beta(x,\infty) = x^{-\beta}$.
For any positive integer $l$, let $\notationdef{def-nu-beta-l-fold}{\nu^l_\beta}$ be the $l-$fold product measure of $\nu_\beta$ restricted on $\{ \bm y = (y_1,\ldots,y_l) \in (0,\infty)^l:\ y_1 \geq y_2 \geq \cdots \geq y_l \}$.
Define the measure (for $l \geq 1$)
$$
\notationdef{measure-C-l-beta}{\mathbf C_l(\cdot)} \delequal \E\Bigg[ \nu_\beta^l\Big\{ \bm y \in (0,\infty)^l:\ \sum_{j = 1}^l y_j \mathbf{I}_{[U_j,1]}\in\cdot\Big\} \Bigg]
$$
where all $U_j$'s are iid copies of $\text{Unif}(0,1)$.
In case that $l = 0$, we set $\mathbf C^0_{\beta}$ as the Dirac measure on $\textbf{0}$. 
The following result provides sharp asymptotics for rare events associated with $\bar X_n$.
Henceforth in this paper, all measurable sets are understood to be Borel measurable.

\begin{result}[Theorem 3.1 of \cite{rhee2019sample}]
\label{result: LD of Levy, one-sided}
Let $A \subset \D$ be measurable.
Suppose that $\mathcal J(A) \delequal \min\{j \in \mathbb{N}:\ \mathbb{D}_j \cap A \neq \emptyset \} < \infty$ and
$A$ is bounded away from $\mathbb D_{<\mathcal J(A)}$ in the sense that $\bm{d}(A,\mathbb D_{<\mathcal J(A)}) > 0$.
Then
$$\mathbf C_{\mathcal J(A)}(A^\circ) \leq \liminf_{ n \rightarrow \infty }\frac{\P(\Bar{X}_n \in A)}{(n\nu[n,\infty) )^{\mathcal J(A)}}\leq \limsup_{ n \rightarrow \infty }\frac{\P(\Bar{X}_n \in A)}{(n\nu[n,\infty) )^{\mathcal J(A)}}\leq \mathbf C_{\mathcal J(A)}(A^-) < \infty $$
where $A^\circ,A^-$ are the interior and closure of $A$ respectively.
\end{result}
Intuitively speaking, Result~\ref{result: LD of Levy, one-sided} embodies a general principle that, in heavy-tailed systems, rare events arise due to several ``large jumps''. 
Here, $\mathcal J(A)$ denotes the minimum number of jumps required in $\bar X_n$ for event $\{\bar X_n \in A\}$ to occur.
As shown above in Result~\ref{result: LD of Levy, one-sided}, $\mathcal J(A)$ dictates the polynomial rate of decay of the probabilities of the rare events $\P(\bar X_n \in A)$.
Furthermore, results such as Corollary 4.1 in \cite{rhee2019sample} characterize the conditional limits of $\bar X_n$:
conditioning on the occurrence of rare events $\{ \bar X_n \in A\}$, the conditional law $\mathscr{L}(\bar X_n | \{ \bar X_n \in A\} )$
converges in distribution to that of a step function over $[0,1]$ with exactly $\mathcal J(A)$ jumps (of random sizes and arrival times) as $n \to \infty$.
Therefore, $\mathcal J(A)$ also dictates the most likely scenarios of the rare events. 
This insight proves to be critical when we develop the importance sampling distributions for the rare events simulation algorithm in Section~\ref{sec: algorithm}.

Results for the two-sided cases admit a similar yet slightly more involved form,
where the L\'evy process $X(t)$ exhibits both positive and negative jumps.
Specifically, let $X(t)$ be a centered L\'evy process such that for $H_+(x) = \nu[x,\infty)$ and $H_-(x) = \nu(-\infty,-x]$,
we have $H_+(x) \in \RV_{-\alpha}(x)$ and $H_-(x) \in \RV_{-\alpha^\prime}(x)$ as $x\rightarrow\infty$ for some $\alpha,\alpha^\prime > 1$.
Let $\notationdef{def-set-D-j,k}{\mathbb D_{j,k}}$ be the set containing all step functions in $\mathbb D$ vanishing at the origin that has exactly $j$ upward jumps and $k$ downward jumps.
As a convention, let $\mathbb D_{0,0} = \{\bm 0\}$.
Given $\alpha,\alpha^\prime > 1$, let 
$
\notationdef{def-set-D-<-j,k}{\mathbb D_{<j,k}} \delequal \bigcup_{ (l,m) \in \mathbb I_{<j,k} }\mathbb D_{l,m}
$
where 
$\notationdef{index-set-I-<-j,k}{\mathbb I_{<j,k}} \delequal \big\{ (l,m) \in \mathbb N^2 \symbol{92} \{(j,k)\}:\ l(\alpha-1) + m(\alpha^\prime - 1) \leq j(\alpha - 1) + k(\alpha^\prime - 1) \big\}.$
Let $\mathbf C_{0,0}$ be the Dirac measure on $\bm 0$.
For any $(j,k) \in \mathbb N^2 \symbol{92}\{(0,0)\}$ let
\begin{align}
    \notationdef{measure-C-j,k}{\mathbf C_{j,k}(\cdot)}
    \delequal 
    \E\Bigg[
    \nu^j_\alpha \times \nu^k_{\alpha^\prime}
    \bigg\{
    (\bm x,\bm y) \in (0,\infty)^j \times (0,\infty)^k:\ 
    \sum_{l = 1}^j x_l \mathbf{I}_{ [U_l,1] } - \sum_{m = 1}^k y_m \mathbf{I}_{[V_m,1]} \in \cdot  
    \bigg\}
    \Bigg]
    \label{def: measure C j k}
\end{align}
where all $U_l$'s and $V_m$'s are iid copies of Unif$(0,1)$ RVs.
Now, we are ready to state the two-sided result.

\begin{result}[Theorem 3.4 of \cite{rhee2019sample}]
\label{result: LD of Levy, two-sided}
Let $A \subset \D$ be measurable.
Suppose that
\begin{align}
    \big(\mathcal{J}(A),\mathcal{K}(A)\big)
    \in 
    \underset{ (j,k) \in \mathbb N^2,\ \mathbb D_{j,k} \cap A \neq \emptyset }{\text{argmin}}
    j(\alpha - 1) + k(\alpha^\prime - 1)
    \label{def: argument minimum LD two sided case}
\end{align}
and $A$ is bounded away from $\mathbb D_{<\mathcal{J}(A),\mathcal{K}(A)}$. Then the argument minimum in \eqref{def: argument minimum LD two sided case} is unique, and
\begin{align*}
    \mathbf{C}_{\mathcal{J}(A),\mathcal{K}(A)}(A^\circ) 
    & \leq 
    \liminf_{n \rightarrow \infty}\frac{ \P(\bar X_n \in A) }{ (n\nu[n,\infty))^{ \mathcal{J}(A) }(n\nu(-\infty,-n])^{ \mathcal{K}(A) } }
    \\
    & \leq 
    \limsup_{n \rightarrow \infty}\frac{ \P(\bar X_n \in A) }{ (n\nu[n,\infty))^{ \mathcal{J}(A) }(n\nu(-\infty,-n])^{ \mathcal{K}(A) } }
    \leq 
    \mathbf{C}_{\mathcal{J}(A),\mathcal{K}(A)}(A^-)  < \infty
\end{align*}
where $A^\circ,A^-$ are the interior and closure of $A$ respectively.
\end{result}

\subsection{Concave Majorants and Stick-Breaking Approximations of L\'evy Processes with Infinite Activities}
\label{subsec: SBA, preliminary}

Next, we review the distribution of the concave majorant of a L\'evy process with infinite activities characterized in \cite{pitman2012convex},
which paves the way to the stick-breaking approximation algorithm proposed in \cite{cazares2018geometrically}.
Let $X(t)$ be a L\'evy process  with generating triplet $(c,\sigma,\nu)$. We say that $X$ has \textbf{infinite activities} if $\sigma > 0$ or $\nu(\R) = \infty$.
Let $M(t) \delequal \sup_{s \leq t}X(s)$ be the running supremum of $X(t)$.
The results in \cite{pitman2012convex} establishes a Poisson–Dirichlet distribution that underlies the joint law of $X(t)$ and $M(t)$.
Specifically, we fix some $T > 0$ and let
$V_i$'s be iid copies of Unif$(0,1)$ RVs.
Recursively, let
\begin{align}
    l_1 = TV_1,\qquad l_j = V_j\cdot (T - l_1 - l_2 -\ldots - l_{j-1})\quad \forall j \geq 2.
    \label{prelim: stick breaking, generate sticks l i}
\end{align}
Conditioning on the values of $(l_j)_{j \geq 1}$, let $\xi_j$ be a random copy of $X(l_j)$, with all $\xi_j$ being independently generated.

\begin{result}[Theorem 1 in \cite{pitman2012convex}]
\label{result: concave majorant of Levy}
Suppose that the L\'evy process $X$ has infinite activities. Then (with $(x)^+ = \max\{x,0\}$)
\begin{align}
    \big(X(T),M(T)\big) \distequal \big( \sum_{j \geq 1}\xi_j, \sum_{j \geq 1}(\xi_j)^+ \big).
    \label{claim, stick breaking representation}
\end{align}
\end{result}

Based on the distribution characterized in \eqref{claim, stick breaking representation},
the stick-breaking approximation algorithm was proposed in \cite{cazares2018geometrically} 
where finitely many $\xi_i$'s are generated in order to approximate $X(T)$ and $M(T)$.
This approximation technique is a key component of our rare event simulation algorithm.
In particular, we utilize a coupling between different L\'evy processes based on the representation \eqref{claim, stick breaking representation} above.
For clarity of our description, we focus on two L\'evy processes $X$ and $\widetilde X$ with generating triplets $(c,\sigma,\nu)$ and $(\widetilde c, \widetilde \sigma, \widetilde \nu)$, respectively.
Suppose that both $X$ and $\widetilde X$ have infinite activities. 
We first generate $l_i$'s as described in \eqref{prelim: stick breaking, generate sticks l i}.
Conditioning on the values of $(l_i)_{i \geq 1}$, we then independently generate $\xi_i$ and $\widetilde \xi_i$, which are random copies of $X(l_i)$ and $\widetilde X(l_i)$, respectively.
Let $\widetilde M(t) \delequal \sup_{s \leq t}\widetilde X(s)$.
Applying Result~\ref{result: concave majorant of Levy}, we identify a coupling between $X(T),M(T),\widetilde X(T),\widetilde M(T)$ such that
\begin{align}
    \big( X(T),M(T),\widetilde X(T),\widetilde M(T) \big) 
    \distequal 
    \big( \sum_{i \geq 1}\xi_i, \sum_{i \geq 1}(\xi_i)^+,\sum_{i \geq 1}\widetilde \xi_i, \sum_{i \geq 1}(\widetilde \xi_i)^+\big).
    \label{prelim: SBA, coupling}
\end{align}

\begin{remark}
It is worth noticing that the method described above in fact 
implies the existence of a probability space $(\Omega,\mathcal{F},\P)$ that supports the entire sample paths $\{X(t):\ t \in [0,T]\}$ and $\{\widetilde X(t):\ t \in [0,T]\}$, 
whose endpoint values $X(T),\widetilde X(T)$ and suprema $M(T),\widetilde M(T)$ admit the joint law in \eqref{prelim: SBA, coupling}.
In particular, once we obtain $l_i$ based on \eqref{prelim: stick breaking, generate sticks l i},
one can generate $\Xi_i$ that are iid copies of the entire paths of $X$.
That is,
we generate a piece of sample path $\Xi_i$ on the stick $l_i$, and the quantities 
$\xi_i$ introduced earlier can be obtained by setting $\xi_i = \Xi_i(l_i)$.
To recover the sample path of $X$ based on the pieces $\Xi_i$,
it suffices to apply Vervatt transform onto each $\Xi_i$ and then reorder the pieces based on their slopes.
We refer the readers to theorem 4 in \cite{pitman2012convex}.
In summary, 
the method described above leads to a coupling between the sample paths of the underlying L\'evy processes $X$ and $\widetilde X$ such that \eqref{prelim: SBA, coupling} holds.
\end{remark}

\subsection{Randomized Debiasing Technique}
\label{subsec: unbiased estimation}

To achieve unbiasedness in our algorithm and remove the errors in the stick-breaking approximations, we apply the randomized multi-level Monte-Carlo technique studied in \cite{rhee2015unbiased}.
In particular, due to $\tau$ being finite (almost surely) in Result \ref{resultDebias} below,
the simulation of $Z$ relies only on $Y_0,Y_1,\cdots,Y_\tau$ instead of the infinite sequence $(Y_n)_{n \geq 0}$.

\begin{result}[Theorem 1 in \cite{rhee2015unbiased}]\label{resultDebias} 
Let random variables $Y$ and $(Y_m)_{m \geq 0}$ be such that $\lim_{m \rightarrow \infty}\E Y_m = \E Y$.
Let $\tau$ be a positive integer-valued random variable with unbounded support, independent of $(Y_m)_{m\geq 0}$ and $Y$.
Suppose that
\begin{align}
    \sum_{m \geq 1} \E|Y_{m-1} - Y|^2 \big/\P(\tau \geq m) < \infty,
    \label{condition, resultDebias}
\end{align}
then
$Z \delequal \sum_{m = 0}^\tau (Y_m - Y_{m - 1})\big/\P(\tau \geq m)$
(with the convention $Y_{-1} = 0$) satisfies
$$
\E Z = \E Y,\qquad
    \E Z^2  = \sum_{m \geq 0}\Bar{v}_m \big/\P(\tau \geq m)
$$
where $\Bar{v}_m = \E|Y_{m- 1} - Y|^2 - \E|Y_m - Y|^2$.

\end{result}{}

\section{Algorithm}
\label{sec: algorithm}

Throughout the rest of this paper, let $\notationdef{Levy-process-X}{X(t)}$ be a L\'evy process with generating triplet $(c_X,\sigma,\nu)$ satisfying the following heavy-tailed assumption.

\begin{assumption}
\label{assumption: heavy tailed process}
$\E X(1) = 0$. $X(t)$ is of infinite activity.
The Blumenthal-Getoor index $\notationdef{def-beta-blumenthal-getoor-index}{\beta} \delequal \inf\{p > 0: \int_{(-1,1)}|x|^p\nu(dx) < \infty\}$ satisfies $\beta < 2$.
Besides, one of the two claims below holds for the L\'evy measure $\nu$.
\begin{itemize}
    \item (One-sided cases) $\nu$ is supported on $(0,\infty)$, and function $H_+(x) = \nu[x,\infty)$ is regularly varying as $x \rightarrow \infty$ with index $-\alpha$ where $\alpha > 1$;
    \item (Two-sided cases) There exist $\notationdef{notation-heavy-tailed-index-alpha}{\alpha,\alpha^\prime} > 1$ such that  $H_+(x) = \nu[x,\infty)$ is regularly varying as $x \rightarrow \infty$ with index $-\alpha$ and $H_-(x) = \nu(-\infty,-x]$ is regularly varying as $x \rightarrow \infty$ with index $-\alpha^\prime$.
\end{itemize}
\end{assumption}

The other assumption on $X(t)$ revolves around the continuity of the law of $\notationdef{process-X-<z}{X^{<z}}$, which is the L\'evy process with generating triplet $(c_X,\sigma,\nu|_{(-\infty,z)})$.
That is, $X^{<z}$ is a modulated version of $X$ where all the upward jumps with size larger than $z$ are removed.

\begin{assumption}
\label{assumption: holder continuity strengthened on X < z t}
    There exist $\notationdef{constant-z0-A2}{z_0},\notationdef{constant-C-A2}{C},\notationdef{constant-lambda-A2}{\lambda} > 0$ such that
    $$
    \P\big(X^{<z}(t) \in [x, x + \delta]\big) \leq \frac{C\delta}{t^\lambda\wedge 1}\qquad \forall z \geq z_0,\ t > 0,\ x \in \R,\ \delta > 0.
    $$
\end{assumption}
\noindent
Assumption~\ref{assumption: holder continuity strengthened on X < z t} can be interpreted as a uniform version of Lipschitz continuity in the law of $X^{<z}(t)$.
In Section~\ref{sec: lipschitz cont, A2},
we show that Assumption \ref{assumption: holder continuity strengthened on X < z t} is a mild condition for L\'evy process with infinite activities and is easy to verify.

Next, we describe a class of target events $(A_n)_{n \geq 1}$ for which we propose a strongly efficient rare event simulation algorithm.
Let $\notationdef{scaled-process-bar-X-n}{\bar X_n(t)} = \frac{1}{n}X(nt)$ and $\bar X_n = \{\bar X_n(t):\ t \in [0,1]\}$ be the scaled version of the process.
Define events
\begin{align}
    \notationdef{def-target-set-A}{A} \delequal \{\xi \in \mathbb{D}: \sup_{t \in [0,1]}\xi(t)\geq \notationdef{def-parameter-a-in-set-A}{a}; \sup_{t \in (0,1] }\xi(t) - \xi(t-) < \notationdef{def-parameter-b-in-set-A}{b} \},\qquad 
    \notationdef{def-rare-event-A-n}{A_n} \delequal \{\bar X_n \in A\}. \label{def set A}
\end{align}
In words, $\xi \in A$ means that the path $\xi$ crossed barrier $a$ even though no upward jumps in $\xi$ is larger than $b$.
For technical reasons, we also impose the following mild condition on the values of the constants $a,b > 0$.

\begin{assumption}
\label{assumption: choice of a,b}
$a, b > 0$ and $a/b \notin \mathbb Z.$
\end{assumption}
\noindent


In this section, we present a strongly efficient rare-event simulation algorithm for $(A_n)_{n \geq 1}$.
Specifically, Section~\ref{subsec: importance sampling dist, algo} presents the design of the importance sampling distribution $\Q_n$,
Section~\ref{subsec: estimators Z_n} discusses how we apply the randomized Monte-Carlo debiasing technique in Result~\ref{resultDebias} in our algorithm,
Section~\ref{subsec: algo, SBA, and debiasing technique}
discusses how we combine the debiasing technique with SBA in Result~\ref{result: concave majorant of Levy},
and 
Section~\ref{subsec: sampling from IS distribution Qn} explains how to sample from the importance sampling distribution $\Q_n$.
Combining all these components in Section~\ref{subsec: strong efficiency and complexity}, we propose Algorithm~\ref{algoISnoARA} for rare-event simulation of $\P(A_n)$ and establish its strong efficiency in Theorem~\ref{theorem: strong efficiency without ARA}.
Section~\ref{subsec: algo, ARA, SBA, and debiasing technique} addresses the case where the exact simulation of $X^{<z}(t)$ is not available.


\subsection{Importance Sampling Distributions $\Q_n$}
\label{subsec: importance sampling dist, algo}

At the core of our algorithm is a principled design of importance sampling strategies based on heavy-tailed large deviations.
This can be seen as an extension of the framework proposed in \cite{chen2019efficient}.
First, note that
\begin{align}
    \notationdef{notation-l*-jump-number}{l^*} \delequal \ceil{a/b}
    \label{def l *}
\end{align}
indicates the number of jumps required to cross the barrier $a$ starting from the origin if no jump is allowed to be larger than $b$.
Based on the sample-path large deviations reviewed in Section~\ref{subsec: review, LD of levy},
we expect the events $A_n = \{\bar X_n \in A\}$ to be almost always caused by exactly $l^*$ large upward jumps in $\bar X_n$.
These insights reveal critical information regarding the conditional law $\P(\ \cdot\ |\bar X_n \in A)$.
More importantly, they lead to a natural yet effective choice of importance sampling distributions to focus on the $l^*$-large-jump paths and provides sufficient approximations to $\P(\ \cdot\ |\bar X_n \in A)$.
Specifically, for any $\notationdef{algo-param-gamma}{\gamma} \in (0,b)$, define events $\notationdef{notaiton-set-B-gamma-n}{B^\gamma_n} \delequal \{\bar{X}_n \in B^\gamma\}$ with
\begin{align}
    \notationdef{notation-set-B-gamma}{B^\gamma} \delequal \big\{ \xi \in \mathbb{D}: \#\{ t \in [0,1]: \xi(t) - \xi(t-) \geq \gamma \} \geq l^* \big\},
    \label{def set B gamma}
\end{align}
where, for any $\xi \in \D$, we define $\xi(t-) = \lim_{s \uparrow t}\xi(s)$ as the left-limit of $\xi$ at time $t$.
Intuitively speaking, the parameter $\gamma \in (0,b)$ acts as a threshold of ``large jumps'':
any path $\xi \in B^\gamma$ has at least $l^*$ upward jumps that are considered large relative to the threshold level $\gamma$.
To prevent the likelihood ratio from blowing up to infinity,
we then consider an importance sampling distribution with defensive mixtures (see \cite{hesterberg1995}) and define (for some $\notationdef{algo-param-w}{w}\in(0,1)$)
\begin{align}
    \notationdef{notation-ISDM-measure-Qn}{\mathbf{Q}_n(\cdot)} \delequal w\P(\cdot) + (1 - w) \P(\ \cdot\ | B^\gamma_n).
    \label{def: IS distribution Qn}
\end{align}
Sampling from $\P(\ \cdot\ | B^\gamma_n)$, and hence $\Q_n(\cdot)$, is straightforward and will be 
addressed in Section~\ref{subsec: sampling from IS distribution Qn}.

With the design of the importance sampling distribution $\Q_n$ in hand,
one would naturally consider an estimator for $\P(A_n)$ of form $\mathbf{I}_{A_n} \cdot \frac{d\P}{d\Q_n}$. This is due to
$$
\E^{\Q_n}\bigg[\mathbf{I}_{A_n}\frac{d\P}{d\Q_n}\bigg] = \E[\mathbf{I}_{A_n}] = \P(A_n).
$$
Here, we use $\E^{\Q_n}$ to denote the expectation operator under law $\Q_n$ and $\E$ for the expectation under $\P$.
Nevertheless, the exact evaluation or simulation of 
$
\mathbf{I}_{A_n} = \mathbf{I}\{\bar X_n \in A\}
$
is generally not computationally feasible 
due to the infinite activities of the process $X$, making it computationally infeasible to simulate or store the entire sample path with finite computational resources.
This marks a significant difference from the tasks in \cite{chen2019efficient},
which focus on random walks or compound Poisson processes with constant drifts that can be simulated exactly.
To overcome this challenge, we instead consider estimators $L_n$ in the form of
\begin{align}
    \notationdef{notation-IS-esitmator-Ln}{L_n} = Z_n \frac{d\P}{d\Q_n} 
    =
    \frac{Z_n}{ w + \frac{1 - w}{ \P(B^\gamma_n) }\mathbf{I}_{ B^\gamma_n } }
    \label{def: estimator Ln}
\end{align}
where $Z_n$ can be simulated within finite computational resources and allows $L_n$ to recover the right expectation under the importance sampling distribution $\Q_n$, i.e., $\E^{\Q_n}[L_n] = \P(A_n)$.
In Section~\ref{subsec: estimators Z_n}, we elaborate on the design of the estimators $Z_n$.

\subsection{Estimators $Z_n$}
\label{subsec: estimators Z_n}

Intuitively speaking, the goal is to construct $Z_n$'s that can be plugged into \eqref{def: estimator Ln} as unbiased estimators of $\mathbf{I}_{A_n}$.
To this end, we consider the following decomposition of the L\'evy process $X$.
For any $\xi \in \D$ and $t \geq 0$, let $\Delta\xi(t) = \xi(t) - \xi(t-)$ be the size of the discontinuity in $\xi$ at time $t$. 
Recall that $\gamma \in (0,b)$ is the threshold of large jumps in the definition of $B^\gamma$ in \eqref{def set B gamma}.
Let
\begin{align}
        \notationdef{notation-process-Jn}{J_n(t)} & \delequal \sum_{s \in [0,t]}\Delta X(s) \mathbf{I}\big( \Delta X(s) \geq n\gamma \big),
        \label{def: process J n Xi n}
        \\
    \notationdef{notation-process-Xi-n}{\Xi_n(t)} & \delequal X(t) - J_n(t) = X(t) - \sum_{s \in [0,t]}\Delta X(s) \mathbf{I}\big( \Delta X(s) \geq n\gamma \big).
    \nonumber
\end{align}
We highlight several important facts regarding the decomposition $X(t) = J_n(t) + \Xi_n(t)$.
\begin{itemize}
    \item By the definition of $\Q_n$, the law of $\Xi_n$ remains unchanged under both $\Q_n$ and $\P$, which is identical to the law of $X^{<n\gamma}$, namely, a L\'evy process with generating triplet $(c_X,\sigma,\nu|_{ (-\infty,n\gamma) })$.
    
    \item Under $\P$, the process $J_n$ admits the law of a L\'evy process with generating triplet $(0,0,\nu|_{[n\gamma,\infty)})$, which is a compound Poisson process.
    
    \item Under $\Q_n$, the path $\{J_n(t):\ t \in [0,n]\}$ follows the same law as a L\'evy process with generating triplet $(0,0,\nu|_{[n\gamma,\infty)})$, conditioned on having at least $l^*$ jumps over $[0,n]$.
    
    \item Under both $\P$ and $\Q_n$, the two processes $J_n$ and $\Xi_n$ are independent.
\end{itemize}

Let $\notationdef{notation-process-bar-Jn}{\bar J_n(t)} = \frac{1}{n}J_n(nt)$, $\bar J_n = \{\bar J_n(t):\ t \in [0,1]\}$,
$\notationdef{notation-process-bar-Xi-n}{\bar \Xi_n(t)} = \frac{1}{n}\Xi_n(nt)$, and $\bar \Xi_n = \{\bar \Xi_n(t):\ t \in [0,1]\}$.
We now discuss how the decomposition
$$
\bar X_n = \bar J_n + \bar \Xi_n
$$
can help us construct unbiased estimators of $\mathbf{I}_{A_n}$.
First, recall that $\gamma \in (0,b)$.
As a result, in the definition of events $A_n = \{\bar X_n \in A\}$ in \eqref{def set A},
the condition $\sup_{t \in (0,1]}\xi(t) - \xi(t-) < b$ only concerns the large jump process $\bar J_n$ since any upward jump in $\bar \Xi_n$ is bounded by $\gamma < b$.
Therefore, with
\begin{align}
    \notationdef{notation-set-E-in-estimator}{E} \delequal \{ \xi \in \mathbb{D}:\ \sup_{t \in (0,1]}\xi(t) - \xi(t-) < b  \},
    \qquad 
    \notationdef{notation-event-bar-En}{E_n} \delequal \{\bar J_n \in E\}
\end{align}
and 
$$
\notationdef{notaiton-running-supremum-Mt}{M(t)} \delequal \sup_{s \leq t}X(s),
    \qquad 
    \notationdef{notation-Y-*-n}{Y^*_n} \delequal \mathbf{I}\big( M(n) \geq na \big),
$$
we have
$$
\mathbf{I}_{A_n} = Y^*_n \mathbf{I}_{E_n}.
$$
As discussed above, the exact evaluation of $Y^*_n$ is generally not computationally possible. 
Instead, suppose that we have access to a sequence of random variables $(\hat{Y}^m_n)_{m \geq 0}$ that only take values in $\{0,1\}$ and provide progressively more accurate approximations to $Y^*_n$ as $m \rightarrow \infty$.
Then in light of the debiasing technique in Result~\ref{resultDebias}, one can consider (under the convention that $\hat Y^{-1}_n \equiv 0$)
\begin{align}
    \notationdef{notation-estimator-Zn}{Z_n} = \sum_{m = 0}^{ \tau}\frac{ \hat Y^m_n - \hat Y^{m-1}_n }{ \P(\tau \geq m) }\mathbf{I}_{E_n}
    \label{def: estimator Zn}
\end{align}
where $\notationdef{notation-truncation-time-tau}{\tau}$ is $\text{Geom}(\rho)$ for some $\notationdef{algo-param-rho}{\rho} \in (0,1)$ and is independent of everything else.
That is, $\P(\tau \geq m) = \rho^{m-1}$ for all $m \geq 1$.
Indeed, this construction of $Z_n$ is justified by the following proposition.
We defer the proof to Section~\ref{subsec: proof, prop strong efficiency}.

\begin{proposition}
\label{proposition: design of Zn}
\linksinthm{proposition: design of Zn}
Let $C_0 > 0$, $\rho_0 \in (0,1)$, $\mu > 2l^*(\alpha - 1)$, and $\bar m \in \mathbb N$. Suppose that
\begin{align}
    \P\Big(Y^*_n \neq \hat Y^m_n\ \Big|\ \mathcal{D}(\bar J_n) = k\Big) \leq C_0 \rho^m_0 \cdot (k+1)\qquad \forall k \geq 0,n\geq 1, m \geq \bar m
    \label{condition 1, proposition: design of Zn}
\end{align}
where $\mathcal{D}(\xi)$ counts the number of discontinuities in $\xi$ for any $\xi \in \mathbb D$.
Besides, suppose that for all $\Delta \in (0,1)$,
\begin{align}
    \P\Big(Y^*_n \neq \hat Y^m_n,\ \bar X_n \notin A^\Delta\ \Big|\ \mathcal{D}(\bar J_n) = k\Big)
    \leq \frac{C_0 \rho^m_0}{ \Delta^2 n^{\mu} }\qquad \forall n\geq 1, m \geq 0, k = 0,1,\cdots,l^* - 1,
    \label{condition 2, proposition: design of Zn}
\end{align}
where 
$\notationdef{notation-set-A-Delta}{A^\Delta} = \big\{\xi \in \mathbb{D}: \sup_{t \in [0,1]}\xi(t)\geq a - \Delta \big\}$.
Then given $\rho \in (\rho_0,1)$,
there exists some $\bar \gamma = \bar \gamma(\rho) \in (0,b)$
such that
for all $\gamma \in (0,\bar \gamma)$,
the estimators $(L_n)_{n \geq 1}$ specified in \eqref{def: estimator Ln} and \eqref{def: estimator Zn} are \textbf{unbiased and strongly efficient} for $\P(A_n) = \P(\bar X_n \in A)$ under the importance sampling distribution $\Q_n$ in \eqref{def: IS distribution Qn}.
\end{proposition}

\subsection{Construction of $\hat Y^m_n$}
\label{subsec: algo, SBA, and debiasing technique}

In light of Proposition \ref{proposition: design of Zn},
our next goal is to design $\hat Y^m_n$'s
that provide sufficient approximations to 
$
Y^*_n = \mathbf{I}( M(n) \geq na )
$
and
satisfy the conditions \eqref{condition 1, proposition: design of Zn} and \eqref{condition 2, proposition: design of Zn}.

Recall the decomposition of $X(t) = \Xi_n(t) + J_n(t)$ in \eqref{def: process J n Xi n}.
Under both $\Q_n$ and $\P$,
the processes $\Xi_n$ and $J_n$ are independent, and $\Xi_n$ admits the law of $X^{<n\gamma}$, i.e., a L\'evy process with generating triplet $(c_X,\sigma,\nu|_{ (-\infty,n\gamma) })$.
This section discusses how, after sampling $J_n$ from $\Q_n$, we approximate the supremum of $\Xi_n$.
Specifically, on event $\{\mathcal D(\bar J_n) = k\}$, i.e., the process $J_n$ makes $k$ jumps over $[0,n]$, $J_n$ admits the form of $\zeta_k$ with
\begin{align}
    \zeta_k(t) = \sum_{i = 1}^k z_i \mathbf{I}_{[u_i,n]}(t)\qquad \forall t \in [0,n]
    \label{representation, process J n with k jumps}
\end{align}
for some $z_i \in [n\gamma,\infty)$ and $u_1 < u_2 < \cdots < u_k$.
This allows us to partition $[0,n]$ into $k+1$ disjoint intervals $[0,u_1),\ [u_1,u_2),\ldots,\ [u_{k-1},u_k),\ [u_k,1]$.
We adopt the convention $u_0 \equiv 0, u_{k+1} \equiv 1$ and set 
\begin{align}
    \notationdef{notation-partition-I-i}{I_i} = [u_{i-1},u_i)\quad \forall i \in [k],\qquad I_{k+1} = [u_k,1]. \label{def, partition of time line, I_i}
\end{align}
For $\zeta_k = \sum_{i = 1}^k z_i \mathbf{I}_{[u_i,n]}$,
define
\begin{align}
    \notationdef{notation-M-n-i}{M_{n}^{(i),*}(\zeta_k)} \delequal \sup_{ t \in I_i}\Xi_n(t) - \Xi_n(u_{i-1})
    \label{def: M n i, supremum on piece i}
\end{align}
as the supremum of the fluctuations of $\Xi_n(t)$ over $I_i$.
Define random function
\begin{align}
    \notationdef{notation-Y-*-n-zeta}{Y^*_n(\zeta_k)}
    \delequal
    \max_{i \in [k+1]}\mathbf{I}\Big( \Xi_n(u_{i-1}) + \zeta_k(u_{i-1}) + M^{(i),*}_n(\zeta_k) \geq na \Big),
    \label{def, random function Y * n}
\end{align}
and note that $Y^*_n(J_n) = \mathbf{I}( \sup_{t \in [0,n]}X(t) \geq na )$.

In theory, the representation \eqref{def, random function Y * n} provides an algorithm for the simulation of $\mathbf{I}(\sup_{t \in [0,n]}X(t) \geq na)$.
Nevertheless, the exact simulation of the supremum $M^{(i),*}_n(\zeta_k)$ is generally not available. 
Instead, we apply SBA introduced in Section~\ref{subsec: SBA, preliminary} to approximate $M^{(i),*}_n(\zeta_k)$, thus providing the construction of $\hat Y^m_n$.
Specifically, define
\begin{align}
    l^{(i)}_1 & = V^{(i)}_1\cdot (u_{i} - u_{i-1}); \label{defStickLength1} \\ 
    \notationdef{notation-stick-length-l-i-j}{l^{(i)}_j} & = V^{(i)}_j\cdot (u_{i} - u_{i-1} - l^{(i)}_1 - l^{(i)}_2 -\cdots - l^{(i)}_{j-1})\qquad \forall j \geq 2 \label{defStickLength2}
\end{align}
where each $V^{(i)}_j$ is an iid copy of Unif$(0,1)$.
That is, for each $i \in [k+1]$, the sequence $(l^{(i)}_j)_{j \geq 1}$ is defined under the recursion in \eqref{prelim: stick breaking, generate sticks l i}, with $T = u_{i} - u_{i-1}$ set as the length of $I_i$.
Then, conditioning on the values of $l^{(i)}_j$'s, we sample
\begin{align}
    \xi^{(i)}_j \sim \P\Big(\Xi_n(l^{(i)}_j) \in \ \cdot\ \Big),
    \label{def xi i j, algo without ARA}
\end{align}
i.e.,
$\xi^{(i)}_j$ is an independent copy of $\Xi_n(l^{(i)}_j)$, with all $\xi^{(i)}_j$ being independently generated.
Result~\ref{result: concave majorant of Levy} then implies
$
\big(\Xi_n(u_i) - \Xi_n(u_{i-1}),\ M^{(i),*}_n(\zeta_k)\big) \distequal \big(\sum_{j \geq 1}\xi^{(i)}_j,\ \sum_{j \geq 1}(\xi^{(i)}_j)^+\big)
$
for each $i \in [k+1]$.
Furthermore, by summing up only finitely many $\xi^{(i)}_j$'s, we define
\begin{align}
    \notationdef{notation-hat-M-n-i-m}{\hat M^{(i),m}_n(\zeta_k)} = \sum_{j = 1}^{m + \ceil{\log_2(n^d)}  }( \xi^{(i)}_j)^+
    \label{def hat M i m without ARA}
\end{align}
as an approximation to $M^{(i),*}_n(\zeta_k)$ defined in \eqref{def: M n i, supremum on piece i}. 
Here, $d > 0$ is another parameter of the algorithm.
For technical reasons, we add an extra $\ceil{ \log_2(n^d) }$ term in the summation in \eqref{def hat M i m without ARA},
which helps ensure that the algorithm achieves strong efficiency as $n \to \infty$ while only introducing a minor increase in the computational complexity.

Now, we are ready to present the design of the approximators $\hat Y^m_n$.
For $\zeta_k = \sum_{i = 1}^k z_i \mathbf{I}_{[u_i,n]}$,
define the random function
\begin{align}
    \notationdef{notation-approximation-hat-Y-m-n}{\hat Y^m_n(\zeta_k)} = \max_{i \in [k+1]}\mathbf{I}\bigg( \sum_{q = 1}^{i-1}\sum_{j \geq 0}\xi^{(q)}_j + \sum_{q = 1}^{i-1}z_q + \hat M^{(i)}_n(\zeta_k) \geq na \bigg).
    \label{def hat Y m n without ARA}
\end{align}
Here, note that $\sum_{q = 1}^{i-1}z_q = \zeta_k(u_{i-1})$.
As a high-level description, the algorithm proceeds as follows.
After sampling $J_n$ from the importance sampling distribution $\Q_n$ defined in \eqref{def: IS distribution Qn},
we plug $\hat Y^m_n(J_n)$ into $Z_n$ defined in \eqref{def: estimator Zn},
which in turn allows us to simulate ${L_n} = Z_n \frac{d\P}{d\Q_n}$ as the importance sampling estimator under $\Q_n$.
\begin{remark}
At first glance, one may get the impression that the simulation of $\hat Y^m_n$ involves the summation of infinitely many elements in 
$\sum_{q = 1}^{i-1}\sum_{j \geq 0}\xi^{(q),m}_j$.
Fortunately, the truncation index $\tau$ in $Z_n$ (see \eqref{def: estimator Zn}) is almost surely finite.
Therefore, when running the algorithm in practice, after $\tau$ is decided, there is no need to simulate any $\hat Y^m_n$ beyond $m \leq \tau$.
Given the construction of $\hat M^{(i),m}_n(\zeta_k)$ in \eqref{def hat M i m without ARA},
the simulation of $\hat Y^m_n(\zeta_k)$ only requires (for each $i \in [k+1]$)
$
\xi^{(i)}_1,\xi^{(i)}_2,\ldots,\xi^{(i)}_{ \tau + \ceil{ \log_2(n^d) } },
$
as well as the sum $\sum_{j \geq \tau + \ceil{ \log_2(n^d) } + 1 }\xi^{(i)}_j$.
Furthermore, conditioning on the value of $u_i - u_{i-1} - \sum_{j = 1}^{\tau + \ceil{ \log_2(n^d) }}\xi^{(i)}_j = t$,
the sum $\sum_{j \geq \tau + \ceil{ \log_2(n^d) } + 1 }\xi^{(i)}_j$ admits the law of $\Xi_n(t)$ (see Result~\ref{result: concave majorant of Levy}).
This allows us to simulate $\sum_{j \geq \tau + \ceil{ \log_2(n^d) } + 1 }\xi^{(i)}_j$ in one shot.
\end{remark}

Note that to implement the importance sampling algorithm and ensure strong efficiency, the following tasks still remain to be addressed.
\begin{enumerate}[$(i)$]
    \item 
        As mentioned above, the evaluation of $\hat Y^m_n(J_n)$ requires the ability to first sample $J_n$ from the importance sampling distribution $\Q_n$ defined in \eqref{def: IS distribution Qn}.
        We address this in Algorithm~\ref{algoSampleJnFromQ} proposed in Section~\ref{subsec: sampling from IS distribution Qn}.
        In summary, the simulation algorithm of estimators $L_n$ is detailed in Algorithm~\ref{algoISnoARA}. 

    \item 
        The strong efficiency of the proposed algorithm needs to be justified by verifying the conditions in Proposition~\ref{proposition: design of Zn}.
        This will be done in Section~\ref{subsec: strong efficiency and complexity} by establishing Theorem~\ref{theorem: strong efficiency without ARA}.

    \item 
        Simulating $\xi^{(i)}_j$'s requires the exact simulation of $X^{<n\gamma}(t)$, which may not be computationally feasible in certain cases.
        To address this challenge, Section~\ref{subsec: algo, ARA, SBA, and debiasing technique} proposes Algorithm~\ref{algoIS}, which builds upon Algorithm~\ref{algoISnoARA} and incorporates another layer of approximation via ARA.
\end{enumerate}

\subsection{Sampling from $\P(\ \cdot\ | B^\gamma_n)$}
\label{subsec: sampling from IS distribution Qn}

In this section, we revisit the task of sampling from $\P(\ \cdot\ |B^\gamma_n)$,
which is at the core of the implementation of the importance sampling distribution $\Q_n$ in \eqref{def: IS distribution Qn}.

Recall that under $\P$, the process $J_n$ is a compound Poisson process with generating triplet $(0,0,\nu|_{[n\gamma,\infty)})$.
More precisely, let $\widetilde N_n(\cdot)$ be a Poisson process with rate $\nu[n\gamma,\infty)$,
and we use $(S_i)_{i \geq 1}$ to denote the arrival times of jumps in $\widetilde N_n(\cdot)$.
Let 
 $(W_i)_{i \geq 1}$ be a sequence of iid random variables from 
    $$\nu^{\text{normalized}}_n(\cdot) = \frac{ \nu_n(\cdot) }{\nu[n\gamma,\infty)},\qquad \nu_n(\cdot) = \nu\big( \cdot \cap [n\gamma,\infty) \big)$$
and let $W_i$'s be independent of $\widetilde N_n(\cdot)$.
Under $\P$, we have
$$
J_n(t) \distequal \sum_{i = 1}^{ \widetilde N_n(t) }W_i
    =
    \sum_{i \geq 1}W_i\mathbf{I}_{ [S_i,\infty) }(t)\qquad \forall t \geq 0.
$$
Furthermore, for each $k \geq 0$, conditioning on $\{ \widetilde{N}_n(n) = k\}$,
the law of $S_1,\ldots,S_k$ is equivalent to that of the order statistics of $k$ iid samples from Unif$(0,n)$,
and $W_i$'s are still independent of $S_i$'s with the law unaltered. 
Therefore, the sampling of $J_n$ from $\P(\ \cdot\ |B^\gamma_n)$ can proceed as follows.
We first generate $k$ from the distribution of Poisson$( n \nu[n\gamma,\infty) )$, conditioning on $k \geq l^*$.
Then, independently, we generate $S_1,\cdots,S_k$ as the order statistics of $k$ iid samples from Unif$(0,n)$,
and $W_1,\cdots,W_k$ as iid samples of law $\nu^{\text{normalized}}_n(\cdot)$.
It is worth mentioning that the sampling of $W_i$ can be addressed with the help of the inverse measure.
Specifically, define $Q^{\leftarrow}_n (y)\delequal{} \inf\{s > 0: \nu_n[s,\infty) < y\}$
as the inverse of $\nu_n$, and observe that
    $$y \leq \nu_n[s,\infty)\qquad \Longleftrightarrow \qquad Q^{\leftarrow}_n (y) \geq s.$$
More importantly, for $U \sim \text{Unif}(0,\nu_n[n\gamma,\infty))$,
the law of $Q^\leftarrow_n(U)$ is $\nu^\text{normalized}_n(\cdot)$.
This leads to the steps detailed in Algorithm~\ref{algoSampleJnFromQ}.

\begin{algorithm}[H]
  \caption{Simulation of $J_n$ from $\P (\ \cdot\ | B^\gamma_n)$}\label{algoSampleJnFromQ}
  \begin{algorithmic}[1]
    \Require $n \in \mathbb{N}, l^* \in \mathbb{N}, \gamma > 0$, the Lévy measure $\nu$.

\State Sample $k$ from a Poisson distribution with rate $n\nu[n\gamma,\infty)  $ conditioning on $k \geq l^*$
\State Simulate $\Gamma_1,\cdots,\Gamma_k \stackrel{\text{iid}}{\sim} Unif\big(0,\nu_n[n\gamma,\infty)\big)$
\State Simulate $U_1,\cdots,U_k \stackrel{\text{iid}}{\sim} Unif(0,n)$
\State \textbf{Return } $J_n = \sum_{i = 1}^k Q^{\leftarrow}_n(\Gamma_i)\mathbf{I}_{[U_i,n]}$
    
  \end{algorithmic}
\end{algorithm}

\subsection{Strong Efficiency and Computational Complexity}
\label{subsec: strong efficiency and complexity}

With all the discussions above, we propose Algorithm~\ref{algoISnoARA} for rare-event simulation of $\P(A_n)$.
Specifically, here is a list of the parameters of the algorithm.
\begin{itemize}
    \item $\gamma \in (0,b)$: the threshold in $B^\gamma$ defined in \eqref{def set B gamma},
    \item $w \in (0,1)$: the weight of the defensive mixture in $\Q_n$; see \eqref{def: IS distribution Qn},
    \item $\rho \in (0,1)$: the geometric rate of decay for $\P(\tau \geq m)$ in \eqref{def: estimator Zn},
    \item $d > 0$: determining the $\log_2(n^d)$ term in \eqref{def hat M i m without ARA}.
\end{itemize}
The choice of $w \in (0,1)$ won't affect the strong efficiency of the algorithm.
Meanwhile, 
under proper parametrization,
Algorithm~\ref{algoISnoARA} meets conditions \eqref{condition 1, proposition: design of Zn} and \eqref{condition 2, proposition: design of Zn} stated in Proposition~\ref{proposition: design of Zn} and attains strong efficiency.
This is verified in Theorem~\ref{theorem: strong efficiency without ARA}.

\begin{theorem}
    \label{theorem: strong efficiency without ARA}
    \linksinthm{theorem: strong efficiency without ARA}
Let $d > \max\{2,\ 2l^*(\alpha - 1)\}$ and $w \in (0,1)$.
There exists $\rho_0\in (0,1)$ such that the following claim holds:
Given $\rho \in (\rho_0,1)$, there exists $\bar \gamma \in (0,b)$ such that
Algorithm~\ref{algoISnoARA} is \textbf{unbiased and strongly efficient} under any $\gamma \in (0,\bar\gamma)$.
\end{theorem}

We defer the proof to Section~\ref{subsec: proof of propositions cond 1 and 2}.
In fact, in Section~\ref{subsec: algo, ARA, SBA, and debiasing technique} we propose Algorithm~\ref{algoIS}, which can be seen as an extended version of Algorithm~\ref{algoISnoARA}
with another layer of approximation.
The strong efficiency of Algorithm~\ref{algoISnoARA} follows directly from that of Algorithm~\ref{algoIS}
(i.e., by setting $\kappa = 0$ in the proof of Theorem~\ref{theorem: strong efficiency}).
The choices of $\bar \gamma$ and $\bar \rho$ that ensure strong efficiency are also specified at the end of Section~\ref{subsec: algo, ARA, SBA, and debiasing technique}.

\begin{algorithm}[H]
  \caption{Strongly Efficient Estimation of $\P (A_n)$}\label{algoISnoARA}
  \begin{algorithmic}[1]
    \Require $w \in (0,1),\ \gamma > 0,\ d > 0,\  \rho \in (0,1)$ as the parameters of the algorithm; $a,b>0$ as the characterization of the set $A$; $(c_X,\sigma,\nu)$ as the generating triplet of $X(t)$.

    \Statex
    \State Set $t_n = \ceil{\log_2(n^d)}$

    \Statex

    \If{$\text{Unif}(0,1) < w$} \Comment{Sample $J_n$ from $\Q_n$}
        \State Sample $J_n = \sum_{i = 1}^k z_i \mathbf{I}_{[u_i,n]}$ from $\P$
    \Else
        \State Sample $J_n = \sum_{i = 1}^k z_i \mathbf{I}_{[u_i,n]}$ from $\P(\ \cdot\ |B^\gamma_n)$ using Algorithm \ref{algoSampleJnFromQ}
    \EndIf
    \State Set $u_0 = 0, u_{k+1} = n$.
    
    \Statex 
    \State Sample $\tau \sim \text{Geom}(\rho)$ \Comment{Decide Truncation Index $\tau$}
    
   \Statex
    \For{$i = 1,2,\ldots,k+1$} \Comment{Stick-Breaking Procedure} 
        \For{$j = 1,2,\ldots,t_n + \tau$}
            \State Sample $V^{(i)}_j\sim \text{Unif}(0,1)$
            \State Set $l^{(i)}_j = V^{(i)}_j(u_{i} - u_{i+1} - l^{(i)}_1 - l^{(i)}_2 - \ldots - l^{(i)}_{j - 1})$
            \State Sample $\xi^{(i)}_j \sim \P\big(X^{<n\gamma}(l^{(i)}_j) \in \ \cdot\ \big)$
        \EndFor
        \State Set $l^{(i)}_{ t_n + \tau + 1 } = u_{i} - u_{i-1} - l^{(i)}_1 - l^{(i)}_2 - \ldots - l^{(i)}_{t_n + \tau}$
        \State Sample $\xi^{(i)}_{t_n + \tau + 1} \sim \P\big(X^{<n\gamma}(l^{(i)}_{t_n + \tau + 1}) \in \ \cdot\ \big)$
    \EndFor
    
    \Statex
    \For{$m = 1,\cdots,\tau$} \Comment{Evaluate $\hat Y^m_{n}$}
        \For{$i = 1,2,\ldots,k+1$}
            \State Set $\hat{M}^{(i),m}_n = \sum_{q = 1}^{i-1}\sum_{j =1}^{t_n + \tau + 1}\xi^{(q)}_{j} + \sum_{q =1}^{i-1}z_q + \sum_{j =1}^{ t_n + m}(\xi^{(i)}_{j})^+   $
        \EndFor 
        \State Set $\hat Y^m_{n} = \mathbf{I}\big\{ \max_{i = 1,\ldots,k+1} \hat{M}^{(i),m}_n   \geq na  \big\}$
    \EndFor

\State Set $Z_n = \hat Y^1_n + \sum_{m = 2}^\tau ( \hat Y^m_{n} - \hat Y^{m-1}_{n} )\big/ \rho^{m-1}$ \Comment{Return the Estimator $L_n$}
\If{ $ \max_{i =1,\cdots,k}z_i > b $ }
    \State \textbf{Return} $L_n = 0$.
\Else 
    \State Set $\lambda_n = n\nu[n\gamma,\infty),\ p_n = 1 - \sum_{l = 0}^{l^* - 1}e^{-\lambda_n}\frac{\lambda_n^l}{l!},\ I_n = \mathbf{I}\{ J_n \in B^\gamma_n \}$
    \State \textbf{Return} $L_n = Z_n/(w + \frac{1-w}{p_n}I_n)$
\EndIf 
    
  \end{algorithmic}
\end{algorithm}
\begin{remark}
To conclude, we add a remark regarding the computational complexity of Algorithm~\ref{algoISnoARA} under the goal of attaining a given level of relative error at a specified confidence level.
First, consider the case where the complexity of simulation of
$X^{<z}(t)$ scales linearly with $t$ (uniformly for all $z \in [z_0,\infty]$ for some constant $z_0$).
This is a standard since the number of jumps we expect to simulate over $[0,t]$ grows linearly with $t$.
Then, the complexity of the SBA steps at step 13 of Algorithm~\ref{algoISnoARA} also scales linearly with $n$, as
the stick lengths of $l^{(i)}_j$'s, in expectation, grow linearly with $n$ because we deal with the time horizon $[0,n]$ given the scale factor $n$.
Next, since the same law for the truncation index $\tau$ (see step 8 of Algorithm~\ref{algoISnoARA}) is applied for all $n$, the only other factor that is varying with $n$ is $t_n = \ceil{\log_2(n^d)}$ in the loop at step 10. 
The strong efficiency of the algorithm
then implies a computational complexity of order $O(n \cdot \log_2 n)$.
If we instead assume that the cost of simulating $X^{<z}(t)$ is also uniformly bounded for all $t$, then the overall complexity of Algorithm~\ref{algoISnoARA} is further reduced to $O(\log_2 n)$.

In comparison, the crude Monte Carlo method requires a number of samples that is inversely proportional to the target probability $\P(A_n) \approx O( 1/n^{l^*(\alpha - 1)} )$ (see Lemma~\ref{lemma: LD, events A n})
with $\alpha > 1$ being the heavy-tailed index in Assumption~\ref{assumption: heavy tailed process}
and $l^* \geq 1$ defined in \eqref{def l *}.
Hypothetically, assuming that the evaluation of $\mathbb{I}_{A_n}$ (which at least requires the simulation of $X(t)$ and $M(t)$) is computationally feasible at a cost
that scales linearly with $n$, we end up with a computational complexity of $O(n \cdot n^{l^*(\alpha - 1)})$ (compared to the $O(n \cdot \log_2 n)$ cost of our algorithm).
Similarly, if we assume that the cost of generating $\mathbb{I}_{A_n}$ is uniformly bounded for all $n$, then the complexity of the crude Monte-Carlo method is $O(n^{l^*(\alpha - 1)})$ (compared to the $O(\log_2 n)$ cost of our algorithm).
In summary, 
not only does the proposed importance sampling algorithm address Lévy processes with infinite activities that are not simulatable for crude Monte Carlo methods, but it also enjoys a significant improvement in terms of computational complexity, with the advantage becoming even more evident for multiple-jump events with large $l^*$.
\end{remark}

\subsection{Construction of $\hat Y^m_n$ with ARA}
\label{subsec: algo, ARA, SBA, and debiasing technique}

As stressed earlier, implementing Algorithm~\ref{subsec: strong efficiency and complexity} requires the ability to sample from $\P(X^{<n\gamma}(t) \in\ \cdot\ )$.
The goal of this section is to address the challenge
when the exact simulation of $X^{<n\gamma}(t)$ is not available.
The plan is to incorporate the Asmussen-Rosiński approximation (ARA) in \cite{asmussen2001approximations}
into the design of the approximation $\hat Y^m_n$ proposed in Section~\ref{subsec: algo, SBA, and debiasing technique}.

To be specific, let 
\begin{align}
    \notationdef{notation-kappa-n,m}{\kappa_{n,m}} \delequal \frac{\kappa^m}{n^r}\qquad \forall n \geq 1,\ m \geq 0
    \label{def: kappa n m}
\end{align}
where $\notationdef{algo-param-kappa}{\kappa} \in (0,1)$ and $\notationdef{algo-param-r}{r} > 0$ are two additional parameters of our algorithm.
As a convention, we set $\kappa_{n,-1} \equiv 1$.
Without loss of generality, we consider $n$ large enough such that $n\gamma > 1 = \kappa_{n,-1}$.
For the L\'evy process $\Xi_n = X^{<n\gamma}$ with the generating triplet $(c_X,\sigma,\nu|_{ (-\infty,n\gamma) })$,
consider the following decomposition (with $B(t)$ being a standard Brownian motion) 
\begin{align}
    \Xi_n(t) & = c_Xt + \sigma B(t) + 
    \underbrace{\sum_{s \leq t}\Delta X(s) \mathbf{I}\Big(\Delta X(s) \in (-\infty,-1]\cup [1,n\gamma)\Big)}_{ \delequal J_{n,-1}(t)}
    \label{decomp, Xi n}
    \\
    &\quad
    + \sum_{m \geq 0}\Bigg[
    \underbrace{\sum_{s \leq t}\Delta X(s) \mathbf{I}\Big(|\Delta X(s)| \in[\kappa_{n,m},\kappa_{n,m-1})\Big)
    -
    t \cdot \nu\Big( (-\kappa_{n,m-1},-\kappa_{n,m}] \cup [\kappa_{n,m},\kappa_{n,m-1})\Big)
    }_{ \delequal J_{n,m}(t) }
    \Bigg].
    \nonumber
\end{align}
Here, for any $m \geq 0$, $J_{n,m}$ is a martingale with $var[J_{n,m}(1)] = \bar \sigma^2(\kappa_{n,m-1}) - \bar \sigma^2(\kappa_{n,m})$
where
\begin{align}
    \notationdef{notation-bar-sigma-function}{\bar\sigma^2(c)} \delequal \int_{(-c,c)} x^2 \nu(dx)\qquad \forall c\in (0,1].
    \label{def bar sigma}
\end{align}
Generally speaking, the difficulty of implementing Algorithm~\ref{subsec: strong efficiency and complexity} lies in the exact simulation of the martingale $\sum_{m \geq 0}J_{n,m}$.
In particular, whenever we have $\nu((-\infty,0)\cup(0,\infty))= \infty$ for the L\'evy measure $\nu$,
the expected number of jumps in $\sum_{m \geq 0}J_{n,m}$ (and hence $X^{<n\gamma}$ and $X$) will be infinite over any time interval with positive length.
By applying ARA, our goal is to approximate the jump martingale $J_{n,m}$'s using Brownian motions, which yields a process that is amenable to exact simulation.
To do so, let $(W^m)_{m \geq 1}$ be a sequence of iid copies of standard Brownian motions, which are also independent of $B(t)$.
For each $m \geq 0$, define
\begin{align}
   \notationdef{process-breve-Xi-m-n}{\Breve{\Xi}^{m}_n(t)} & \delequal c_Xt + \sigma B(t) + \sum_{q = -1}^{m}J_{n,q}(t)
        +
        \sum_{q \geq m + 1}\sqrt{\bar \sigma^2(\kappa_{n,q-1}) - \bar \sigma^2(\kappa_{n,q})} \cdot W^{q}(t).
        \label{def breve Xi n m, ARA}
\end{align}
Here, the process $\Breve{\Xi}^{m}_n$ can be interpreted as an approximation to  $\Xi_n$,
where the jump martingale (with jump sizes under $\kappa_{n,m}$) is substituted by a standard Brownian motion with the same variance.
Note that for any $t > 0$, the random variable $\Breve{\Xi}^{m}_n(t)$ is exactly simulatable, as it is a convolution of a compound Poisson process with constant drift and a Gaussian random variable.

Based on the approximations $\breve \Xi^m_n$ in \eqref{def breve Xi n m, ARA},
we apply SBA and reconstruct $\hat M^{(i),m}_n$ (originally defined in \eqref{def hat M i m without ARA})
and
$\hat Y^m_n$ (originally defined in \eqref{def hat Y m n without ARA}) as follows.
Let $\zeta_k(t) = \sum_{i = 1}^k z_i \mathbf{I}_{[u_i,n]}(t)$ be a piece-wise step function with $k$ jumps over $(0,n]$, i.e., admitting the form in \eqref{representation, process J n with k jumps}.
Recall that the jump times in $\zeta_k$ leads to a partition of $[0,n]$ of $(I_i)_{i \in [k+1]}$ defined in \eqref{def, partition of time line, I_i}.
For any $I_i$, let the sequence $l^{(i)}_j$'s be defined as in \eqref{defStickLength1}--\eqref{defStickLength2}.
Next, conditioning on $(l^{(i)}_j)_{j \geq 1}$, one can sample $\notationdef{notation-increment-on-sticks-xi-i-j-m}{\xi^{(i),m}_j,\xi^{(i)}_j}$ as
\begin{align}
    \big(\xi^{(i)}_j,\xi^{(i),0}_j,\xi^{(i),1}_j,\xi^{(i),2}_j,\ldots) \distequal \Big(\Xi_n(l^{(i)}_j),\ \Breve{\Xi}^0_n(l^{(i)}_j),\ \Breve{\Xi}^1_n(l^{(i)}_j),\ \Breve{\Xi}^2_n(l^{(i)}_j),\ldots\Big).
    \label{def xi i m j, ARA plus SBA}
\end{align}
The coupling in \eqref{prelim: SBA, coupling} then implies
\begin{align}
    & \Big( 
         \Xi_n(u_i) - \Xi_n(u_{i-1}),\ 
         \sup_{t \in I_i}\Xi_n(t) - \Xi_n(u_{i-1}),\ 
         \breve \Xi^{0}_n(u_i) - \breve \Xi^{0}_n(u_{i-1}),\ 
         \sup_{t \in I_i}\breve \Xi^{0}_n(t) - \breve \Xi^{0}_n(u_{i-1}),
         \label{coupling, ARA plus SBA}
         \\
         &\qquad\qquad\qquad\qquad \qquad\qquad\qquad\qquad 
         \breve \Xi^{1}_n(u_i) - \breve \Xi^{1}_n(u_{i-1}),\ 
         \sup_{t \in I_i}\breve \Xi^{1}_n(t) - \breve \Xi^{1}_n(u_{i-1}),
    \ldots
    \Big)
    \nonumber
    \\
    &\qquad \distequal 
    \Big(
    \sum_{j \geq 1}\xi^{(i)}_j,\ \sum_{j \geq 1}(\xi^{(i)}_j)^+,\
    \sum_{j \geq 1}\xi^{(i),0}_j,\ \sum_{j \geq 1}(\xi^{(i),0}_j)^+,\
    \sum_{j \geq 1}\xi^{(i),1}_j,\ \sum_{j \geq 1}(\xi^{(i),1}_j)^+,\ldots  
    \Big).
    \nonumber
\end{align}
Now, we define
\begin{align}
    {\hat M^{(i),m}_n(\zeta_k)} = \sum_{j = 1}^{m + \ceil{\log_2(n^d)}  }( \xi^{(i),m}_j)^+
    \label{def hat M i m}
\end{align}
as an approximation to
$
{M_{n}^{(i),*}(\zeta_k)} = \sup_{ t \in I_i}\Xi_n(t) - \Xi_n(u_{i-1}) = \sum_{j \geq 1}(\xi^{(i)}_j)^+
$
using both ARA and SBA.
Compared to the original design in \eqref{def hat M i m without ARA},
the main difference in \eqref{def hat M i m} is that we substitute $\xi^{(i)}_j$ with $\xi^{(i),m}_j$, and the latter is exactly simulatable as, conditioning on the values of $l^{(i)}_j$'s, it admits the law of $\breve \Xi^m_n$.
Similarly, let
\begin{align}
    {\hat Y^m_n(\zeta_k)} = \max_{i \in [k+1]}\mathbf{I}\Big( \sum_{q = 1}^{i-1}\sum_{j \geq 0}\xi^{(q),m}_j + \sum_{q = 1}^{i-1}z_q + \hat M^{(i),m}_n(\zeta_k) \geq na \Big);
    \label{def hat Y m n}
\end{align}
Again, the main difference between \eqref{def hat Y m n} and \eqref{def hat Y m n without ARA} is that we incorporate ARA and substitute $\xi^{(i)}_j$'s with $\xi^{(i),m}_j$'s.

Plugging the design of $\hat Y^m_n(\zeta_k)$ in \eqref{def hat Y m n}
into the estimator $Z_n$ in \eqref{def: estimator Zn},
we propose Algorithm~\ref{algoIS} for rare-event simulation of $\P(A_n)$ when exact simulation of $X^{<n\gamma}$ is not available.
Below is a summary of the parameters of the algorithm.
\begin{itemize}
    \item $\gamma \in (0,b)$: the threshold in $B^\gamma$ defined in \eqref{def set B gamma},
    \item $w \in (0,1)$: the weight of the defensive mixture in $\Q_n$; see \eqref{def: IS distribution Qn},
    \item $\rho \in (0,1)$: the geometric rate of decay for $\P(\tau \geq m)$ in \eqref{def: estimator Zn},
    \item $\kappa \in [0,1),\ r > 0$: determining the truncation threshold $\kappa_{n,m}$ in \eqref{def: kappa n m},
    \item $d > 0$: determining the $\log_2(n^d)$ term in \eqref{def hat M i m}.
\end{itemize}
Theorem~\ref{theorem: strong efficiency} justifies that, under proper parametrization, Algorithm~\ref{algoIS} is unbiased and strongly efficient.

\begin{theorem}
\label{theorem: strong efficiency}
\linksinthm{theorem: strong efficiency}
Let $\mu > 2l^*(\alpha - 1)$ and $\beta_+ \in (\beta,2)$ where $\alpha > 1$ is the heavy-tail index and $\beta \in (0,2)$ is the Blumenthal-Getoor index in Assumption~\ref{assumption: heavy tailed process}.
Let $w \in (0,1)$ and
\begin{align}
    \kappa^{2 - \beta_+} < \frac{1}{2},
    \qquad 
    r(2 - \beta_+) > \max\{2, \mu - 1\},
    \qquad 
    d > \max\{2, 2\mu - 1\}.
    \label{choice of parameters, theorem: strong efficiency}
\end{align}
There exists $\rho_0 \in (0,1)$ such that the following claim holds:
Given $\rho \in (\rho_0,1)$, there exists $\bar \gamma \in (0,b)$ such that
Algorithm~\ref{algoIS} is \textbf{unbiased and strongly efficient} under any $\gamma \in (0,\bar\gamma)$.
\end{theorem}
In Section~\ref{subsec: proof of propositions cond 1 and 2} we provide the proof, the key arguments of which are the verification of conditions \eqref{condition 1, proposition: design of Zn} and \eqref{condition 2, proposition: design of Zn} in Proposition~\ref{proposition: design of Zn}.
Here, we specify the choices of the parameters. 
First, pick $\alpha_3 \in (0,\frac{1}{\lambda}),\ \alpha_4 \in (0, \frac{1}{2\lambda})$
where $\lambda > 0$ is the constant in Assumption~\ref{assumption: holder continuity strengthened on X < z t}.
Next, pick $\alpha_2 \in (0, \frac{\alpha_3}{2}\wedge 1)$ and $\alpha_1 \in (0,\frac{\alpha_2}{\lambda})$.
Also, fix $\delta \in (1/\sqrt{2},1)$.
This allows us to pick $\rho_0 \in (0,1)$ such that
$$
\rho_0 > \max\bigg\{
    \delta^{\alpha_1}, \  \frac{\kappa^{2 - \beta_+}}{\delta^2},\  \frac{1}{\sqrt{2}\delta},\ 
    \delta^{\alpha_2 - \lambda\alpha_1},\ \delta^{1 - \lambda\alpha_3}, \delta^{-\alpha_2 + \frac{\alpha_3}{2}}
    \bigg\}.
$$
After picking $\rho \in (\rho_0, 1)$,
one can find some $q > 1$ such that $\rho_0^{1/q} < \rho$.
Let $p > 1$ be such that $\frac{1}{p} + \frac{1}{q} = 1$.
Let $\Delta > 0$ be small enough such that $a - \Delta > (l^*-1)b$.
Then, we pick $\bar\gamma \in (0,b)$ small enough such that
$$
    \frac{a - \Delta - (l^* - 1)b}{\bar\gamma} + l^* - 1 > 2l^*p.
$$
Again, the details of the parameter choices can be found at the beginning of Section~\ref{sec: proof}.
It is also worth mentioning that, by setting $\kappa = 0$, Algorithm~\ref{algoIS} would reduce to Algorithm~\ref{algoISnoARA}, as $\xi^{(i),m}_j$'s in \eqref{coupling, ARA plus SBA} would reduce to $\xi^{(i)}_j$'s in \eqref{def xi i j, algo without ARA}; in other words, the ARA mechanism is effective only if the truncation threshold $\kappa_{n,m} = \kappa^m/n^r > 0$.
As a result, Theorem~\ref{theorem: strong efficiency without ARA} follows directly from Theorem~\ref{theorem: strong efficiency} by setting $\kappa = 0$.

\begin{remark}
While Algorithm~\ref{algoIS} terminates within finite steps almost surely,
its computational complexity may not be finite in expectation.
This is partially due to the implementation of ARA
as
we approximate the jump martingale $J_{n,m}(t)$ in \eqref{decomp, Xi n}
using a independent Brownian motion term in \eqref{def breve Xi n m, ARA}.
In theory, a potential remedy is to identify a better coupling between the jump martingales and Brownian motions; see, for instance, Theorem 9 of \cite{10.1214/18-EJS1456}.
This would allow us to pick a larger $\kappa$ for the truncation threshold $\kappa_{n,m}$ in ARA, under which the simulation algorithm generates significantly fewer jumps when sampling $\xi^{(i),m}_j$'s.
However, to the best of our knowledge, there is no practical implementation of the coupling in \cite{10.1214/18-EJS1456}.
We note that similar issues arise in works such as \cite{gonzalez2022simulation}, where the coupling in \cite{10.1214/18-EJS1456} imply a much tighter error bound in theory but cannot be implemented in practice.
\end{remark}

\section{Lipschitz Continuity of the Distribution of $X^{<z}(t)$}
\label{sec: lipschitz cont, A2}

This section investigates the sufficient conditions for Assumption \ref{assumption: holder continuity strengthened on X < z t}.
That is, there exist ${z_0},\ {C},\ {\lambda} > 0$ such that
    \begin{align}
        \P\big(X^{<z}(t) \in [x, x + \delta]\big) \leq \frac{C\delta}{t^\lambda\wedge 1}\qquad \forall z \geq z_0,\ t > 0,\ x \in \R,\ \delta > 0.
        \label{condition, sec: lipschitz cont, A2}
    \end{align}
Here, recall that $X^{>z}$ is the L\'evy process with generating triplet $(c_X,\sigma,\nu|_{(-\infty,z)})$.
In other words, this is a modulated version of $X$ where any the upward jump larger than $z$ is removed.

\begin{algorithm}[H]
  \caption{Strongly Efficient Estimation of $\P (A_n)$ with ARA}\label{algoIS}
  \begin{algorithmic}[1]
    \Require $w \in (0,1),\ \gamma > 0,\ r >0,\ d > 0,\ \kappa \in [0,1),\ \rho \in (0,1)$ as the parameters in the algorithm; $a,b>0$ as the characterization of the set $A$; $(c_X,\sigma,\nu)$ as the generating triplet of $X(t)$; 
    $\bar\sigma(\cdot)$ is defined in \eqref{def bar sigma}.

    \Statex
    \State Set $t_n = \ceil{\log_2(n^d)}$ and $\kappa_{n,m} = \kappa^m/n^r$
    \Statex 
    
    \If{$\text{Unif}(0,1) < w$} \Comment{Sample $J_n$ from $\Q_n$}
        \State Sample $J_n = \sum_{i = 1}^k z_i \mathbf{I}_{[u_i,n]}$ from $\P$
    \Else
        \State Sample $J_n = \sum_{i = 1}^k z_i \mathbf{I}_{[u_i,n]}$ from $\P(\ \cdot\ |B^\gamma_n)$ using Algorithm \ref{algoSampleJnFromQ}
    \EndIf
    \State Set $u_0 = 0, u_{k+1} = n$.
    
    \Statex 
    \State Sample $\tau \sim Geom(\rho)$ \Comment{Decide Truncation Index $\tau$}
    
    \Statex
    \For{$i = 1,2,\ldots,k+1$} \Comment{Stick-Breaking Procedure} 
        \For{$j = 1,2,\ldots,t_n + \tau$}
            \State Sample $V^{(i)}_j\sim \text{Unif}(0,1)$
            \State Set $l^{(i)}_j = V^{(i)}_j(u_{i} - u_{i+1} - l^{(i)}_1 - l^{(i)}_2 - \ldots - l^{(i)}_{j - 1})$
        \EndFor
        \State Set $l^{(i)}_{ t_n + \tau + 1 } = u_{i} - u_{i-1} - l^{(i)}_1 - l^{(i)}_2 - \ldots - l^{(i)}_{t_n + \tau}$
    \EndFor
    
    \Statex 
    \For{$i = 1,\cdots,k+1$ } \Comment{Sample $\xi^{(i),m}_j$}
        \For{$ j = 1,2,\cdots,t_n + \tau + 1$}
            \State Sample $x^{(i)}_j\sim N(0, \sigma^2\cdot l^{(i)}_j )$ 
            \State Sample $y^{(i),-1}_{j}\sim \P(J_{n,-1}(l^{(i)}_j) \in \ \cdot\ )$
            \For{$m = 0,1,\ldots,\tau$}
                \State Sample $y^{(i),m}_j \sim \P(J_{n,m}(l^{(i)}_j) \in \ \cdot\ )$
                \State Sample $w^{(i),m}_j \sim N(0,(\bar{\sigma}^2(\kappa_{n,m-1}) - \bar{\sigma}^2(\kappa_{n,m}))\cdot l^{(i)}_j)$
            \EndFor
            \State Sample $w^{(i),\tau+1}_j \sim N(0,\bar{\sigma}^2( \kappa_{n,\tau} ) \cdot l^{(i)}_j)$
            \For{$m = 0,\ldots,\tau$}
                \State Set $\xi^{(i),m}_j = c_X\cdot l^{(i)}_j +x^{(i)}_j + \sum_{q = -1}^{m}y^{(i),q}_j + \sum_{q = m+1}^{\tau+1}w^{(i),q}_j$
            \EndFor
            
        \EndFor 
    \EndFor 
    
    \Statex
    \For{$m = 1,\cdots,\tau$} \Comment{Evaluate $\hat Y^m_{n}$}
        \For{$i = 1,2,\ldots,k+1$}
            \State Set $\hat{M}^{(i),m}_n = \sum_{q = 1}^{i-1}\sum_{j =1}^{t_n + \tau + 1}\xi^{(q),m}_{j} + \sum_{q =1}^{i-1}z_q + \sum_{j =1}^{ t_n + m }(\xi^{(i),m}_{j})^+   $
        \EndFor 
        \State Set $\hat Y^m_{n} = \mathbf{I}\big\{ \max_{i = 1,\ldots,k+1} \hat{M}^{(i),m}_n   \geq na  \big\}$
    \EndFor

\Statex
\State Set $Z_n = \hat Y^1_n + \sum_{m = 2}^\tau ( \hat Y^m_{n} - \hat Y^{m-1}_{n} )\big/ \rho^{m-1}$ \Comment{Return the Estimator $L_n$}
\If{ $ \max_{i =1,\cdots,k}z_i > b $ }
    \State \textbf{Return} $L_n = 0$.
\Else 
    \State Set $\lambda_n = n\nu[n\gamma,\infty),\ p_n = 1 - \sum_{l = 0}^{l^* - 1}e^{-\lambda_n}\frac{\lambda_n^l}{l!},\ I_n = \mathbf{I}\{ J_n \in B^\gamma_n \}$
    \State \textbf{Return} $L_n = Z_n/(w + \frac{1-w}{p_n}I_n)$
\EndIf 
    
  \end{algorithmic}
\end{algorithm}

To demonstrate our approach for establishing condition~\eqref{condition, sec: lipschitz cont, A2},
we start by considering a simple case where the L\'evy process $X(t)$ has generating tripet $(c_X,\sigma,\nu)$ with $\sigma > 0$.
This leads to the decomposition
$$X^{<z}(t) \distequal \sigma B(t) + Y^{<z}(t)\qquad \forall t,z > 0 $$
where $B$ is a standard Brownian motion, $Y^{<z}$ is a L\'evy process with generating triplet $(c_X,0,\nu|_{(-\infty,z)})$, and the two processes are independent.
Now, for any $x \in \mathbb{R},\ t > 0$ and $\delta \in (0,1)$,
\begin{align}
    \P (X^{<z}(t) \in [x,x+\delta])
    &=\int_{\mathbb{R}}\P (\sigma B(t) \in [x-y,x-y+\delta])\cdot\P (Y^{<z}(t) \in dy) 
    \nonumber
    \\
    & = 
    \int_{\mathbb{R}}\P\bigg( \frac{B(t)}{ \sqrt{t} } \in \Big[ \frac{x-y}{ \sigma \sqrt{t} },\frac{x-y+\delta}{ \sigma \sqrt{t} }\Big]\bigg)\cdot\P (Y^{<z}(t) \in dy)
    \nonumber
    \\
    &\leq\frac{1}{\sigma\sqrt{2\pi}}\cdot\frac{\delta}{\sqrt{t}}.
    \label{fact, convolution structure, lipschit cont A2}
\end{align}
The last inequality follows from the fact that a standard Normal distribution admits a density function bounded by $1/\sqrt{2\pi}$.
Therefore, we verified Assumption~\ref{assumption: holder continuity strengthened on X < z t} under $\lambda = 1/2, C = \frac{1}{\sigma \sqrt{2\pi}}$, and any $z_0 > 0$.
The simple idea behind \eqref{fact, convolution structure, lipschit cont A2} is that continuity conditions such as \eqref{condition, sec: lipschitz cont, A2} can be passed from one distribution to another through convolutional structures.
To generalize this approach to the scenarios where $\sigma = 0$ in the generating triplet of the L\'evy process $X$,
we introduce the following definition.

\begin{definition}
Let $\mu_1$ and $\mu_2$ be Borel measures on $\R$.
For any Borel set $A\subset \R$, we say that $\mu_1$ \textbf{majorizes} $\mu_2$ \textbf{when restricted on } $A$
(denoted as $\notationdef{notation-def-measure-majorization}{(\mu_1 - \mu_2)|_A \geq 0}$)
if 
$\mu(B \cap A) = \mu_1(B \cap A) - \mu_2(B \cap A) \geq 0$ for any Borel set $B \subset \R$.
In other words $\mu|_A = (\mu_1 - \mu_2)|_A$ is a \textbf{positive} measure.
\end{definition}


Now, let us consider the case where the generating triplet of $X$ is $(c_X,0,\nu$).
For the L\'evy measure $\nu$,
if we can find some $z_0 > 0$, some Borel set $A \subseteq (-\infty,z_0)$ and some (positive) Borel measure $\mu$ such that $(\nu - \mu)|_A\geq 0$,
then through a straightforward superposition of Poisson random measures,
we obtain the decomposition (let $\mu_A = \mu|_A$)
\begin{align}
    X^{<z}(t) \distequal Y(t) + \widetilde X^{<z,-A}(t)\qquad \forall z \geq z_0
    \label{decomp, X < z without BM}
\end{align}
where 
$Y(t)$ is a L\'evy process with generating triplet $(0,0,\mu_A)$,
$\widetilde X^{<z,-A}(t)$ is a L\'evy process with generating triplet $(c_X,0,\nu -\mu_A)$,
and the two processes are independent.
Furthermore, if Assumption~\ref{assumption: holder continuity strengthened on X < z t} (conditions of form \eqref{condition, sec: lipschitz cont, A2}) holds for the process $Y(t)$
with generating triplet $(0,0,\mu_A )$,
then by repeating the arguments in \eqref{fact, convolution structure, lipschit cont A2} we can show that Assumption~\ref{assumption: holder continuity strengthened on X < z t} holds in $X^{<z}(t)$ for any $z \geq z_0$.

Recall our running assumption that the L\'evy process $X(t)$ is of infinite activities (see Assumption~\ref{assumption: heavy tailed process}).
In case that $\sigma = 0$, we must have $\nu((-\infty,0)\cup(0,\infty)) = \infty$ for $X$ to have infinite activity.
Therefore, the key step is to identify the majorized measure $\mu$ such that 
\begin{itemize}
    \item $(\nu - \mu)|_A \geq 0$
holds for $\nu$ with infinite mass and some set $A$,

\item condition~\eqref{condition, sec: lipschitz cont, A2} holds for the L\'evy process $Y(t)$ in \eqref{decomp, X < z without BM} with generating triplet $(0,0,\mu|_A)$.
\end{itemize}

In the first main result of this section,
we show that measures of form $\mu[x,\infty)$ that roughly increase at a power-law rate $1/x^{\alpha}$ (as $x\downarrow 0$) provide ideal choices for such majorized measures.
In particular, the corresponding L\'evy process $Y(t)$ in \eqref{decomp, X < z without BM} is intimately related to $\alpha$-stable processes that naturally satisfy continuity properties of form \eqref{condition, sec: lipschitz cont, A2}.
We collect the proof in
 Section~\ref{subsec: proof, lipschitz cont}.





\begin{proposition}
\label{proposition: lipschitz cont, 1}
\linksinthm{proposition: lipschitz cont, 1}
Let $\alpha \in (0,2), z_0 > 0$, and $\epsilon \in (0, (2 - \alpha)/2)$.
Suppose that $\mu[x,\infty)$ is regularly varying as $x \downarrow 0$ with index $-(\alpha+2\epsilon)$.
Then the L\'evy process $Y(t)$ with generating triplet $(0,0,\mu|_{(0,z_0)})$ has a continuous density function $f_{Y(t)}$ for each $t > 0$.
Furthermore,
there exists a constant $C < \infty$ such that
$$
\norm{f_{Y(t)}}_\infty \leq \frac{C}{t^{1/\alpha} \wedge 1}\qquad \forall t > 0.
$$
where $\notationdef{notation-density-norm}{\norm{ f }_\infty} = \sup_{x \in \R}|f(x)|$.
\end{proposition}


Equipped with Proposition~\ref{proposition: lipschitz cont, 1}, we obtain the following set of sufficient conditions for Assumption~\ref{assumption: holder continuity strengthened on X < z t}.

\begin{theorem}\label{ CorollaryRVlevyMeasureAtOrigin }
\linksinthm{ CorollaryRVlevyMeasureAtOrigin }
Let $(c_X,\sigma,\nu)$ be the generating triplet of L\'evy process $X$.
\begin{enumerate}[(i)]
    \item If $\sigma > 0$, then Assumption \ref{assumption: holder continuity strengthened on X < z t} holds for 
     $\lambda = 1/2$ and any $z_0 > 0$.
     \item If there exist Borel measure $\mu$, some $z_0 > 0$, and some $\alpha^\prime \in (0,2)$ such that $(\nu - \mu)|_{ (0,z_0) } \geq 0$ 
     (resp., $(\nu - \mu)|_{ (-z_0,0) } \geq 0$)
     and $\mu[x,\infty)$ (resp., $\mu(-\infty,x]$) is regularly varying with index $-\alpha^\prime$ as $x \downarrow 0$,
     then Assumption \ref{assumption: holder continuity strengthened on X < z t} holds with
     $\lambda = 1/\alpha$ for any $\alpha \in (0,\alpha^\prime)$.
\end{enumerate}
\end{theorem}

\begin{proof}
Part $(i)$ follows immediately from the calculations in \eqref{fact, convolution structure, lipschit cont A2}.
To prove part $(ii)$, we fix some $\alpha \in (0,\alpha^\prime)$,
and without loss of generality assume that 
$(\nu - \mu)|_{ (0,z_0) } \geq 0$ 
     and $\mu[x,\infty)$ is regularly varying with index $\alpha^\prime$ as $x \downarrow 0$.
This allows us to fix some $\epsilon = (\alpha^\prime-\alpha)/2 \in \big(0,(2-\alpha)/2\big)$.

For any $z \geq z_0$, 
let $Y(t)$ and $\widetilde X^{<z,-A}(t)$ be defined as in \eqref{decomp, X < z without BM} with $A = (0,z_0)$.
First of all, applying Proposition \ref{proposition: lipschitz cont, 1}, we can find $C > 0$ such that
$
\norm{f_{Y(t)}}_\infty \leq \frac{C}{t^{1/\alpha} \wedge 1}\ \forall t > 0.
$
Next, due to the independence between $Y$ and $\widetilde X^{<z,-A}(t)$,
it holds for all $x \in \R, \delta \geq 0$, and $t > 0$ that 
$$
\P (X^{<z}(t) \in [x,x+\delta])
    =\int_{\mathbb{R}}\P (Y(t) \in [x-y,x-y+\delta])\cdot\P (\widetilde X^{<z,-A}(t) \in dy)
    \leq\frac{C}{t^{1/\alpha} \wedge 1} \cdot \delta.
$$
This concludes the proof.
\end{proof}

\begin{remark}
It is worth noting that the conditions stated in Theorem \ref{ CorollaryRVlevyMeasureAtOrigin } are mild for L\'evy process $X(t)$ with infinite activities.
In particular, for $X$ to exhibit infinite activity,
we must have either $\sigma > 0$ or $\nu(\R) = \infty$.
Theorem~\ref{ CorollaryRVlevyMeasureAtOrigin } (i) deals with the case where $\sigma > 0$.
On the other hand, when $\sigma = 0$ we must have either $\lim_{\epsilon \downarrow 0}\nu[\epsilon,\infty) = \infty$ or 
$\lim_{\epsilon \downarrow 0}\nu(-\infty,-\epsilon] = \infty$.
To satisfy the conditions in part (ii) of Theorem \ref{ CorollaryRVlevyMeasureAtOrigin },
the only other requirement is that $\nu[\epsilon,\infty)$ (or $\nu(-\infty,-\epsilon]$)
approaches infinity at a rate that is at least comparable to some power-law functions.
\end{remark}

The next set of sufficient conditions for Assumption~\ref{assumption: holder continuity strengthened on X < z t}
revolves around another type of self-similarity structure in the L\'evy measure $\nu$.
\begin{definition}\label{def: semi stable process}
    Given $\alpha \in (0,2)$ and $b > 1$, a L\'evy process $X$ is \textbf{$\alpha$-semi-stable with span $b$} if its L\'evy measure $\nu$ satisfies
\begin{align}
    \nu = b^{-\alpha}T_b \nu \label{semiStable_scaleInvariance_lCont}
\end{align}
where the transformation $T_r$ ($\forall r > 0$) onto a Borel measure $\rho$ on $\mathbb{R}$ is given by $(T_r \rho)(B) = \rho(r^{-1}B)$. 
\end{definition}
As a special case of semi-stable processes, note that $X$ is {$\alpha$-stable} if
$$
\nu(dx) = c_1\cdot\frac{dx}{x^{1+\alpha}}\mathbf{I}\{x > 0\} + c_2\cdot\frac{dx}{|x|^{1+\alpha}}\mathbf{I}\{x < 0\}
$$
where $c_1,c_2 \geq 0,\ c_1+c_2 > 0.$  See Theorem 14.3 in \cite{sato1999levy} for details.
However, it is worth noting that the L\'evy processes with regularly varying L\'evy measures $\nu$ studied in Proposition~\ref{proposition: lipschitz cont, 1} are not strict subsets of the semi-stable processes introduced in Definition~\ref{def: semi stable process}.
For instance,
given
a Borel measure $\nu$, suppose that $f(x) = \nu\big( (-\infty,-x]\cup[x,\infty) \big)$ is regularly varying at 0 with index $\alpha > 0$.
Even if $\nu$ satisfies the scaling-invariant property in \eqref{semiStable_scaleInvariance_lCont} for some $b>1$,
we can fix a sequence of points $\{x_n = \frac{1}{b^n}\}_{n \geq 1}$ and assign an extra mass of $\ln n$ onto $\nu$ at each point $x_n$.
In doing so, we break the scaling-invariant property but still maintain the regular variation of $\nu$.
On the other hand, to show that semi-stable processes may not have regularly varying L\'evy measure (when restricted on some neighborhood of the origin), let us consider a simple example.
For some $b>1$ and $\alpha \in (0,2)$, define the following measure:
$$
 \nu( \{ b^{-n} \} ) = b^{n\alpha}\ \ \forall n \geq 0;\qquad \nu\big( \mathbb{R}\symbol{92}\{b^n:\ n\in\mathbb{N}\} \big) = 0.
$$
Clearly, $\nu$ can be seen as the restriction of the L\'evy measure (restricted on $(-1,1)$) of some $\alpha$-semi-stable process. Now define function $f(x) = \nu[x,\infty)$ on $(0,\infty)$. For any $t > 0$,
$$
\frac{f(tx)}{f(x)} = \frac{ \sum_{n = 0}^{\floor{ \log_b(1/tx) }}b^{n\alpha}    }{ \sum_{n = 0}^{\floor{ \log_b(1/x) }}b^{n\alpha}  }
    = \frac{ b^{ \floor{ \log_b(1/tx) } + 1 } - 1   }{ b^{ \floor{ \log_b(1/x) } + 1 } -1  }.
$$
As $x \rightarrow 0$, we see that $f(tx)/f(x)$ will be very close to 
$$b^{\alpha(  \floor{ \log_b(1/tx) } - \floor{ \log_b(1/x) } )}.$$
As long as we didn't pick $t = b^k$ for some $k \in \mathbb{Z}$,
asymptotically, the value of $f(tx)/f(x)$ will repeatedly cycle through the following three different values
$$\{ b^{\alpha\floor{ \log_b(1/t) }}, b^{\alpha\floor{ \log_b(1/t) }+\alpha},b^{\alpha\floor{ \log_b(1/t) }-\alpha} \},$$
thus implying that $f(tx)/f(x)$ does not converge as $x$ approaches $0$.
This confirms that $\nu[x,\infty)$ is not regularly varying as $x \downarrow 0$.

In Proposition~\ref{proposition: lipschitz cont, 2}, we show that semi-stable processes, as well as their truncated counterparts, satisfy continuity conditions of form \eqref{condition, sec: lipschitz cont, A2}.
We say that the process $Y(t)$ is non-trivial if it is not a deterministic linear function (i.e., $Y(t) \equiv ct$ for some $c \in \R$).
The proof
is again detailed in Section~\ref{subsec: proof, lipschitz cont}.

\begin{proposition}
\label{proposition: lipschitz cont, 2}
\linksinthm{proposition: lipschitz cont, 2}
Let $\alpha \in (0,2)$ and $N \in \mathbb Z$. 
Suppose that $\mu$ is the L\'evy measure of a non-trivial $\alpha$-semi-stable process $Y^\prime(t)$ of span $b > 1$. 
Then under $z_0 = b^N$, the L\'evy process $\{Y(t):\ t>0\}$ with generating triplet $(0,0,\mu|_{(-z_0,z_0)})$ has a continuous density function $f_{Y(t)}$ for any $t > 0$.
Furthermore, there exists some $C\in (0,\infty)$ such that
$$\norm{ f_{Y(t)} }_\infty \leq \frac{C}{t^{1/\alpha}\wedge 1}\qquad \forall t > 0.$$
\end{proposition}

Lastly,
by applying Proposition \ref{proposition: lipschitz cont, 2},
we yield another set of sufficient conditions for Assumption \ref{assumption: holder continuity strengthened on X < z t}.

\begin{theorem}\label{ CorollarySemiStableLevyMeasureAtOrigin }
\linksinthm{ CorollarySemiStableLevyMeasureAtOrigin }
Let $(c_X,\sigma,\nu)$ be the generating triplet of L\'evy process $X$.
Suppose that there exist some Borel measure $\mu$ and some $z^\prime > 0,\ \alpha \in (0,2)$ such that
$
(\nu - \mu)|_{(-z^\prime,z^\prime)} \geq 0,
$
and $\mu$ is the L\'evy measure of some $\alpha$-semi-stable process.
Then  Assumption~\ref{assumption: holder continuity strengthened on X < z t} holds for
     $\lambda = 1/\alpha$.
\end{theorem}

\begin{proof}
Let $b > 1$ be the span of the $\alpha$-semi-stable process.
Fix some $N \in \mathbb Z$ such that $z_0 \delequal b^N \leq z^\prime$.
For any $z \geq z_0$,
let $Y(t)$ and $\widetilde X^{<z,-A}(t)$ be defined as in \eqref{decomp, X < z without BM} with $A = (-z_0,z_0)$.
First of all, applying Proposition \ref{proposition: lipschitz cont, 2}, we can find $C > 0$ such that
$
\norm{f_{Y(t)}}_\infty \leq \frac{C}{t^{1/\alpha} \wedge 1}\ \forall t > 0.
$
Next, due to the independence between $Y$ and $\widetilde X^{<z,-A}(t)$,
it holds for all $x \in \R, \delta \geq 0$, and $t > 0$ that 
$$
\P (X^{<z}(t) \in [x,x+\delta])
    =\int_{\mathbb{R}}\P (Y(t) \in [x-y,x-y+\delta])\cdot\P (\widetilde X^{<z,-A}(t) \in dy)
    \leq\frac{C}{t^{1/\alpha} \wedge 1} \cdot \delta.
$$
This concludes the proof.
\end{proof}

\section{Numerical Experiments}
\label{sec: experiment}

In this section, we apply the importance sampling strategy outlined in Algorithms~\ref{algoISnoARA}  and \ref{algoIS} and demonstrate that
$(i)$ the performance of the importance sampling estimators under different scaling factors and tail distributions,
and $(ii)$ the strong efficiency of the proposed algorithms when compared to crude Monte Carlo methods. 
Specifically, consider a L\'evy process
$X(t) = B(t) + \sum_{i=1}^{N(t)}W_i$,
where $B(t)$ is the standard Brownian motion, $N$ is a Poisson process with arrival rate $0.5$, and $\{W_i\}_{i \geq 1}$ is a sequence of iid random variables with law (for some $\alpha > 1$)
$$
    \P(W_1 > x) = \P(-W_1 > x) = \frac{0.5}{ (1 + x)^\alpha  },\qquad 
    \forall x > 0.
$$
For each $n \geq 1$, we define the scaled process $\bar X_n(t) = \frac{X(nt)}{n}$.
The goal is to estimate the probability of $A_n = \{X_n \in A\}$, where the set $A$ is defined as in \eqref{def set A}
with $a = 2$ and $b = 1.15$. 
Note that this is a case with $l^* = \ceil{a/b} = 2$.

To evaluate the performance of the importance sampling estimator under different scaling factors and tail distributions,
we run experiments under $\alpha \in \{1.45,1.6,1.75\}$, and $n \in \{100,200,\cdots,1000\}$. 
The efficiency is evaluated by the {relative error} of the algorithm, namely the ratio between the standard deviation and the estimated mean.
In Algorithm~\ref{algoISnoARA},
we set $\gamma = 0.25,\ w = 0.05, \rho = 0.97$, and $d = 4$.
In Algorithm~\ref{algoIS}, we further set $\kappa = 0.5$ and $r = 1.5$.
For both algorithms, we generate 10,000 independent samples for each combination of 
$\alpha \in \{1.45,1.6,1.75\}$ and $n \in \{1000,2000,\cdots,10000\}$.
For the number of samples in crude Monte Carlo estimation, we ensure that at least $64/\hat{p}_{\alpha,n}$ samples are generated,
where $\hat{p}_{\alpha,n}$ is the probability estimated by Algorithm~\ref{algoISnoARA}.

The results are summarized in Table \ref{tab: IS relative error} and Figure \ref{fig: reinsurance}. In Table \ref{tab: IS relative error}, we see that for a fixed $\alpha$, the relative error of the importance sampling estimators stabilizes around a constant level as $n$ increases.
This aligns with
the strong efficiency established in Theorems~\ref{theorem: strong efficiency without ARA} and \ref{theorem: strong efficiency}.
In comparison, the relative error of crude Monte Carlo estimators continues to increase as $n$ tends to infinity.
Figure~\ref{fig: reinsurance} further highlights that our importance sampling estimators significantly outperform crude Monte Carlo methods by orders of magnitude.
In summary, when  
Algorithms~\ref{algoISnoARA} and \ref{algoIS} are appropriately parameterized,
their efficiency becomes increasingly evident when compared against the crude Monte Carlo approach as the scaling factor $n$ grows larger and the target probability approaches $0$.
\begin{figure}[htb]
{
\centering
\includegraphics[width=0.7\textwidth]{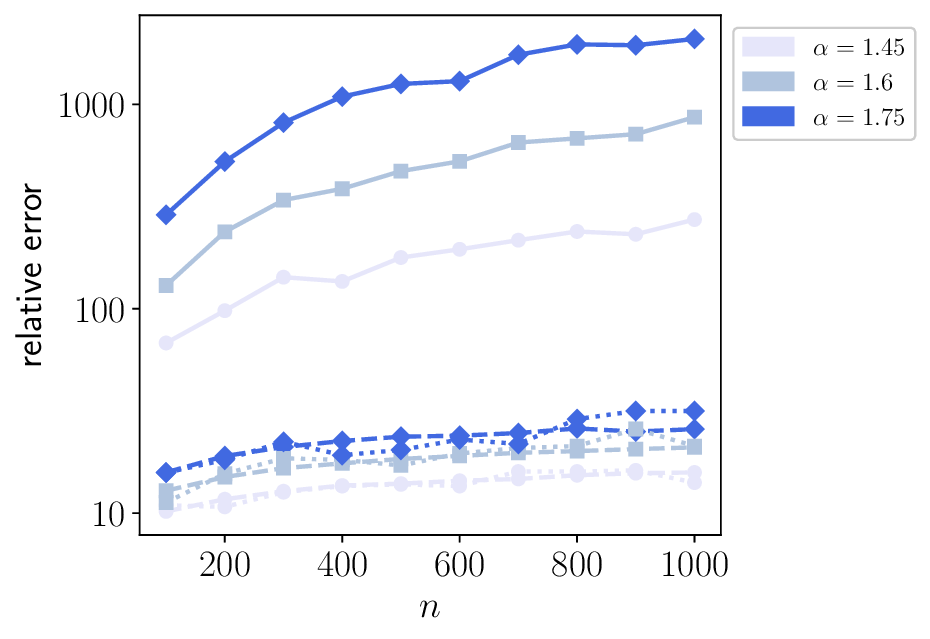}
\caption{
Relative errors of the proposed importance sampling estimator.
Results are plotted under log scale. 
\textbf{Dashed lines}: the importance sampling estimator in
Algorithm~\ref{algoISnoARA},
\textbf{Dotted lines}: 
the importance sampling estimator with ARA in Algorithm~\ref{algoIS}, 
\textbf{Solid lines}: the crude Monte-Carlo methods (solid lines).
\label{fig: reinsurance}}
}
\end{figure}

\begin{table}[htb]
\centering
\caption{Relative errors of Algorithm~\ref{algoISnoARA} (first row),  Algorithm~\ref{algoIS} (second row), and crude Monte Carlo (third row). }\label{tab: IS relative error}
\begin{tabular}{c c c c c c} 
 \hline
 n & 200 & 400 & 600 & 800 & 1000 \\
 \hline
 \multirow{3}{4em}{$\alpha = 1.45$} 
 & 
 $11.70$ & $13.65$ & $14.40$ & $15.33$ & $15.82$ \\
 & $10.74$ & $13.57$ & $13.57$ & $15.98$ & $14.11$ \\
 & 97.86 & 136.03 & 195.40 & 238.81 & 273.13 \\ 
 \hline
 \multirow{3}{4em}{$\alpha = 1.6$} 
 & 
 $15.03$ & $17.53$ & $19.06$ & $20.12$ & $20.98$ \\
 & 15.59 & 18.23 & 19.59 & 21.30 & 21.30 \\
 & 237.82 & 386.35 & 526.13 & 681.79 & 866.02 \\ 
 \hline
\multirow{3}{4em}{$\alpha = 1.75$} 
& 
$19.03$ & $22.54$ & $23.94$ & $25.97$ & $25.77$ \\
 & 18.23 & 19.22 & 22.92 & 28.85 & 31.61 \\
 & 524.78 & 1091.29 & 1298.98 & 1965.22 & 2089.82 \\
 \hline
\end{tabular}
\end{table}

\section{Proofs}
\label{sec: proof}

\subsection{Proof of Proposition~\ref{proposition: design of Zn}}
\label{subsec: proof, prop strong efficiency}

We first prepare two technical lemmas using the sample-path large deviations for heavy-tailed L\'evy processes reviewed in Section~\ref{subsec: review, LD of levy}.

\begin{lemma}
\label{lemma: LD, events A n}
\linksinthm{lemma: LD, events A n}
For the set $A\subset \D$ defined in \eqref{def set A} and the quantity $l^*$ defined in \eqref{def l *},
\begin{align*}
     0 <  \liminf_{n \rightarrow \infty}\frac{\P(\bar X_n \in A)}{ (n\nu[n,\infty))^{l^*} }
    \leq \limsup_{n \rightarrow \infty}\frac{\P(\bar X_n \in A)}{ (n\nu[n,\infty))^{l^*} } < \infty.
\end{align*}
\end{lemma}

\begin{proof}
\linksinpf{lemma: LD, events A n}
In this proof, we focus on the two-sided case in Assumption \ref{assumption: heavy tailed process}.
It is worth noticing that analysis for the one-sided case is almost identical, with the only major difference being that we apply Result \ref{result: LD of Levy, one-sided} (i.e., the one-sided version of the large deviations of $\bar X_n$) instead of Result \ref{result: LD of Levy, two-sided} (i.e., the two-sided version).
Specifically, we claim that
\begin{enumerate}[($i$)]
    \item 
$
\big(l^*,0\big)
    \in 
    \underset{ (j,k) \in \mathbb N^2,\ \mathbb D_{j,k} \cap A \neq \emptyset }{\text{argmin}}
    j(\alpha - 1) + k(\alpha^\prime - 1);
$

\item $ \mathbf{C}_{l^*,0}(A^\circ) > 0$; 

\item the set $A$ is bounded away from $\mathbb D_{<l^*,0}$.
\end{enumerate}
Then by applying Result~\ref{result: LD of Levy, two-sided},
we yield
\begin{align*}
0 < {\mathbf C_{l^*,0}(A^\circ)}
\leq 
    \liminf_{n \rightarrow \infty} \frac{\P(\bar X_n \in A)}{ (n\nu[n,\infty))^{l^*} }
    \leq 
    \limsup_{n \rightarrow \infty} \frac{\P(\bar X_n \in A)}{ (n\nu[n,\infty))^{l^*} }
    \leq {\mathbf C_{l^*,0}(A^-)} < \infty
\end{align*}
and conclude the proof.
Now, it remains to prove claims ($i$), ($ii$), and ($iii$).

\medskip
\noindent\textbf{Proof of Claim $(i)$}.

By definitions of $\mathbb D_{j,k}$, for any $\xi \in \mathbb D_{j,k}$
there exist $(u_i)_{i = 1}^j \in (0,\infty)^j$, $(t_i)_{i = 1}^j \in (0,1]^j$ and $(v_i)_{i = 1}^k \in (0,\infty)^k$, $(s_i)_{i = 1}^k \in (0,1]^k$ such that
\begin{align}
    \xi(t) = \sum_{i = 1}^j u_i\mathbf{I}_{ [t_i,1] }(t) - \sum_{i = 1}^k v_i\mathbf{I}_{ [s_i,1] }(t)
    \qquad \forall t \in [0,1].
    \label{proof, property of D j k, proposition: design of Zn}
\end{align}

First, 
from Assumption~\ref{assumption: choice of a,b},
one can choose $\epsilon > 0$ small enough such that $l^*(b - \epsilon) > a$.
Then for the case with $(j,k) = (l^*,0)$ in \eqref{proof, property of D j k, proposition: design of Zn},
by picking $u_i = b - \epsilon$ for all $i \in [l^*]$, we have
$\sup_{t \in [0,1]}\xi(t) = \sum_{i = 1}^{l^*}u_i = l^*(b-\epsilon) > a$, and hence $\xi \in A$.
This verifies $\D_{l^*,0} \cap A \neq \emptyset$.


Next, suppose we can show that 
$j \geq l^*$ is a necessary condition for $\mathbb D_{j,k} \cap A \neq \emptyset$.
Then we get
\begin{align*}
    \big\{(j,k) \in \mathbb N^2:\ \mathbb D_{j,k} \cap A \neq \emptyset\big\} \subseteq
    \big\{(j,k) \in \mathbb N^2:\ j \geq l^*,\ k \geq 0\big\},
\end{align*}
which immediately verifies claim $(i)$ due to $\alpha,\alpha^\prime > 1$; see Assumption~\ref{assumption: heavy tailed process}.
Now, to show that $j \geq l^*$ is a necessary condition for $\mathbb D_{j,k} \cap A \neq \emptyset$,
note that from \eqref{proof, property of D j k, proposition: design of Zn}, it holds for any $\xi \in \D_{j,k} \cap A$ that
$
 a < \sup_{t \in [0,1]}\xi(t) \leq \sum_{i = 1}^j u_i < jb.
$
As a result, we must have $j > a/b$ and hence $j \geq l^* = \ceil{a/b}$ due to $a/b \notin \mathbb Z$; see Assumption~\ref{assumption: choice of a,b}. This concludes the proof of claim $(i)$.


\medskip
\noindent\textbf{Proof of Claim $(ii)$}.

Again, choose some $\epsilon > 0$ small enough such that $l^*(b - \epsilon) > a$.
Given any $u_i \in (b-\epsilon,b)$ and $0 < t_1 < t_2 < \cdots < t_{l^*} < 1$,
 the step function $ \xi(t) = \sum_{i = 1}^{l^*} u_i\mathbf{I}_{ [t_i,1] }(t)$ satisfies $\sup_{t \in [0,1]}\xi(t) \geq l^*(b-\epsilon) > a$, thus implying $\xi \in A$.
Therefore, (for the definition of $\mathbf C_{j,k}$, see \eqref{def: measure C j k})
\begin{align*}
    \mathbf{C}_{l^*,0}(A^\circ) & \geq 
    \nu^{l^*}_\alpha\Big( (b-\epsilon,b)^{l^*} \Big) = \frac{1}{l^*!}\bigg[ \frac{1}{ (b-\epsilon)^\alpha } - \frac{1}{b^\alpha} \bigg]^{l^*} > 0.
\end{align*}

\medskip
\noindent\textbf{Proof of Claim $(iii)$}.

Assumption~\ref{assumption: choice of a,b} implies that $a > (l^*-1)b$,
allowing us to choose $\epsilon > 0$ small enough that $a - \epsilon > (l^* - 1)(b + \epsilon)$.
It suffices to show that
\begin{align}
    \bm{d}(\xi,\xi^\prime) \geq \epsilon
    \qquad 
    \forall \xi \in \mathbb D_{<l^*,0},\ \xi^\prime \in A.
    \label{goal, claim iii, lemma: LD, events A n}
\end{align}
Here, $\bm{d}$ is the Skorokhod $J_1$ metric; see \eqref{def: J 1 metric} for the definition.
To prove \eqref{goal, claim iii, lemma: LD, events A n}, we start with the following observation:
due to claim $(i)$, for any $(j,k) \in \mathbb N^2$ with $(j,k)\in \mathbb I_{<l^*,0}$, we must have $j \leq l^* - 1$.
Now, we proceed with a proof by contradiction. Suppose there is some $\xi \in \D_{j,k}$ with $j \leq l^* - 1$ and some $\xi^\prime \in A$ such that $\bm d(\xi,\xi^\prime) < \epsilon$.
Due to $\xi^\prime \in A$ (and hence no upward jump in $\xi^\prime$ is larger than $b$) and $\bm d(\xi,\xi^\prime) < \epsilon$,
under the representation \eqref{proof, property of D j k, proposition: design of Zn}
we must have $u_i < b + \epsilon\ \forall i \in [j]$.
This implies $\sup_{t \in[0,1]} \xi(t) \leq \sum_{i = 1}^j u_i < j(b+\epsilon) \leq (l^* -1)(b+\epsilon)$.
Due to $\bm d(\xi,\xi^\prime) < \epsilon$ again,
we yield the contradiction that $\sup_{t \in [0,1]}\xi^\prime(t) < (l^*-1)(b+\epsilon) + \epsilon < a$ (and hence $\xi^\prime \notin A$). This concludes the proof of claim $(iii)$.
\end{proof}

\begin{lemma}
\label{lemma: LD, event A Delta}
\linksinthm{lemma: LD, event A Delta}
Let $p > 1$. Let  $\Delta > 0$ be such that $a - \Delta > (l^* - 1)b$ and $[a - \Delta - (l^* - 1)b]/\gamma \notin \mathbb Z$.
Suppose that $
(J_\gamma + l^* - 1)/p > 2l^*
$
holds for
\begin{align}
    J_\gamma  \delequal \ceil{\frac{a - \Delta - (l^* - 1)b}{\gamma}}.
    \label{proof, def J gamma, lemma: LD, event A Delta}
\end{align}
Then
    $$\P\big(\bar X_n \in A^\Delta \cap E,\ \mathcal{D}(\bar J_n) \leq l^* - 1\big) = \bm{o}\Big(\big(n\nu[n,\infty)\big)^{ 2pl^* }\Big)
    \qquad \text{as }n\to\infty
    $$
where 
$A^\Delta = \{\xi \in \mathbb{D}: \sup_{t \in [0,1]}\xi(t)\geq a - \Delta\},\ E \delequal\{ \xi \in \mathbb D: \sup_{t \in (0,1] }\xi(t) - \xi(t-) < b \}$
and
the function $\mathcal{D}(\xi)$ counts the number of discontinuities in $\xi \in \mathbb D$.
\end{lemma}

\begin{proof}
\linksinpf{lemma: LD, event A Delta}
Similar to the proof of Lemma \ref{lemma: LD, events A n},
we focus on the two-sided case in Assumption \ref{assumption: heavy tailed process}.
Still, it is worth noticing that the proof of the one-sided case is almost identical, with the only major difference being that we apply Result~\ref{result: LD of Levy, one-sided} instead of Result~\ref{result: LD of Levy, two-sided}.

First, observe that 
$
\P(\bar X_n \in A^\Delta \cap E,\ \mathcal{D}(\bar J_n) \leq l^* - 1) = \P(\bar X_n \in F)
$
where 
\begin{align*}
    F & \delequal 
    \big\{ \xi \in \D:\ \sup_{t \in [0,1]}\xi(t)\geq a - \Delta;\ \sup_{t \in (0,1] }\xi(t) - \xi(t-) < b,
    \\ 
    &\qquad \qquad \qquad \qquad \qquad \qquad 
    \#\{ t \in [0,1]:\ \xi(t) - \xi(t-) \geq \gamma \} \leq l^* - 1 \big\}.
\end{align*}
Furthermore, we claim that 
\begin{enumerate}[$(i)$]
    \item 
$
(J_\gamma + l^* - 1,0)
    \in 
    \underset{ (j,k) \in \mathbb N^2,\ \mathbb D_{j,k} \cap F \neq \emptyset }{\text{argmin}}
    j(\alpha - 1) + k(\alpha^\prime - 1);
$

\item the set $F$ is bounded away from $\mathbb D_{<J_\gamma + l^* - 1,0}$.
\end{enumerate}
Then we are able to apply Result~\ref{result: LD of Levy, two-sided} and obtain
$$
\P(\bar X_n \in A^\Delta \cap E,\ \mathcal{D}(\bar J_n) \leq l^* - 1) 
= \P(\bar X_n \in F)
= \bm{O}\big( (n\nu[n,\infty))^{ J_\gamma + l^* - 1 } \big)
\qquad \text{as }n \to \infty.
$$
Lastly, by our assumption $(J_\gamma + l^* - 1)/p > 2l^*$, we get $(n\nu[n,\infty))^{ l^*  - 1+ J_\gamma }  = \bm{o}\big(\big(n\nu[n,\infty)\big)^{ 2pl^* } \big)$
and conclude the proof.
Now, it remains to prove claims $(i)$ and $(ii)$.

\medskip
\noindent\textbf{Proof of Claim $(i)$}.

By definition of $\mathbb D_{j,k}$, given any $\xi \in \mathbb D_{j,k}$
there exist $(u_i)_{i = 1}^j \in (0,\infty)^j, (t_i)_{i = 1}^j \in (0,1]^j$ and $(v_i)_{i = 1}^k \in (0,\infty)^k, (s_i)_{i = 1}^k \in (0,1]^k$ such that the representation \eqref{proof, property of D j k, proposition: design of Zn} holds.
By assumption $[a - \Delta - (l^* - 1)b]/\gamma \notin \mathbb Z$,
for $J_\gamma$ defined in \eqref{proof, def J gamma, lemma: LD, event A Delta} we have
\begin{align}
    (J_\gamma - 1)\gamma < a - \Delta - (l^* - 1)b < J_\gamma \cdot \gamma.
    \label{proof, ineq 1, lemma: LD, event A Delta}
\end{align}
It then holds for all $\epsilon > 0$ small enough that 
$a - \Delta < J_\gamma (\gamma - \epsilon) + (l^* - 1)(b-\epsilon)$.
As a result, for the case with $(j,k) = (l^* - 1 + J_\gamma,0)$ in \eqref{proof, property of D j k, proposition: design of Zn},
by picking 
$u_i = b - \epsilon$ for all $i \in [l^* - 1]$, and $u_i = \gamma - \epsilon$ for all $i = l^*, l^* + 1, \cdots, l^* - 1 + J_\gamma$,
we get $\sup_{t \in [0,1]}\xi(t) = J_\gamma(\gamma - \epsilon) + (l^*-1)(b- \epsilon) > a - \Delta$.
This proves that $\xi \in \D_{l^* - 1 + J_\gamma,0}\cap F$, and hence $\D_{l^* - 1 + J_\gamma,0}\cap F \neq \emptyset$.

Next, suppose we can show that
$j \geq l^* - 1 + J_\gamma$ is the necessary condition for $\D_{j,k}\cap F \neq \emptyset$.
Then, we get
\begin{align*}
    \{(j,k) \in \mathbb N^2:\ \D_{j,k}\cap F \neq \emptyset\}
    \subseteq \{(j,k) \in \mathbb N^2:\ j \geq l^* - 1 + J_\gamma,\ k \geq 0\},
\end{align*}
which immediately verifies claim $(i)$
due to $\alpha,\alpha^\prime > 1$; see Assumption \ref{assumption: heavy tailed process}.
Now, to show that $j \geq l^* - 1 + J_\gamma$ is a necessary condition,
note that, from \eqref{proof, property of D j k, proposition: design of Zn},
it holds for any $\xi \in \D_{j,k} \cap F$ that
$
 a - \Delta < \sup_{t \in [0,1]}\xi(t) \leq \sum_{i = 1}^j u_i.
$
Furthermore, by the definition of the set $F$,
we must have (here, w.l.o.g., we order $u_i$'s by $u_1 \geq u_2 \geq \ldots \geq u_j$) $u_i < b$ for all $i \in [l^* - 1]$ and $u_i < \gamma$ for all $i = l^*, l^* + 1, \cdots, j$.
This implies 
$
(l^*-1)b + (j-l^*+1)\gamma > a - \Delta,
$
and hence
$
j > \frac{a - \Delta - (l^*-1)b}{\gamma} + l^* - 1,
$
which is equivalent to $j \geq J_\gamma + l^* - 1$.

\medskip
\noindent\textbf{Proof of Claim $(ii)$}.

From \eqref{proof, ineq 1, lemma: LD, event A Delta},
we can fix some $\epsilon > 0$
small enough such that
\begin{align}
    a - \Delta - \epsilon > (l^* - 1)(b+ \epsilon) + (J_\gamma - 1)(\gamma + \epsilon).
    \label{proof, choice of epsilon, claim 2, lemma: LD, event A Delta}
\end{align}
It suffices to show that
\begin{align}
    \bm{d}(\xi,\xi^\prime) \geq \epsilon
    \qquad 
    \forall \xi \in \D_{< J_\gamma + l^* - 1,0 },\ \xi^\prime \in F.
    \label{goal, claim ii, lemma: LD, event A Delta}
\end{align}
Here, $\bm{d}$ is the Skorokhod $J_1$ metric; see \eqref{def: J 1 metric} for the definition.
To prove \eqref{goal, claim ii, lemma: LD, event A Delta},
we start with the following observation:
using claim $(i)$,
for any $(j,k) \in \mathbb N^2$ with $(j,k) \in \mathbb I_{ < J_\gamma + l^* - 1, 0 }$,
we must have $j \leq J_\gamma + l^* - 2$.
Next, we proceed with a proof by contradiction.
Suppose there is some $\xi \in \D_{j,k}$ with $j \leq J_\gamma + l^* - 2$
and some $\xi^\prime \in F$ such that $\bm d(\xi,\xi^\prime) < \epsilon$.
By the definition of the set $F$ above,
any upward jump in $\xi^\prime$ is bounded by $b$, and at most $l^* - 1$ of them is larger than $\gamma$.
Then from $\bm d(\xi,\xi^\prime) < \epsilon$,
we know that any upward jump in $\xi$ is bounded by $b + \epsilon$, and at most $l^* - 1$ of them is larger than $\gamma + \epsilon$.
Through \eqref{proof, property of D j k, proposition: design of Zn}, we now have
\begin{align*}
    \sup_{t \in [0,1]}\xi(t) & \leq \sum_{i = 1}^j u_i
    \leq (l^* - 1)(b + \epsilon) + (J_\gamma -  1)(\gamma + \epsilon) < a - \Delta - \epsilon.
\end{align*}
The last inequality follows from \eqref{proof, choice of epsilon, claim 2, lemma: LD, event A Delta}.
Using $\bm d(\xi,\xi^\prime) < \epsilon$ again, 
we yield the contraction that $\sup_{t \in [0,1]}\xi^\prime(t) < a - \Delta$ and hence $\xi^\prime \notin F$.
This concludes the proof of \eqref{goal, claim ii, lemma: LD, event A Delta}.
\end{proof}

Now, we are ready to prove Proposition \ref{proposition: design of Zn}.

\begin{proof}[Proof of Proposition~\ref{proposition: design of Zn}]
\linksinpf{proposition: design of Zn}
We start by proving the \textbf{unbiasedness} of the importance sampling estimator
\begin{align*}
    L_n = Z_n \frac{d \P}{d \Q_n} = 
    \sum_{m = 0}^\tau \frac{ \hat Y^m_n \mathbf{I}_{E_n}\frac{d \P}{d \Q_n} - \hat Y^{m-1}_n \mathbf{I}_{E_n}\frac{d \P}{d \Q_n} }{\P(\tau \geq m)}.
\end{align*}
under $\Q_n$.
Note that under both $\P$ and $\Q_n$, we have $\tau \sim \text{Geom}(\rho)$ (i.e., $\P(\tau \geq m)=\rho^{m-1}$) and that $\tau$ is independent of everything else.
In light of Result~\ref{resultDebias},
it suffices to verify $\lim_{m \to \infty}\E^{\Q_n}[Y_m] = \E^{\Q_n}[Y]$ and condition~\eqref{condition, resultDebias} (under $\Q_n$)
with the choice of
$
Y_m = \hat Y^m_n \mathbf{I}_{E_n}\frac{d\P}{d\Q_n}
$
and
$
Y = Y^*_n\mathbf{I}_{E_n}\frac{d\P}{d\Q_n}.
$
In particular, it suffices to show that (for any $n \geq 1$)
\begin{align}
    \sum_{m \geq 1}\E^{\Q_n}\Bigg[\bigg|  \hat Y^{m-1}_n \mathbf{I}_{E_n}\frac{d \P}{d \Q_n} - Y^*_n \mathbf{I}_{E_n}\frac{d \P}{d \Q_n}\bigg|^2 \Bigg]\Bigg/ \P(\tau \geq m) < \infty.
    \label{proof, goal, unbiasedness, proposition: design of Zn}
\end{align}
To see why, note that \eqref{proof, goal, unbiasedness, proposition: design of Zn} directly verifies condition~\eqref{condition, resultDebias}.
Furthermore,
it implies
$
\lim_{m \to \infty}\E^{\Q_n}\big|  \hat Y^{m-1}_n \mathbf{I}_{E_n}\frac{d \P}{d \Q_n} - Y^*_n \mathbf{I}_{E_n}\frac{d \P}{d \Q_n}\big|^2 = 0.
$
The $\mathcal L_2$ convergence then implies the $\mathcal L_1$ convergence, i.e.,
$
\lim_{m \to \infty}\E^{\Q_n}[\hat Y^{m-1}_n \mathbf{I}_{E_n}\frac{d \P}{d \Q_n}] = \E^{\Q_n}[Y^*_n \mathbf{I}_{E_n}\frac{d \P}{d \Q_n}].
$

To prove claim \eqref{proof, goal, unbiasedness, proposition: design of Zn},
observe that
\begin{align*}
    \E^{\Q_n}\Bigg[\bigg|  \hat Y^{m-1}_n \mathbf{I}_{E_n}\frac{d \P}{d \Q_n} - Y^*_n \mathbf{I}_{E_n}\frac{d \P}{d \Q_n}\bigg|^2\Bigg]
    &
    \leq 
    \E^{\Q_n}
    \bigg[
    |  \hat Y^{m-1}_n  - Y^*_n|^2 \cdot  \bigg(\frac{d \P}{d \Q_n}\bigg)^2
    \bigg]
    \\ 
    & = 
    \E
    \bigg[
    |  \hat Y^{m-1}_n  - Y^*_n|^2 \cdot  \frac{d \P}{d \Q_n}
    \bigg]
    \\ 
    & \leq \frac{1}{w}\E
    |  \hat Y^{m-1}_n  - Y^*_n|^2
    \qquad 
    \text{due to } \frac{d \P}{d \Q_n} \leq \frac{1}{w}\text{, see \eqref{def: estimator Ln}}.
\end{align*}
In particular, since $\hat Y^m_n$ and $Y^*_n$ only take values in $\{0,1\}$,
we have 
$
\E|  \hat Y^{m}_n  - Y^*_n|^2 = \P(\hat Y^{m}_n \neq Y^*_n),
$
and 
\begin{align}
  \P(\hat Y^{m}_n \neq Y^*_n)
  & = \sum_{k \geq 0}\P(Y^*_n \neq \hat Y^m_n\ |\ \mathcal{D}(\bar J_n) = k)\P(\mathcal{D}(\bar J_n) = k)
  \nonumber
  \\ 
  & \leq
  \sum_{k \geq 0} 
   C_0 \rho^m_0 \cdot (k+1) \cdot \P(\mathcal{D}(\bar J_n) = k)
  \qquad \text{for all $m \geq \bar m$ due to \eqref{condition 1, proposition: design of Zn}}
  \nonumber
  \\ 
  & = C_0\rho^m_0 \cdot \E\Big[1 + \text{Poisson}\big( n\nu[n\gamma,\infty)\big)\Big]
  = C_0\rho^m_0 \cdot\big(1 + n\nu[n\gamma,\infty)\big).
  \label{proof: P hat Y m n neq Y * n, proposition: design of Zn}
\end{align}
The last line in the display above follows from the definition of $\bar J_n(t) = \frac{1}{n}J(nt)$ in \eqref{def: process J n Xi n}.
To conclude, 
note that $\nu(x) \in \RV_{-\alpha}(x)$ and hence $n\nu[n\gamma,\infty) \in \RV_{-(\alpha - 1)}(n)$ with $\alpha > 1$,
thus implying $\sup_{n \geq 1}n\nu[n\gamma,\infty) < \infty$;
also, as prescribed in Proposition~\ref{proposition: design of Zn} we have $\rho \in (\rho_0, 1)$.

The rest of the proof is devoted to establishing the \textbf{strong efficiency} of $L_n$.
Observe that
\begin{align*}
    \E^{\Q_n}[L^2_n]
    & = \int Z^2_n \frac{d\P}{d\Q_n}\frac{d\P}{d\Q_n} d\Q_n
    = \int Z^2_n \frac{d\P}{d\Q_n} d\P
    = \int Z^2_n \mathbf{I}_{ B^\gamma_n }\frac{d\P}{d\Q_n} d\P 
    +
    \int Z^2_n \mathbf{I}_{ (B^\gamma_n)^c }\frac{d\P}{d\Q_n} d\P.
\end{align*}
By definitions in \eqref{def: estimator Ln}, on event $(B^\gamma_n)^c$ we have $\frac{d\P}{d\Q_n} \leq \frac{1}{w}$,
while on event $B^\gamma_n$ we have $\frac{d\P}{d\Q_n} \leq \frac{\P(B^\gamma_n)}{1 - w}$.
As a result,
\begin{align}
    \E^{\Q_n}[L^2_n] \leq \frac{\P(B^\gamma_n)}{ 1 - w }\E[ {Z^2_n\mathbf{I}_{B^\gamma_n}}]
    +
    \frac{1}{w}\E[ {Z^2_n\mathbf{I}_{(B^\gamma_n)^c}}].
    \label{proof, decomp Ln, proposition: design of Zn}
\end{align}
Meanwhile, Lemma~\ref{lemma: LD, events A n} implies that
\begin{align}
    0 < \liminf_{n \rightarrow \infty}\frac{\P(A_n)}{ (n\nu[n,\infty))^{l^*} }
    \leq \limsup_{n \rightarrow \infty}\frac{\P(A_n)}{ (n\nu[n,\infty))^{l^*} } < \infty.
    \label{goal 1, proposition: design of Zn}
\end{align}
Let $Z_{n,1} \delequal Z_n\mathbf{I}_{B^\gamma_n}$ and $Z_{n,2} \delequal Z_n\mathbf{I}_{(B^\gamma_n)^\complement}$.
Furthermore, given $\rho \in (\rho_0,1)$, we claim the existence of some $\bar \gamma = \bar \gamma(\rho) \in (0,b)$ such that for any $\gamma \in (0,\bar \gamma)$,
\begin{align}
    \P(B^\gamma_n) & = \bm{O}\big((n\nu[n,\infty))^{l^*}\big),
    \label{goal 2, proposition: design of Zn}
    \\ 
    \E[Z_{n,1}^2] & = \bm{O}\big( (n\nu[n,\infty))^{l^*}\big),
     \label{goal 3, proposition: design of Zn}
     \\
     \E[Z_{n,2}^2] & = \bm{o}\big( (n\nu[n,\infty))^{2l^*}\big),
     \label{goal 4, proposition: design of Zn}
\end{align}
as $n \to \infty$.
Then, using \eqref{goal 2, proposition: design of Zn} and \eqref{goal 3, proposition: design of Zn} we get 
$
\P(B^\gamma_n)\E[ {Z^2_n\mathbf{I}_{B^\gamma_n}}]
=
\bm{O}\big( (n\nu[n,\infty))^{2l^*}\big) = \bm{O}\big(\P^2(A_n)\big).
$
The last equality follows from \eqref{goal 1, proposition: design of Zn}.
Similarly, from \eqref{goal 1, proposition: design of Zn} and \eqref{goal 4, proposition: design of Zn}
we get
$
\E[ {Z^2_n\mathbf{I}_{(B^\gamma_n)^c}}] = \bm{o}\big( (n\nu[n,\infty))^{2l^*}\big) = \bm{o}\big(\P^2(A_n)\big).
$
Therefore, in \eqref{proof, decomp Ln, proposition: design of Zn} we have $\E^{\Q_n}[L^2_n] = \bm{O}\big(\P^2(A_n)\big)$,
thus establishing the strong efficiency.
Now, it remains to prove claims \eqref{goal 2, proposition: design of Zn}, \eqref{goal 3, proposition: design of Zn}, and \eqref{goal 4, proposition: design of Zn}.



\medskip
\noindent\textbf{Proof of Claim }\eqref{goal 2, proposition: design of Zn}.

We show that the claim holds for all $\gamma \in (0,b)$.
For any $c > 0$ and $k \in \mathbb{N}$, note that
\begin{align}
    \P\big(\text{Poisson}(c) \geq k\big) & = \sum_{j \geq k}\exp(-c)\frac{c^j}{j!}
    = 
    c^k \sum_{j \geq k}\exp(-c)\frac{c^{j - k}}{j!}
    \leq 
    c^k \sum_{j \geq k}\exp(-c)\frac{c^{j - k}}{(j - k)!} = c^k.
    \label{proof, bound poisson dist tail, proposition: design of Zn}
\end{align}
Recall that $B^\gamma_n = \{\bar X_n \in B^\gamma\}$ and $B^\gamma \delequal \{ \xi \in \mathbb{D}: \#\{ t \in [0,1]: \xi(t) - \xi(t-) \geq \gamma \} \geq l^* \}.$
Therefore,
\begin{align*}
    \P(B^\gamma_n) &= \P\big( \#\{ t \in [0,n]:\ X(t) - X(t-) \geq n\gamma \} \geq l^* \big)
    \qquad \text{due to }\bar X_n(t) = X(nt)/n
    \\
    & = \sum_{k \geq l^*}\exp\big( - n\nu[n\gamma,\infty)\big) \frac{ \big(n\nu[n\gamma,\infty)\big)^{k}}{k!}
\leq
     \big(n\nu[n\gamma,\infty)\big)^{l^*}
     \qquad \text{due to }\eqref{proof, bound poisson dist tail, proposition: design of Zn}.
\end{align*}
Lastly, the regularly varying nature of $\nu[x,\infty)$ (see Assumption \ref{assumption: heavy tailed process}) implies
$
\lim_{n \rightarrow \infty}\frac{(n\nu[n\gamma,\infty))^{l^*}}{(n\nu[n,\infty))^{l^*}} = 1/\gamma^{\alpha l^*} \in (0,\infty),
$
and hence $\P(B^\gamma_n) = \bm{O}\big((n\nu[n,\infty))^{l^*}\big)$.

\medskip
\noindent\textbf{Proof of Claim }\eqref{goal 3, proposition: design of Zn}.

Again, we prove the claim for all $\gamma \in (0,b)$.
By the definition of $Z_n$ in \eqref{def: estimator Zn},
\begin{align*}
    Z_{n,1} = Z_n\mathbf{I}_{B^\gamma_n} = \sum_{m = 0}^{  \tau}\frac{ \hat Y^m_n\mathbf{I}_{E_n \cap B^\gamma_n } - \hat Y^{m-1}_n\mathbf{I}_{E_n \cap B^\gamma_n } }{ \P(\tau \geq m) }.
\end{align*}
Meanwhile, by the definition of $B^\gamma_n$, we have
$
\mathbf{I}_{B^\gamma_n} = 0
$
on $\{ \mathcal D(\bar J_n) < l^* \}$,
where $\mathcal D(\xi)$ counts the number of discontinuities for any $\xi \in \D$.
By applying Result~\ref{resultDebias}
under the choice of 
$
Y_m = \hat Y^m_n \mathbf{I}_{ E_n \cap B^\gamma_n }
$
and
$
Y = Y^*_n \mathbf{I}_{ E_n \cap B^\gamma_n },
$
we yield
\begin{align}
    \E Z^2_{n,1}
    & \leq 
    \sum_{m \geq 1}
    \frac{ \E\Big[ \big| Y^*_n\mathbf{I}_{E_n \cap B^\gamma_n } - \hat Y^{m-1}_n\mathbf{I}_{E_n \cap B^\gamma_n }  \big|^2 \Big] }{\P(\tau \geq m)}
    \nonumber
    \\
     & \leq \sum_{m \geq 1}\sum_{k \geq l^*}\frac{ \E\Big[ \mathbf{I}\big( Y^*_n \neq \hat Y^{m-1}_n\big)\ \Big|\ \{\mathcal{D}(\bar J_n) = k\}\Big] }{\P(\tau \geq m)}\cdot\P\big(\mathcal{D}(\bar J_n) = k\big)
    \ \ \text{due to }\mathbf{I}_{B^\gamma_n} = 0\text{ on }\{\mathcal{D}(\bar J_n) < l^*\}
    \nonumber
    \\
    & \leq
    \sum_{k \geq l^*}\P\big(\mathcal{D}(\bar J_n) = k\big) \cdot \sum_{m \geq 1}
    \frac{ \P\Big( Y^*_n \neq \hat Y^{m-1}_n\ \Big|\ \{\mathcal{D}(\bar J_n) = k\}\Big) }{\P(\tau \geq m)}
    \nonumber
    \\
    & \leq
    \sum_{k \geq l^*}\P\big(\mathcal{D}(\bar J_n) = k\big) \cdot \bigg[
    \sum_{m = 1}^{\bar m}\frac{1}{\rho^{m-1}}+
    \sum_{m \geq \bar m + 1}
    \frac{C_0 \rho_0^{m-1} \cdot (k+1) }{ \rho^{m-1} }
    \bigg]
    \qquad \text{by condition \eqref{condition 1, proposition: design of Zn}}
    \nonumber
    \\ 
    & \leq \sum_{k \geq l^*}\P\big(\mathcal{D}(\bar J_n) = k\big) \cdot (k + 1)
    \cdot 
    \bigg[
    \underbrace{\sum_{m = 1}^{\bar m}\frac{1}{\rho^{m-1}}+
    \sum_{m \geq \bar m + 1}
    \frac{C_0 \rho_0^{m-1} }{ \rho^{m-1} }}_{ \delequal \widetilde C_{\rho,1} }
    \bigg].
    \label{proof, bound E Z 2 n 1, goal 3, proposition: design of Zn}
\end{align}
In particular, given $\rho \in (\rho_0,1)$, we have $\widetilde C_{\rho,1} < \infty$, and hence
\begin{align*}
   \E Z^2_{n,1} & \leq \widetilde C_{\rho,1} \sum_{k \geq l^*} (k+1) \cdot \P\big(\mathcal{D}(\bar J_n) = k\big)
   \\
   & =  \widetilde C_{\rho,1} \sum_{k \geq l^*} (k+1) \cdot \exp\big( - n\nu[n\gamma,\infty)\big) \frac{ \big(n\nu[n\gamma,\infty)\big)^{k}}{k!}
   \\ 
   & \leq 
  2\widetilde C_{\rho,1} \sum_{k \geq l^*}  k \cdot \exp\big( - n\nu[n\gamma,\infty)\big) \frac{ \big(n\nu[n\gamma,\infty)\big)^{k}}{k!}
   \qquad 
   \text{ due to }l^* \geq 1\ \Longrightarrow\ \frac{k+1}{k}\leq 2\ \forall k \geq l^*
   \\
   & \leq 
   2\widetilde C_{\rho,1} \cdot \big(n\nu[n\gamma,\infty)\big)^{l^*} \sum_{k \geq l^*} \exp\big( - n\nu[n\gamma,\infty)\big) \frac{ \big(n\nu[n\gamma,\infty)\big)^{k - l^*}}{(k-l^*)!}\qquad\text{ due to }l^* \geq 1
   \\
   & =  2\widetilde C_{\rho,1} \cdot \big(n\nu[n\gamma,\infty)\big)^{l^*}.
\end{align*}
Again, the regular varying nature of $\nu[x,\infty)$
allows us to conclude that $ \E Z^2_{n,1} = \bm{O}\big((n\nu[n,\infty))^{l^*}\big)$.

\medskip
\noindent\textbf{Proof of Claim }\eqref{goal 4, proposition: design of Zn}.

Fix some  $\rho \in (\rho_0,1)$ and some $q > 1$ such that $\rho_0^{1/q} < \rho$.
Let $p > 1$ be such that $\frac{1}{p} + \frac{1}{q} = 1$.
By Assumption \ref{assumption: choice of a,b},
we can pick some $\Delta_0 > 0$ small enough 
such that $a - \Delta_0 > (l^*_1)b$.
This allows us to pick $\bar \gamma \in (0,b)$ small enough such that
$(\hat J + l^* - 1)/p > 2l^*$ where
\begin{align}
    \hat J \delequal \frac{a - \Delta_0 - (l^* - 1)b}{\bar \gamma}.
    \label{proof, def: hat J, goal 4, proposition: design of Zn}
\end{align}

We prove the claim for all $\gamma \in (0,\bar \gamma)$.
Specifically, given any $\gamma \in (0,\bar \gamma)$,
one can pick $\Delta \in (0,\Delta_0)$ such that 
$[a - \Delta - (l^* - 1)b]/\gamma \notin \mathbb Z$.
Due to our choice of $\gamma$ and $\Delta$, it follows from \eqref{proof, def: hat J, goal 4, proposition: design of Zn} that
$
(J_\gamma + l^* - 1)/p > 2l^*
$
where
\begin{align*}
    J_\gamma  \delequal \ceil{\frac{a - \Delta - (l^* - 1)b}{\gamma}}.
\end{align*}


Let
$
A^\Delta = \{\xi \in \mathbb{D}: \sup_{t \in [0,1]}\xi(t)\geq a - \Delta \}
$
and $A^\Delta_n = \{\bar X_n \in A^\Delta\}$.
Also, note that
\begin{align*}
    Z_{n,2} & = {Z_n\mathbf{I}_{(B^\gamma_n)^c}}
  = \underbrace{Z_n \mathbf{I}_{ A^\Delta_n \cap (B^\gamma_n)^c }}_{ \delequal Z_{n,3} } + \underbrace{Z_n \mathbf{I}_{ (A^\Delta_n)^c \cap (B^\gamma_n)^c }}_{\delequal Z_{n,4}}.
\end{align*}
Specifically, $Z_{n,3} = \sum_{m = 0}^{\tau}\frac{ \hat Y^m_n\mathbf{I}_{ A^\Delta_n \cap E_n \cap (B^\gamma_n)^c } - \hat Y^{m-1}_n\mathbf{I}_{A^\Delta_n \cap E_n \cap (B^\gamma_n)^c } }{ \P(\tau \geq m) }$.
Analogous to the calculations in \eqref{proof, bound E Z 2 n 1, goal 3, proposition: design of Zn},
by applying Result~\ref{resultDebias}
under the choice of 
$
Y_m = \hat Y^m_n \mathbf{I}_{ A^\Delta_n \cap E_n \cap (B^\gamma_n)^\complement }
$
and
$
Y = Y^*_n \mathbf{I}_{ A^\Delta_n \cap E_n \cap (B^\gamma_n)^\complement },
$
we yield
\begin{align*}
    \E Z_{n,3}^2
    & \leq 
    \sum_{m \geq 1}
    \frac{ \E\Big[ \big| Y^*_n - \hat Y^{m-1}_n  \big|^2\mathbf{I}_{A^\Delta_n \cap E_n \cap (B^\gamma_n)^c } \Big] }{\P(\tau \geq m)}
    \\
    & = 
    \sum_{m \geq 1}
    \frac{ \E\Big[ \mathbf{I}\big(Y^*_n \neq \hat Y^{m-1}_n\big) \cdot \mathbf{I}_{A^\Delta_n \cap E_n \cap (B^\gamma_n)^c} \Big] }{\P(\tau \geq m)}
    \qquad \text{because }\hat Y^m_n\text{ and }Y^*_n\text{ only take values in }\{0,1\}
    \\
    & \leq 
    \sum_{m \geq 1}
    \frac{ 
    \Big(\P\big(Y^*_n \neq \hat Y^{m-1}_n\big)\Big)^{1/q} \cdot \Big(\P\big(A^\Delta_n \cap E_n \cap (B^\gamma_n)^c\big)\Big)^{1/p}
    }{\P(\tau \geq m)}
    \qquad \text{ by Hölder's inequality}.
\end{align*}
Applying Lemma~\ref{lemma: LD, event A Delta},
we get
$
\big(\P(A^\Delta_n \cap E_n \cap (B^\gamma_n)^\complement)\big)^{1/p} = \bm{o}\big( (n\nu[n,\infty))^{2l^*} \big).
$
On the other hand, it has been shown in \eqref{proof: P hat Y m n neq Y * n, proposition: design of Zn} that for any $n \geq 1$ and $m \geq \bar m$,
we have 
$
\P(Y^*_n \neq \hat Y^{m}_n) \leq C_0 C_\gamma \rho^m_0
$
where 
$
C_\gamma \delequal \sup_{n \geq 1}n\nu[n\gamma,\infty) + 1 < \infty.
$
In summary,
\begin{align}
     \E Z_{n,3}^2 & \leq  
    \bm{o}\big( (n\nu[n,\infty))^{2l^*} \big) \cdot 
     \bigg[ 
     \underbrace{\sum_{m = 1}^{\bar m}
    \frac{ 
    1
    }{\rho^{m-1}}
    +
     \sum_{m \geq \bar m + 1}
    \frac{ 
    (C_0C_\gamma)^{1/q} \cdot (\rho_0^{1/q})^{m-1}
    }{\rho^{m-1}} }_{ \delequal \widetilde C_{\rho,2} }
    \bigg].
    \label{proof, bound Z n 3, proposition: design of Zn}
\end{align}
Note that $\widetilde C_{\rho,2} < \infty$ due to our choice of $\rho_0^{1/q} < \rho$.

Similarly, to bound the second-order moment of  $Z_{n,4} = \sum_{m = 0}^{\tau}\frac{ \hat Y^m_n\mathbf{I}_{ (A^\Delta_n)^c \cap E_n \cap (B^\gamma_n)^c } - \hat Y^{m-1}_n\mathbf{I}_{ (A^\Delta_n)^c \cap E_n \cap (B^\gamma_n)^c } }{ \P(\tau \geq m) }$,
we apply Result~\ref{resultDebias} again and get
\begin{align}
    \E Z^2_{n,4} 
    & \leq 
    \sum_{m \geq 1}
    \frac{ \E\Big[ \big| Y^*_n - \hat Y^{m-1}_n  \big|^2\mathbf{I}_{(A^\Delta_n)^c \cap E_n \cap (B^\gamma_n)^c } \Big] }{\P(\tau \geq m)}
    \nonumber
    \\
    & = 
     \sum_{m \geq 1}
    \frac{ \E\Big[ \mathbf{I}\big( Y^*_n \neq \hat Y^{m-1}_n \big)\cdot\mathbf{I}_{(A^\Delta_n)^c \cap E_n \cap (B^\gamma_n)^c } \Big] }{\P(\tau \geq m)}
    \qquad \text{because }\hat Y^m_n\text{ and }Y^*_n\text{ only take values in }\{0,1\}
    \nonumber
    \\
    & \leq 
     \sum_{m \geq 1}
    \frac{ \P\Big( \big\{Y^*_n \neq \hat Y^{m-1}_n,\ \bar X_n \notin A^\Delta\big\} \cap (B^\gamma_n)^c \Big)}{\P(\tau \geq m)}
    \qquad \text{ due to }A^\Delta_n = \{\bar X_n \in A^\Delta\}
    \nonumber
    \\
    & = 
     \sum_{m \geq 1}
    \frac{ \P\Big( \big\{Y^*_n \neq \hat Y^{m-1}_n,\ \bar X_n \notin A^\Delta\big\} \cap \{ \mathcal{D}(\bar J_n) < l^* \} \Big)}{\P(\tau \geq m)}
    \qquad 
    \text{due to }B^\gamma_n =\{ \mathcal{D}(\bar J_n) \geq l^*\}
    \nonumber
    \\
    & = 
     \sum_{m \geq 1}\sum_{k = 0}^{l^* - 1}
     \frac{ \P\big( Y^*_n \neq \hat Y^{m-1}_n,\ \bar X_n \notin A^\Delta\ \big|\ \{\mathcal{D}(\bar J_n) = k\} \big)}{\P(\tau \geq m)}
    \cdot \P\big(\mathcal{D}(\bar J_n) = k\big)
    \nonumber
    \\
    & \leq 
     \sum_{m \geq 1}\sum_{k = 0 }^{l^* - 1}
     \frac{ C_0 \rho^{m-1}_0 }{ \Delta^2 n^\mu \cdot \rho^{m-1} }
    \qquad \text{ due to \eqref{condition 2, proposition: design of Zn}}
    \nonumber
    \\
    & = l^* \sum_{m \geq 1}
     \frac{ C_0 \rho^{m-1}_0 }{ \Delta^2 n^\mu \cdot \rho^{m-1} }
     =
     \frac{C_0 l^*}{\Delta\cdot (1 - \frac{\rho_0}{\rho})} \cdot \frac{1}{n^\mu}
     =
     \bm{o}\Big(\big(n\nu[n,\infty)\big)^{2l^*}\Big).
     \label{proof, bound Z n 4, proposition: design of Zn}
\end{align}
The last equality follows from the condition $\mu > 2l^*(\alpha - 1)$ prescribed in Proposition~\ref{proposition: design of Zn}
and the fact that  $n\nu[n,\infty) \in \RV_{-(\alpha - 1)}(n)$ as $n \to \infty$.
Combining \eqref{proof, bound Z n 3, proposition: design of Zn} and \eqref{proof, bound Z n 4, proposition: design of Zn}
with the preliminary bound $(x+y)^2 \leq 2x^2 + 2y^2$, we yield
$
\E Z^2_{n,2} \leq 2\E Z^2_{n,3} + 2\E Z^2_{n,4} =  \bm{o}\big((n\nu[n,\infty))^{2l^*}\big)
$
and conclude the proof of \eqref{goal 4, proposition: design of Zn}.
\end{proof}

\subsection{Proof of Theorems~\ref{theorem: strong efficiency without ARA} and \ref{theorem: strong efficiency}}
\label{subsec: proof of propositions cond 1 and 2}
We stress again that Theorem~\ref{theorem: strong efficiency without ARA} follows directly from Theorem~\ref{theorem: strong efficiency} with $\kappa = 0$ (i.e., by disabling ARA from Algorithm~\ref{algoIS}).
We devote the remainder of this section to proving Theorem~\ref{theorem: strong efficiency}.

Throughout Section~\ref{subsec: proof of propositions cond 1 and 2},
we fix the following constants and parameters.
First, let $\beta \in [0,2)$ be the Blumenthal-Getoor index of $X(t)$ and $\alpha > 1$ be the regularly varying index of $\nu[x,\infty)$; see Assumption~\ref{assumption: heavy tailed process}.
Fix some
\begin{align}
    \beta_+ \in (\beta, 2),\qquad \mu > 2l^*(\alpha - 1).
    \label{proofFixBetaPlus}
\end{align}
This allows us to pick $d,r$ large enough such that
\begin{align}
    r(2 - \beta_+) > \max\{2, \mu - 1\},
    \qquad 
    d > \max\{2, 2\mu - 1\}\label{proof, choose d and r, proposition: hat Y m n condition 1}
\end{align}
for $d$ in \eqref{def hat M i m} and $r$ in \eqref{def: kappa n m}.
Let $\lambda > 0$ be the constant in Assumption~\ref{assumption: holder continuity strengthened on X < z t}.
Choose 
\begin{align}
    \alpha_3 \in (0,\frac{1}{\lambda}),\qquad \alpha_4 \in (0, \frac{1}{2\lambda}). \label{proofChooseAlpha_34}
\end{align}
Next, fix
\begin{align}
    \alpha_2 \in (0, \frac{\alpha_3}{2}\wedge 1). \label{proofChooseAlpha_2}
\end{align}
Based on the chosen value of $\alpha_2$, fix
\begin{align}
    \alpha_1 \in (0,\frac{\alpha_2}{\lambda}). \label{proofChooseAlpha_1}
\end{align}
Pick
\begin{align}
    \delta 
    \in (1/\sqrt{2},1). \label{proofChooseDelta}
\end{align}
Since we require $\alpha_2$ to be strictly less than $1$, there is some integer $\Bar{m}$ such that
\begin{align}
    \delta^{m\alpha_2} - \delta^{m} \geq \frac{\delta^{m\alpha_2}}{2}\text{ and }\delta^{m\alpha_2} < a
    \qquad \forall m \geq \bar m
    \label{proofChooseMbar}
\end{align}
where $a>0$ is the parameter in set $A$; see Assumption~\ref{assumption: choice of a,b}.
Based on the values of $\delta$ and $\beta_+$, it holds for all $\kappa \in [0,1)$ small enough that
\begin{align}
    \kappa^{2-\beta_+} < \frac{1}{2} < \delta^2 \label{proofChooseKappa}
\end{align}
Then, based on all previous choices, it holds for all $\rho_1 \in (0,1)$ close enough to 1 such that
\begin{align}
    \delta^{\alpha_1} & < \rho_1, \label{proofChooseRhoTimeBound} \\
    \frac{\kappa^{2 - \beta_+}}{\delta^2} 
    & < \rho_1 \label{proofChooseRhoByKappa} \\
    \frac{1}{\sqrt{2}\delta}  & < \rho_1 \label{proofChooseRhoByDelta} \\
    \delta^{\alpha_2 - \lambda\alpha_1} & < \rho_1 \label{proofChooseRhoByAlpha_12} \\
    \delta^{1 - \lambda\alpha_3} & < \rho_1 \label{proofChooseRhoSBALongStick} \\
    \delta^{-\alpha_2 + \frac{\alpha_3}{2}} & < \rho_1, \label{proofChooseRhoSBAShortStick} \\
    (1/\sqrt{2})\vee \kappa^{2 - \beta_+} & < \rho_1. \label{proofChooseRhoConstantBound}
\end{align}
Lastly, pick $\rho_0 \in (\rho_1,1)$. By picking a larger $\bar m$ if necessary, we can ensure that
\begin{align}
    m^2 \rho^m_1 \leq \rho^m_0\qquad \forall m \geq \bar m.
    \label{proof, choose bar m and rho 0}
\end{align}

Next, we make a few observations.
Given some non-negative integer $k$,
let
\begin{align}
    \zeta_k(t) = \sum_{i = 1}^k z_i \mathbf{I}_{[u_i,n]}(t)
    \label{proof: def zeta k}
\end{align}
where $0 < u_1 < u_2 < \ldots < u_k < n$
are the order statistics of $k$ iid samples of Unif$(0,n)$,
and $z_i$'s are iid samples from $\nu(\cdot \cap [n\gamma,\infty))/\nu[n\gamma,\infty)$.
We adopt the convention that $u_0 \equiv 0$ and $u_{k+1} \equiv 1$.
Note that when $k = 0$, we set $\zeta_0(t) \equiv 0$ as the zero function, and set $I_1 = [0,n]$, $u_0 = 0$, and $u_1 = n$.

For $Y^*_n(\cdot)$ defined in \eqref{def, random function Y * n}
and $\hat Y^m_n(\cdot)$ defined in \eqref{def hat Y m n}, note that
\begin{align}
    Y^*_n(\zeta_k) = \max_{i \in [k+1]}\mathbf{I}\{W^{(i),*}_n(\zeta_k) \geq na\},
    \qquad 
    \hat Y^m_n(\zeta_k) = \max_{i \in [k+1]}\mathbf{I}\{\hat W^{(i),m}_n(\zeta_k) \geq na\}
    \label{proof, representation, Y * n and hat Y m n}
\end{align}
 where
\begin{align}
    W^{(i),*}_n(\zeta_k) & \delequal \sum_{q = 1}^{i - 1}\sum_{j \geq 0}\xi^{(q)}_j + \sum_{q = 1}^{i-1}z_q + \sum_{j \geq 1}(\xi^{(i)}_j)^+,
    \label{proof,def: W i * n}
    \\ 
    \hat W^{(i),m}_n(\zeta_k)  & \delequal \sum_{q = 1}^{i - 1}\sum_{j \geq 0}\xi^{(q),m}_j + \sum_{q = 1}^{i-1}z_q + \sum_{j = 1}^{m + \ceil{ \log_2(n^d) }}(\xi^{(i),m}_j)^+.
     \label{proof,def: hat W i m n}
\end{align}
See \eqref{defStickLength1}--\eqref{def xi i j, algo without ARA} and \eqref{def xi i m j, ARA plus SBA}
for the definitions $\xi^{(i)}_j$'s and $\xi^{(i),m}_j$'s, respectively.
Also, define
\begin{align}
        \widetilde  W^{(i),m}_n(\zeta_k) & \delequal \sum_{q = 1}^{i - 1}\sum_{j \geq 0}\xi^{(q)}_j + \sum_{q = 1}^{i-1}z_q + \sum_{j = 1}^{m + \ceil{ \log_2(n^d) }}(\xi^{(i)}_j)^+.
        \label{proof, def: tilde W i m n}
\end{align}
As intermediate steps for the proof of
Theorem~\ref{theorem: strong efficiency},
we present the following two results.
Proposition~\ref{proposition, intermediate 1, strong efficiency ARA} states that, using 
$
\widetilde  W^{(i),m}_n(\zeta_k)
$
as an anchor,
we see that $W^{(i),*}_n(\zeta_k)$ and $\hat W^{(i),m}_n(\zeta_k)$ would stay close enough with high probability, especially for large $m$.
Proposition~\ref{proposition, intermediate 2, strong efficiency ARA} then shows that it is unlikely for the law of 
$
\widetilde  W^{(i),m}_n(\zeta_k)
$
to concentrate around any $y \in \R$.

\begin{proposition}
\label{proposition, intermediate 1, strong efficiency ARA}
\linksinthm{proposition, intermediate 1, strong efficiency ARA}
    There exists some constant $C_1 \in (0,\infty)$ such that the inequality
    \begin{align*}
        \P\bigg( \Big| W^{(i),*}_n(\zeta_k) - \widetilde  W^{(i),m}_n(\zeta_k) \Big| \vee 
        \Big|  \hat W^{(i),m}_n(\zeta_k) - \widetilde  W^{(i),m}_n(\zeta_k) \Big|
        > 
        x
        \bigg) \leq \frac{C_1 \kappa^{m(2-\beta_+)}}{x^2 \cdot n^{r (2- \beta_+) - 1}} + \frac{C_1}{x}\sqrt{\frac{1}{n^{d-1}\cdot 2^m}}
    \end{align*}
    holds for any $k \in \mathbb N$, $i \in [k+1]$, $n \geq 1$, $m \in \mathbb N$, and $x > 0$.
\end{proposition}

\begin{proposition}
\label{proposition, intermediate 2, strong efficiency ARA}
\linksinthm{proposition, intermediate 2, strong efficiency ARA}
    There exists some constant $C_2 \in (0,\infty)$ such that the inequality
    \begin{align*}
        \P\bigg(
            \widetilde W^{(i),m}_n(\zeta_k) \in \bigg[ y - \frac{\delta^m}{\sqrt{n}},  y + \frac{\delta^m}{\sqrt{n}}  \bigg]
            \text{ for some }i \in [k+1]
        \bigg)
        \leq (k+1) \cdot C_2 \rho^m_0
    \end{align*}
    holds for any $k \in \mathbb N$, $i \in [k+1]$, $n \geq 1$, $m \geq \bar m$, and $y > \delta^{m\alpha_2}$.
\end{proposition}

First, equipped with Propositions \ref{proposition, intermediate 1, strong efficiency ARA} and \ref{proposition, intermediate 2, strong efficiency ARA},
we are able to prove the main results of Section~\ref{subsec: algo, ARA, SBA, and debiasing technique}, i.e., Theorem~\ref{theorem: strong efficiency}.

\begin{proof}[Proof of Theorem~\ref{theorem: strong efficiency}]
    \linksinpf{theorem: strong efficiency}
In light of Proposition~\ref{proposition: design of Zn}, it suffices to verify conditions \eqref{condition 1, proposition: design of Zn} and \eqref{condition 2, proposition: design of Zn}.

\medskip
\noindent\textbf{Verification of \eqref{condition 1, proposition: design of Zn}}.

Conditioning on $\{ \mathcal D(\bar J_n) = k \}$,
 the conditional law of $J_n = \{J_n(t):\ t \in [0,n]\}$
 is the same as the law of the process $\zeta_k$ specified in \eqref{proof: def zeta k}.
This implies
\begin{align*}
    \P\big(Y^*_n(J_n) \neq \hat Y^{ m }_n(J_n)\ \big|\ \mathcal{D}(\bar J_n) = k\big) 
    & =
    \P\big(Y^*_n(\zeta_k) \neq \hat Y^{ m }_n(\zeta_k)\big).
\end{align*}
Next, on event 
\begin{align*}
    &\bigcap_{ i \in [k+1] }
    \Bigg(\bigg\{
    \Big| W^{(i),*}_n(\zeta_k) - \widetilde  W^{(i),m}_n(\zeta_k) \Big| \vee 
        \Big|  \hat W^{(i),m}_n(\zeta_k) - \widetilde  W^{(i),m}_n(\zeta_k) \Big|
        \leq 
        \frac{\delta^m}{\sqrt{n}}
    \bigg\}
    \\ 
    &\qquad\qquad \qquad \qquad\qquad \qquad\qquad \qquad\cap 
    \bigg\{
    \widetilde W^{(i),m}_n(\zeta_k) \notin \bigg[ na - \frac{\delta^m}{\sqrt{n}},  na + \frac{\delta^m}{\sqrt{n}}  \bigg]
    \bigg\}
    \Bigg),
\end{align*}
we must have (for any $i \in [k+1]$)
\begin{align*}
    W^{(i),*}_n(\zeta_k) \vee  \hat W^{(i),m}_n(\zeta_k) < na
    \qquad\text{ or }\qquad
    W^{(i),*}_n(\zeta_k) \wedge \hat W^{(i),m}_n(\zeta_k) > na.
\end{align*}
It then follows from \eqref{proof, representation, Y * n and hat Y m n}
that, on this event, we have $Y^*_n(\zeta_k) = \hat Y^m_n(\zeta_k)$.
Therefore,
\begin{align}
    & \P\big(Y^*_n(\zeta_k) \neq \hat Y^{ m }_n(\zeta_k)\big)
    \label{proof, ineq 1, proposition: hat Y m n condition 1}
    \\
    & \leq 
    \sum_{i \in [k+1]}
    \P\bigg( \Big| W^{(i),*}_n(\zeta_k) - \widetilde  W^{(i),m}_n(\zeta_k) \Big| \vee 
        \Big|  \hat W^{(i),m}_n(\zeta_k) - \widetilde  W^{(i),m}_n(\zeta_k) \Big|
        > 
        \frac{\delta^m}{\sqrt{n}}
        \bigg)
    \nonumber
    \\ 
    & +
    \P\bigg(
            \widetilde W^{(i),m}_n(\zeta_k) \in \bigg[ na - \frac{\delta^m}{\sqrt{n}},  na + \frac{\delta^m}{\sqrt{n}}  \bigg]
            \text{ for some }i \in [k+1]
        \bigg).
    \nonumber
\end{align}
Applying Proposition~\ref{proposition, intermediate 1, strong efficiency ARA}
(with $x = \delta^m/\sqrt{n})$, we get (for any $i \in [k+1]$)
\begin{align}
    & \P\bigg( \Big| W^{(i),*}_n(\zeta_k) - \widetilde  W^{(i),m}_n(\zeta_k) \Big| \vee 
        \Big|  \hat W^{(i),m}_n(\zeta_k) - \widetilde  W^{(i),m}_n(\zeta_k) \Big|
        > 
        \frac{\delta^m}{\sqrt{n}}
        \bigg)
    \nonumber
    \\ 
    & \leq C_1 \cdot \Bigg[
        \frac{ \kappa^{m(2-\beta_+)} \cdot n  }{ \delta^{2m} \cdot n^{r(2 - \beta_+) - 1 }}
            +
        \frac{\sqrt{n}}{ (\sqrt{2}\delta)^m \cdot\sqrt{n^{ d-1  }}  }
    \Bigg]
    \nonumber
    \\ 
    & = 
    C_1 \cdot \Bigg[
    \bigg( \frac{\kappa^{2 - \beta_+}}{\delta^2} \bigg)^m \cdot \frac{1}{ n^{ r(2 - \beta_+) - 2 } }
        + 
    \bigg( \frac{1}{\sqrt{2}\delta} \bigg)^m \cdot  \sqrt{\frac{1}{n^{d-2}}}\ 
    \Bigg]
    \nonumber
    \\ 
    & \leq 
    C_1 \cdot \Bigg[
        \bigg( \frac{\kappa^{2 - \beta_+}}{\delta^2} \bigg)^m +  \bigg( \frac{1}{\sqrt{2}\delta} \bigg)^m
    \Bigg]
    \qquad \text{due to the choices of $d$ and $r$ in \eqref{proof, choose d and r, proposition: hat Y m n condition 1}}
    \nonumber
    \\ 
    & \leq 2C_1\rho^m_0
    \qquad 
    \text{due to the choices in \eqref{proofChooseRhoByKappa} and \eqref{proofChooseRhoByDelta}, and $\rho_0 \in (\rho_1,1)$}.
    \label{proof, ineq 2, proposition: hat Y m n condition 1}
\end{align}
On the other hand, 
due to \eqref{proofChooseMbar}, we have $na - \delta^{m\alpha_2} \geq a  - \delta^{m\alpha_2} > 0$.
for all $n \geq 1$ and $m \geq \bar m$.
This allows us to apply Proposition~\ref{proposition, intermediate 2, strong efficiency ARA} (with $y = na$)
and yield (for any $i \in [k+1]$)
\begin{align}
    \P\bigg(
            \widetilde W^{(i),m}_n(\zeta_k) \in \bigg[ na - \frac{\delta^m}{\sqrt{n}},  na + \frac{\delta^m}{\sqrt{n}}  \bigg]
            \text{ for some }i \in [k+1]
        \bigg)
    \leq (k+1) \cdot C_2 \rho^m_0\qquad \forall m \geq \bar m.
    \label{proof, ineq 3, proposition: hat Y m n condition 1}
\end{align}
Plugging \eqref{proof, ineq 2, proposition: hat Y m n condition 1} and \eqref{proof, ineq 3, proposition: hat Y m n condition 1} into \eqref{proof, ineq 1, proposition: hat Y m n condition 1},
we conclude the proof by setting $C_0 = 2C_1 + C_2$.

\medskip
\noindent\textbf{Verification of \eqref{condition 2, proposition: design of Zn}}.

Fix some $\Delta \in (0,1)$ and $k = 0,1,\ldots,l^*-1$.
Again, conditioning on $\{ \mathcal D(\bar J_n) = k \}$,
 the conditional law of $J_n = \{J_n(t):\ t \in [0,n]\}$
 is the same as the law of the process $\zeta_k$ specified in \eqref{proof: def zeta k}.
 This implies
\begin{align}
    & \P\Big(Y^*_n(J_n) \neq \hat Y^{ m }_n(J_n),\ \bar X_n \notin A^\Delta \ \Big|\ \mathcal{D}(\bar J_n) = k\Big)  
    \nonumber
    \\
     & = 
     \P\Big(Y^*_n(J_n) \neq \hat Y^{ m }_n(J_n),\ \sup_{t \in [0,n]}X(t) < n(a-\Delta) \ \Big|\ \mathcal{D}(\bar J_n) = k\Big)
     \qquad \text{by definition of set $A^\Delta$}
     \nonumber
     \\ 
     & = 
     \P\Big( \max_{i \in [k+1]}\hat W^{(i),m}_n(\zeta_k) \geq na,\ \max_{i \in [k+1]} W^{(i),*}_n(\zeta_k) < n(a-\Delta) \Big)
     \nonumber
     \\ 
     & 
     \leq \sum_{i \in [k+1]}\P\Big( \big|\hat W^{(i),m}_n(\zeta_k) - W^{(i),*}_n(\zeta_k)\big| > n\Delta\Big)
     \nonumber
     \\ 
     & \leq \sum_{i \in [k+1]}
     \P\bigg( \Big| W^{(i),*}_n(\zeta_k) - \widetilde  W^{(i),m}_n(\zeta_k) \Big| \vee 
        \Big|  \hat W^{(i),m}_n(\zeta_k) - \widetilde  W^{(i),m}_n(\zeta_k) \Big|
        > 
        \frac{n\Delta}{2}
        \bigg)
        \nonumber
    \\
    & \leq 
    (k+1) \cdot \Bigg[
    \frac{4C_1}{\Delta^2 n^2} \cdot \frac{\kappa^{ m(2 - \beta_+) }}{ n^{r(2 - \beta_+) - 1 } }
    +
    \frac{2C_1}{\Delta}\cdot \frac{1}{n}\sqrt{\frac{1}{n^{d-1}\cdot 2^m}}\ 
    \Bigg]
    \qquad 
    \text{by Proposition~\ref{proposition, intermediate 1, strong efficiency ARA}}
    \nonumber
    \\ 
    & =
    (k+1)\cdot 
    \Bigg[
    \frac{4C_1}{\Delta^2} \cdot  \frac{\kappa^{ m(2 - \beta_+) }}{ n^{r(2 - \beta_+) + 1 } }
    +
    \frac{2C_1}{\Delta} \cdot \frac{(1/\sqrt{2})^m}{ n^{\frac{d+1}{2}} }
    \Bigg]
    \nonumber
    \\ 
    & \leq 
    \frac{k+1}{n^\mu}\cdot 
    \Bigg[
    \frac{4C_1}{\Delta^2} \cdot \kappa^{ m(2 - \beta_+) }
    +
    \frac{2C_1}{\Delta} \cdot (1/\sqrt{2})^m
    \Bigg]
    \qquad 
    \text{by the choices of $r$ and $d$ in \eqref{proof, choose d and r, proposition: hat Y m n condition 1}}
    \nonumber
    \\
    & \leq 
    \frac{k+1}{n^\mu}\cdot 
    \Bigg[
    \frac{4C_1}{\Delta^2} \cdot \rho_0^m
    +
    \frac{2C_1}{\Delta} \cdot \rho_0^m
    \Bigg]
    \quad 
    \text{due to the choice of $\rho_1$ in \eqref{proofChooseRhoConstantBound} and $\rho_0 \in (\rho_1,1)$}.
    \nonumber
\end{align}
Due to $\Delta \in (0,1)$ (and hence $\frac{1}{\Delta} < \frac{1}{\Delta^2}$) and $k \leq l^* - 1$,
we conclude the proof by setting $C_0 = 6l^*C_1$.
\end{proof}

The rest of this section is devoted to proving Propositions~\ref{proposition, intermediate 1, strong efficiency ARA} and \ref{proposition, intermediate 2, strong efficiency ARA}.
First, we collect a useful result.
\begin{result}[Lemma~1 of \cite{cazares2018geometrically}]\label{result: bound bar sigma}
    Let $\nu$ be the L\'evy measure of a L\'evy process $X$. Let
    $
    I_0^p(\nu) \delequal \int_{(-1,1)}|x|^p \nu(dx).
    $
    Suppose that $\beta < 2$ for the Blumenthal-Getoor index $\beta \delequal \inf\{p >0:\ I^p_0(\nu)<\infty\}$.
    Then
    \begin{align*}
        \int_{(-\kappa,\kappa)}x^2\nu(dx) \leq \kappa^{2 - \beta_+} I^{\beta_+}_0(\nu)\qquad \forall \kappa \in (0,1],\ \beta_+ \in (\beta,2).
    \end{align*}
\end{result}

Next, we prepare two lemmas regarding the expectations of the supremum of $\Xi_n$ (see \eqref{def: process J n Xi n} for the definition)
and the difference between $\Xi_n$ and $\breve \Xi^m_n$ (see \eqref{def breve Xi n m, ARA}).

\begin{lemma}
\label{lemma: algo, bound supremum of truncated X}
\linksinthm{lemma: algo, bound supremum of truncated X}
There exists a constant $C_X < \infty$ (depending only on the law of L\'evy process $X(t)$) such that
    \begin{align*}
    \E\bigg[ \sup_{s \in [0,t]}\Xi_n(t)\bigg]
    \leq C_X(\sqrt{t} + t)\qquad \forall t > 0,\ n \geq 1.
\end{align*}
\end{lemma}

\begin{proof}\linksinpf{lemma: algo, bound supremum of truncated X}
   Recall that the generating triplet of $X$ is $(c_X,\sigma,\nu)$
and for the  Blumenthal-Getoor index $\beta \delequal \inf\{p > 0: \int_{(-1,1)}|x|^p\nu(dx) < \infty\}$ we have $\beta < 2$; see Assumption \ref{assumption: heavy tailed process}.
Fix some $\beta_+ \in (1 \vee \beta,2)$ in this proof. We prove the lemma for
\begin{align*}
    C_X \delequal \max\Big\{|\sigma|\sqrt{\frac{2}{\pi}} + 2\sqrt{I^{\beta_+}_0(\nu)},\ (c_X)^+ + I^1_+(\nu) + 2I_0^{\beta_+}(\nu)\Big\}
\end{align*}
where $(x)^+ = x \vee 0$, $I^1_+(\nu) = \int_{[1,\infty)}x\nu(dx)$, and $I^p_0(\nu) = \int_{(-1,1)}|x|^p\nu(dx).$

Recall that $ \Xi_n$ is a L\'evy process with generating triplet $(c_X,\sigma,\nu|_{ (-\infty,n\gamma) })$.
Let $\nu_n \delequal \nu|_{ (-\infty,n\gamma) }$. 
It follows from Lemma~2 of \cite{cazares2018geometrically} (specifically, by setting $t=T$ in equation (26)) that, for all $t > 0$ and $n \geq 1$,
\begin{align}
    \E \sup_{s \in [0,t]} \Xi_n(t)
    \leq 
    \bigg( |\sigma|\sqrt{\frac{2}{\pi}} + 2\sqrt{I^{\beta_+}_0(\nu_n)} \bigg)\sqrt{t} 
    + \Big( (c_X)^+ + I^1_+(\nu_n) + 2I_0^{\beta_+}(\nu_n) \Big)t.
    \label{proof, ineq, lemma: algo, bound supremum of truncated X}
\end{align}
In particular, note that
$
I^{\beta_+}_0(\nu_n) = \int_{(-1,1)}|x|^p\nu_n(dx) = \int_{(-1,1) \cap (-\infty,n\gamma)}|x|^p \nu(dx) \leq I^{\beta_+}_0(\nu)
$
and
$
I^1_+(\nu_n) = \int_{[1,\infty)}x\nu_n(dx) = \int_{[1,\infty) \cap (-\infty,n\gamma)}x\nu(dx) \leq I^1_+(\nu).
$
Plugging these two bounds into \eqref{proof, ineq, lemma: algo, bound supremum of truncated X}, we conclude the proof.
\end{proof}

\begin{lemma}
\label{lemma: algo, bound term 1}\linksinthm{lemma: algo, bound term 1}
There exists some $C \in (0,\infty)$ (only depending on the choice of $\beta_+ \in (\beta,2)$ in \eqref{proofFixBetaPlus} and the law of L\'evy process $X$) such that
$$
\P\Big( \sup_{t \in [0,n]}\Big| \Xi_n(t) - \breve \Xi^m_n(t) \Big| > x \Big) \leq \frac{C\kappa^{m(2-\beta_+)}}{x^2 n^{r(2 - \beta_+) - 1} }\qquad \forall x > 0,\ n \geq 1,\ m \in \mathbb N
$$
where $r$ is the parameter in the truncation threshold $\kappa_{n,m} = \kappa^m/n^r$ (see \eqref{def: kappa n m}).
\end{lemma}

\begin{proof}\linksinpf{lemma: algo, bound term 1}
From the definitions of $\Xi_n$ and $\breve \Xi^m_n$ in \eqref{decomp, Xi n} and \eqref{def breve Xi n m, ARA}, respectively, we have
\begin{align*}
    \Xi_n(t) - \breve \Xi^m_n(t) \distequal X^{ (-\kappa_{n,m},\kappa_{n,m}) }(t) - \bar\sigma(\kappa_{n,m})B(t)
\end{align*}
where $\notationdef{notation-X-jump-bounded-by-c}{X^{(-c,c)}}$ is the L\'evy process with generating triplet $(0,0,\nu|_{(-c,c)})$,
$
\kappa_{n,m} = \kappa^m/n^r,
$
and $B$ is a standard Brownian motion independent of $X^{ (-\kappa_{n,m},\kappa_{n,m}) }$.
In particular, $X^{ (-\kappa_{n,m},\kappa_{n,m}) }$ is a martingale with variance $var\big[X^{ (-\kappa_{n,m},\kappa_{n,m}) }(1)\big] = \bar\sigma^2(\kappa_{n,m})$; see \eqref{def bar sigma} for the definition of $\bar\sigma^2(\cdot)$.
Therefore, 
\begin{align*}
    & \P\Big( \sup_{t \in [0,n]}\Big| \Xi_n(t) - \breve \Xi^m_n(t) \Big| > x \Big)
    \\
    & \leq 
    \frac{1}{x^2}\E\Big| X^{ (-\kappa_{n,m},\kappa_{n,m}) }(n) - \bar\sigma(\kappa_{n,m})B(n)\Big|^2
    \qquad \text{using Doob's inequality}
    \\
    & = \frac{2n}{x^2}\bar\sigma^2(\kappa_{n,m})
    \qquad \text{due to the independence between $X^{ (-\kappa_{n,m},\kappa_{n,m}) }$ and $B$}
    \\
    & \leq 
    \frac{2n}{x^2} \cdot \kappa^{2 - \beta_+}_{ n,m }I_0^{\beta_+}(\nu)
    \qquad \text{ using Result \ref{result: bound bar sigma}}
    \\
    & = \frac{2I_0^{\beta_+}(\nu)}{x^2} \cdot \frac{ n\kappa^{m(2-\beta_+)} }{ n^{r(2-\beta_+)} } = \frac{2I_0^{\beta_+}(\nu)}{x^2} \cdot \frac{\kappa^{m(2-\beta_+)}}{ n^{r(2-\beta_+) - 1} }
    \qquad \text{ due to }\kappa_{n,m} = \kappa^m/n^r.
\end{align*}
To conclude the proof, we set $C = 2I_0^{\beta_+}(\nu) = 2\int_{(-1,1)}\int |x|^{\beta_+}\nu(dx)$.
\end{proof}

To facilitate the presentation of the next few lemmas,
we consider a slightly more general version of the stick-breaking procedure
described in \eqref{defStickLength1}--\eqref{def xi i m j, ARA plus SBA},
to allow for arbitrary stick length.
Specifically, for any $l > 0$,
let
\begin{align}
    l_1(l) = V_1 \cdot l,\qquad 
    \notationdef{notation-stick-length-l-j-l}{l_j(l)} = V_j \cdot \big(l - l_1(l) - l_2(l) - \cdots - l_{j-1}(l)\big)\quad \forall j \geq 2,
    \label{def general xi n m l, 2}
\end{align}
where $V_j$'s are iid copies of Unif$(0,1)$.
Independent of $V_j$'s, for any $n$ and $m$,
let $\Xi_n$ and $\breve \Xi^m_n$ be L\'evy processes with joint law specified in \eqref{decomp, Xi n} and \eqref{def breve Xi n m, ARA}, respectively.
Conditioning on the values of $l_j(l)$, define $\notationdef{notation-increments-on-sticks-xi-n-j-l}{\xi^{[n]}_j(l),\xi^{[n],m}_j(l)}$ using (for all $j \geq 1$)
\begin{align}
    \Big(\xi^{[n]}_j(l),\xi^{[n],0}_j(l),\xi^{[n],1}_j(l),\xi^{[n],2}_j(l),\ldots \Big)
    =
    \Big( \Xi_n\big(l_j(l)\big),\ \breve \Xi^0_n\big(l_j(l)\big),\ \breve \Xi^1_n\big(l_j(l)\big),\ \breve \Xi^2_n\big(l_j(l)\big),\ \ldots \Big).
    \label{def general xi n m l, 3}
\end{align}

\begin{lemma}
\label{lemma: algo, bound term 2}
\linksinthm{lemma: algo, bound term 2}
There exists some $C \in (0,\infty)$ (only depending on the choice of $\beta_+ \in (\beta,2)$ in \eqref{proofFixBetaPlus} and the law of L\'evy process $X$) such that, for all $m \in \mathbb N$ and $n \geq 1$,
$$
\P\bigg(\bigg| \sum_{j = 1}^{ m + \ceil{ \log_2(n^d) } } \big(\xi^{[n]}_j(l)\big)^+ - \sum_{j = 1}^{m + \ceil{ \log_2(n^d) }} \big(\xi^{[n],m}_j(l)\big)^+ \bigg| > y \bigg) \leq \frac{C\kappa^{m(2-\beta_+)}}{y^2 n^{r(2 - \beta_+) - 1} }
\qquad \forall y > 0,\ l \in [0,n]
$$
where $r$ is the parameter in the truncation threshold $\kappa_{n,m} = \kappa^m/n^r$ (see \eqref{def: kappa n m})
and $(x)^+ = x \vee 0$.
\end{lemma}

\begin{proof}
\linksinpf{lemma: algo, bound term 2}
For notational simplicity, set $k(n) = \ceil{ \log_2(n^d) }$.
Due to $|(x)^+ - (y)^+| \leq |x - y|$,
\begin{align}
    & \P \Big( \Big| \sum_{j = 1}^{ m + k(n) } \big(\xi^{[n]}_j(l)\big)^+ - \sum_{j = 1}^{m + k(n)} \big(\xi^{[n],m}_j(l)\big)^+ \Big| > y \Big)
    \nonumber \\
    & \leq 
 \P \Big( \sum_{j = 1}^{ m + k(n)} \Big|\big(\xi^{[n]}_j(l)\big)^+ - \big(\xi^{[n],m}_j(l)\big)^+ \Big| > y \Big)
    \leq   \P \Big(  \sum_{j = 1}^{ m + k(n) } \Big|\underbrace{\xi^{[n]}_j(l) - \xi^{[n],m}_j(l)}_{\delequal q_j} \Big| > y \Big).
    \label{proofIntermediateARAerrorCurrent}
\end{align}
Furthermore, we claim the existence of some constant $\tilde C \in (0,\infty)$
such that (for any $y,d > 0$, $l \in [0.n]$, and any $n \geq 1,\ m \in \mathbb N$)
\begin{align}
    \P(\sum_{j= 1}^{m + k(n)}|q_j| > y) \leq \tilde C \cdot \frac{n \bar \sigma^2(\kappa_{n,m}) }{y^2}.
    \label{proof, goal, lemma: algo, bound term 2}
\end{align}
Then 
using Result~\ref{result: bound bar sigma}, we yield
 \begin{align*}
     n\bar\sigma^2(\kappa_{n,m}) \leq n \cdot \kappa^{2 - \beta_+}_{ n,m }I_0^{\beta_+}(\nu)
     =
      \frac{\kappa^{m(2-\beta_+)}}{n^{r(2-\beta_+)-1}}
      \cdot I_0^{\beta_+}(\nu)
 \end{align*}
where $I_0^{\beta_+}(\nu) = \int_{(-1,1)}\int |x|^{\beta_+}\nu(dx)$.
Setting $C = \tilde C I_0^{\beta_+}(\nu)$, we conclude the proof.

Now, it only remains to prove claim \eqref{proof, goal, lemma: algo, bound term 2}.
Let $\chi = 2^{1/4}$.
Note that
\begin{align*}
1 
=
(\chi - 1)\sum_{j \geq 1}\frac{1}{\chi^j}
\geq (\chi - 1)\Big( \frac{1}{\chi} + \frac{1}{\chi^2} + \cdots + \frac{1}{\chi^{m +k(n) }}  \Big).
\end{align*}{}
As a result,
\begin{align}
    \P \Big( \sum_{j = 1}^{ k(n) + m } |q_j| > y   \Big) 
    \leq & \P \Big( \sum_{j = 1}^{ k(n) + m } |q_j| > y(\chi - 1)\sum_{j = 1}^{ k(n) + m }\frac{1}{\chi^j}   \Big)
    \leq 
    \sum_{j = 1}^{ k(n) + m } \P \Big( |q_j| > y\cdot \frac{\chi - 1}{\chi^j} \Big) \label{proofSumOfBj}
\end{align}
Next, we bound each $\P( |q_j| > y\frac{\chi - 1}{\chi^j})$. 
Conditioning on $l_j(l) = t$ (for any $t \in [0,l]$), we get
\begin{align*}
    \P \bigg( |q_j| > y\frac{\chi - 1}{\chi^j}\ \bigg|\ l_j(l) = t \bigg)
    & = 
    \P\bigg(\Big| \Xi_n(t) - \breve \Xi^m_n(t) \Big| > y\frac{\chi - 1}{\chi^j}\bigg)
    \qquad \text{ due to \eqref{def general xi n m l, 3}}
    \\
    & \leq 
    \frac{\chi^{2j}}{y^2 (\chi - 1)^2 }\E\Big|X^{ (-\kappa_{n,m},\kappa_{n,m}) }(t) - \bar\sigma(\kappa_{n,m})B(t)\Big|^2
    \\ 
    & = \frac{\chi^{2j}}{y^2 (\chi - 1)^2 } \cdot 2\bar\sigma^2(\kappa_{n,m})t
    \\
\Longrightarrow 
\P \bigg( |q_j| > y\frac{\chi - 1}{\chi^j}\bigg) 
    & \leq 
    \frac{\chi^{2j}}{y^2 (\chi - 1)^2 } \cdot 2\bar\sigma^2(\kappa_{n,m}) \cdot \E[l_j(l)]
    \\
    & = 
    \frac{\sqrt{2^j}}{y^2 (2^{1/4} - 1)^2 } \cdot 2\bar\sigma^2(\kappa_{m,n}) \cdot \E[l_j(l)]
    \qquad\text{ due to }\chi = 2^{1/4}
    \\
    & = \frac{\sqrt{2^j}}{y^2 (2^{1/4} - 1)^2 } \cdot 2\bar\sigma^2(\kappa_{m,n}) \cdot \frac{l}{2^j}
    \qquad \text{by definition of $l_j(l)$ in \eqref{def general xi n m l, 2}}
    \\
    & \leq \frac{2}{ (2^{1/4} - 1)^2 \sqrt{2^j} } \cdot \frac{n\bar\sigma^2(\kappa_{m,n})}{y^2}
    \qquad \text{ due to }l \leq n.
\end{align*}
Therefore, in \eqref{proofSumOfBj}, we get
\begin{align*}
\P \Big( \sum_{j = 1}^{ k(n) + m } |q_j| > y   \Big) 
& \leq 
\frac{n\bar\sigma^2(\kappa_{m,n})}{y^2}\sum_{j \geq 1}\frac{2}{(2^{1/4} - 1)^2 \sqrt{2^j} }
=
\frac{n\bar\sigma^2(\kappa_{m,n})}{y^2}
\cdot \underbrace{\frac{2\sqrt{2}}{(2^{1/4} - 1)^2(\sqrt{2}-1) }}_{ \delequal \tilde C },
\end{align*}
thus establishing claim \eqref{proof, goal, lemma: algo, bound term 2}.
\end{proof}

\begin{lemma}
\label{lemma: algo, bound term 3}
\linksinthm{lemma: algo, bound term 3}
    Let $n \in \mathbb Z_+$ and $l \in [0,n]$.
    Let $C_X < \infty$ be the constant characterized in Lemma \ref{lemma: algo, bound supremum of truncated X} that only  depends on the law of L\'evy process $X$.
    The inequality
    \begin{align*}
        \P\Big( \sum_{j > m + \ceil{ \log_2(n^d) }} \big(\xi^{[n]}_j(l)\big)^+ > x \Big) \leq \frac{2C_X}{x}\sqrt{\frac{1}{n^{d-1} \cdot 2^m}}
    \end{align*}
    holds for all $x >0$, $n \geq 1$, and $m \geq 0$, where $(x)^+ = x \vee 0$.
\end{lemma}

\begin{proof}
\linksinpf{lemma: algo, bound term 3}
For this proof, we adopt the notation $\breve l_k(l) \delequal l - l_1(l) - l_2(l) - \ldots - l_k(l)$
for the remaining stick length after the first $k$ sticks.
Conditioning on $\breve l_{m + \ceil{\log_2(n^d)}}(l) = t$, 
\begin{align*}
    \P\Big( \sum_{j > m + \ceil{ \log_2(n^d) }} \big(\xi^{[n]}_j(l)\big)^+ > x \ \bigg|\ \breve l_{m + \ceil{\log_2(n^d)}}(l) = t\Big) 
    & = \P\Big(\sup_{s \in [0,t]}\Xi_n(s) > x\Big)
    \qquad 
    \text{ by Result~\ref{result: concave majorant of Levy}}
    \\ 
    & \leq 
    \frac{C_X}{x}(\sqrt{t}+t)
    \qquad\text{using Lemma~\ref{lemma: algo, bound supremum of truncated X}}.
\end{align*}
Therefore, unconditionally,
\begin{align*}
    \P\Big( \sum_{j > m + \ceil{ \log_2(n^d) }} \big(\xi^{[n]}_j(l)\big)^+ > x\Big)
    & \leq \frac{C_X}{x}\E\Big[ \sqrt{ \breve l_{m + \ceil{\log_2(n^d)}}(l)} + \E \breve l_{m + \ceil{\log_2(n^d)}}(l) \Big]
    \\
    & \leq 
     \frac{C_X}{x}\Big[ \sqrt{\E \breve l_{m + \ceil{\log_2(n^d)}}(l)} + \E \breve l_{m + \ceil{\log_2(n^d)}}(l) \Big]
\end{align*}
The last line follows from Jensen's inequality.
Lastly, by definition of $l_j(l)$'s in \eqref{def general xi n m l, 2}, we have
\begin{align*}
    \E \breve l_{m + \ceil{\log_2(n^d)}}(l) & = \frac{l}{2^{m + \ceil{\log_2(n^d)}}} 
    \leq \frac{l}{2^m \cdot n^d} \leq \frac{n}{2^m \cdot n^d} = \frac{1}{n^{d-1} \cdot 2^m}
    \qquad \text{due to $l \in [0,n]$}.
\end{align*}
This concludes the proof.
\end{proof}

\begin{lemma}
\label{lemma: algo, bound term 4}
\linksinthm{lemma: algo, bound term 4}
 Let $n \in \mathbb Z_+$ and $l \in [0,n]$.
 Let $C$ and $\lambda$ be the constants in Assumption \ref{assumption: holder continuity strengthened on X < z t}.
 Let  $C_X < \infty$ be the constant characterized in Lemma \ref{lemma: algo, bound supremum of truncated X} that only  depends on the law of L\'evy process $X$.
The inequality
\begin{align*}
     \P\Big( \sum_{j = 1}^{m + \ceil{ \log_2(n^d) }} \big(\xi^{[n]}_j(l)\big)^+ \in [y,y + c ]\Big)
     \leq 
     C\frac{ (m + (\ceil{\log_2(n^d)}) n^{\alpha_4 \lambda}}{ \delta^{\alpha_3 \lambda}  }c
    +
    4C_X\big(m^2 + (\ceil{\log_2(n^d)})^2\big)\frac{ \delta^{\alpha_3/2} }{y_0 \cdot  n^{\alpha_4/2} }.
\end{align*}
holds for all $y \geq y_0 > 0$, $c > 0$, $n \geq 1$, and $m \in \mathbb N$.
\end{lemma}

\begin{proof}
\linksinpf{lemma: algo, bound term 4}
To simplify notations, in this proof we set $k(n) = \ceil{\log_2(n^d)}$ and write $l_j = l_j(l)$ when there is no ambiguity.
For the sequence of random variables $(l_1,\cdots,l_{m + k(n)})$,
let $\Tilde{l}_1 \geq \Tilde{l}_2 \geq \cdots \geq \Tilde{l}_{m + k(n)}$ be its order statistics.
Given any ordered positive real sequence $t_1 \geq t_2 \geq \cdots \geq t_{m + k(n)} > 0$, by conditioning on $\Tilde{l}_j = t_j\ \forall j \in [m + k(n)]$,
it follows from \eqref{def general xi n m l, 3} that
\begin{align}
    \P\Big( \sum_{j = 1}^{m + k(n)}\big(\xi^{[n]}_j(l)\big)^+ \in [y,y+c]\ \Big|\ \Tilde{l}_j = t_j\ \forall j \in [m + k(n)]\Big)
    =
    \P\Big( \sum_{j = 1}^{m + k(n)} \big(\Xi_{n}^{(j)}(t_j)\big)^+ \in [y,y+c] \Big)
    \label{proof, eq 1, lemma: algo, bound term 4}
\end{align}
where $\Xi_{n}^{(j)}$'s are iid copies of the L\'evy processes $\Xi_n = X^{<n\gamma}$.
Next, fix
\begin{align*}
    \eta = \delta^{m\alpha_3}/n^{\alpha_4}.
\end{align*}
Given the sequence of real numbers $t_j$'s, we define 
$
J \delequal \#\{ j \in [m + k(n)]:\ t_j > \eta \}
$
as the number of elements in the sequence that are larger than $\eta$.
In case that $t_1 \leq \eta$, we set $J = 0$.
With $J$ defined,
we consider a decomposition of events in \eqref{proof, eq 1, lemma: algo, bound term 4}
based on the first $j \in [m + k(n)]$ such that $\Xi_{n}^{(j)}(t_j) > 0$ (and hence $\big(\Xi_{n}^{(j)}(t_j)\big)^+ > 0$),
especially if such $t_j$ is larger than $\eta$ or not.
To be specific,
\begin{equation}
\label{proof, decomp, lemma: algo, bound term 4}
    \begin{split}
    & \P\Big( \sum_{j = 1}^{m + k(n)} \big(\Xi_{n}^{(j)}(t_j)\big)^+ \in [y,y+c] \Big)
    \\
    & = 
    \sum_{j = 1}^J \underbrace{ \P\bigg( \Xi_{n}^{(i)}(t_i) \leq 0\ \forall i \in [j-1];\  \Xi_{n}^{(j)}(t_j) > 0;\ 
    \sum_{i = j}^{m + k(n)}\big(\Xi_{n}^{(i)}(t_i)\big)^+  \in [y,y+c]
    \bigg) }_{ \delequal p_j }
    \\& + \underbrace{ 
    \P\bigg( 
    \Xi_{n}^{(i)}(t_i) \leq 0\ \forall i \in [J];
    \sum_{j = J + 1}^{m + k(n)}\big(\Xi_{n}^{(j)}(t_j)\big)^+  \in [y,y+c]
    \bigg)
    }_{\delequal p_*}.
    \end{split}
\end{equation}
We first bound terms $p_j$'s.
For any $j \in [J]$, observe that
\begin{align}
    p_j & \leq 
    \P\bigg( \Xi_{n}^{(j)}(t_j) > 0;\ 
    \sum_{i = j}^{m + k(n)}\big(\Xi_{n}^{(i)}(t_i)\big)^+  \in [y,y+c]
    \bigg)
    \nonumber
    \\
    & = \int_\R
    \P\bigg( \Xi_{n}^{(j)}(t_j) \in [y - x, y - x + c] \cap (0,\infty)  \bigg)
    \P\bigg( \sum_{i = j+ 1}^{m + k(n)}\big(\Xi_{n}^{(i)}(t_i)\big)^+ \in dx \bigg)
    \nonumber
    \\
    & \leq \frac{Cc}{ t_j^\lambda \wedge 1 }
    \qquad \text{ by Assumption~\ref{assumption: holder continuity strengthened on X < z t}}
    \nonumber
    \\
    & \leq \frac{C n^{\alpha_4 \lambda}}{ \delta^{m\alpha_3 \lambda}  }\cdot c
    \qquad\text{ due to $j \leq J$, and hence $t_j > \eta = \delta^{m\alpha_3}/n^{\alpha_4}$}.
    \label{proof, bound term pj, lemma: algo, bound term 4}
\end{align}
On the other hand,
\begin{align}
    p_* & \leq 
     \P\bigg( 
    \sum_{j = J + 1}^{m + k(n)}\big(\Xi_{n}^{(j)}(t_j)\big)^+  \in [y,y+c]
    \bigg)
    \leq 
    \P\bigg( 
    \sum_{j = J + 1}^{m + k(n)}\big(\Xi_{n}^{(j)}(t_j)\big)^+  \geq y_0
    \bigg)
    \qquad \text{ due to }y \geq y_0 > 0
    \nonumber
    \\
    & \leq
    \sum_{j = J + 1}^{m + k(n)}\P\Big( 
   \Xi_{n}^{(j)}(t_j) \geq y_0/N
    \Big)
    \qquad \text{ with }N \delequal m + k(n) - J
    \nonumber
    \\
    & \leq 
     \sum_{j = J + 1}^{m + k(n)}\frac{ C_X(\sqrt{t_j} + t_j) \cdot N }{y_0}
     \qquad \text{ by Lemma~\ref{lemma: algo, bound supremum of truncated X}}
\nonumber
     \\
     & \leq 
     \sum_{j = J + 1}^{m + k(n)}\frac{ C_X(\sqrt{\eta} + \eta) \cdot N }{y_0}
     \qquad \text{due to }j > J\text{, and hence }t_j \leq \eta = \delta^{m\alpha_3}/n^{\alpha_4}
     \nonumber
     \\
     & = 
     N^2 \cdot \frac{ C_X(\sqrt{\eta} + \eta) }{y_0}
     \leq 
     \big(m + k(n)\big)^2\cdot \frac{ C_X(\sqrt{\eta} + \eta) }{y_0}
     \qquad \text{ due to }N \leq m+k(n)
     \nonumber
     \\
     & \leq 
     2C_X\big(m^2 + (\ceil{\log_2(n^d)})^2\big)\frac{\sqrt{\eta} + \eta}{y_0}
     \qquad \text{ using }(u+v)^2\leq 2(u^2 + v^2)
     \nonumber
     \\
     & \leq 
     4C_X\big(m^2 + (\ceil{\log_2(n^d)})^2\big)\frac{\sqrt{\eta}}{y_0}
     =
     4C_X\big(m^2 + (\ceil{\log_2(n^d)})^2\big)\frac{ \delta^{m\alpha_3/2} }{y_0 \cdot  n^{\alpha_4/2} }.
     \label{proof, bound term p*, lemma: algo, bound term 4}
\end{align}
Plugging \eqref{proof, bound term pj, lemma: algo, bound term 4} and \eqref{proof, bound term p*, lemma: algo, bound term 4} into \eqref{proof, decomp, lemma: algo, bound term 4}, we yield
\begin{align*}
    & \P\Big( \sum_{j = 1}^{m + k(n)} \big(\xi^{[n]}_j(l)\big)^+ \in [y,y + c ]\ \Big|\ \Tilde{l}_j = t_j\ \forall j \in [m + k(n)]\Big) 
    \\
    & \leq 
    J \cdot \frac{C n^{\alpha_4 \lambda}}{ \delta^{m\alpha_3 \lambda}  }c
    +
    4C_X\big(m^2 + (\ceil{\log_2(n^d)})^2\big)\frac{ \delta^{m\alpha_3/2} }{y_0 \cdot  n^{\alpha_4/2} }
    \\ 
    & 
    \leq 
    C\frac{ (m + (\ceil{\log_2(n^d)}) n^{\alpha_4 \lambda}}{ \delta^{m\alpha_3 \lambda}  }c
    +
    4C_X\big(m^2 + (\ceil{\log_2(n^d)})^2\big)\frac{ m\delta^{\alpha_3/2} }{y_0 \cdot  n^{\alpha_4/2} }
    \qquad 
    \text{ due to }J \leq m + \ceil{ \log_2(n^d) }.
\end{align*}
To conclude the proof, just note that the inequality above holds when conditioning on any sequence of $t_1 \geq t_2 \geq \cdots \geq t_{m + k(n)} > 0$, so it would also hold unconditionally.
\end{proof}

Now, we are ready to prove Propositions~\ref{proposition, intermediate 1, strong efficiency ARA} and \ref{proposition, intermediate 2, strong efficiency ARA}.

\begin{proof}[Proof of Proposition~\ref{proposition, intermediate 1, strong efficiency ARA}]
\linksinpf{proposition, intermediate 1, strong efficiency ARA}
   In this proof, we fix some $k \in \mathbb N$, $n \geq 1$, and $m \in \mathbb N$.
Let the process $\zeta_k$ be defined as in \eqref{proof: def zeta k}.
Recall the definitions of $W^{(i),*}_n(\zeta_k)$, $\hat W^{(i),m}_n(\zeta_k)$, and $\widetilde  W^{(i),m}_n(\zeta_k)$
in \eqref{proof,def: W i * n}--\eqref{proof, def: tilde W i m n}. 
See also \eqref{defStickLength1}--\eqref{def xi i m j, ARA plus SBA}
for the definitions $\xi^{(i)}_j$'s and $\xi^{(i),m}_j$'s.

To simplify notations, define $t(n) = \ceil{\log_2(n^d)}$.
Define events
\begin{align*}
    E_{1}^{(i)}(x) & \delequal \bigg\{ \Big|\sum_{q = 1}^{i - 1}\sum_{j \geq 0}\xi^{(q)}_j - \sum_{q = 1}^{i - 1}\sum_{j \geq 0}\xi^{(q),m}_j\Big| \leq \frac{x}{2} \bigg\},
    \\ 
    E_{2}^{(i)}(x) & \delequal \bigg\{
    \Big|\sum_{j = 1}^{m + t(n)}(\xi^{(i)}_j)^+ - \sum_{j = 1}^{m + t(n)}(\xi^{(i),m}_j)^+\Big| \leq \frac{x}{2}
    \bigg\},
    \\
    E_{3}^{(i)}(x) & \delequal \bigg\{
    \sum_{j \geq m + t(n) + 1}(\xi^{(i)}_j)^+ \leq x
    \bigg\}.
\end{align*}
Note that on event $E_{1}^{(i)}(x) \cap E_{2}^{(i)}(x) \cap E_{3}^{(i)}(x)$,
we must have 
$
| W^{(i),*}_n(\zeta_k) - \widetilde  W^{(i),m}_n(\zeta_k)| \leq x
$
and
$
|\hat W^{(i),m}_n(\zeta_k) - \widetilde  W^{(i),m}_n(\zeta_k)| \leq x.
$
As a result,
\begin{align*}
    \P\bigg( \Big| W^{(i),*}_n(\zeta_k) - \widetilde  W^{(i),m}_n(\zeta_k) \Big| \vee 
        \Big|  \hat W^{(i),m}_n(\zeta_k) - \widetilde  W^{(i),m}_n(\zeta_k) \Big|
        > 
        x
        \bigg)
        \leq \sum_{q = 1}^3 \P\Big(\big(E^{(i)}_q(x)\big)^\complement\Big).
\end{align*}
Furthermore, we claim the existence of constant $(\widetilde C_q)_{q = 1,2,3}$, the values of which do not depend on $x,k,n$, and $m$,
such that (for any $x > 0$ and $i \in [k+1]$)
\begin{align}
    \P\Big(\big(E_{1}^{(i)}(x)\big)^c\Big)
    =
    \P\bigg(\Big|\sum_{q = 1}^{i - 1}\sum_{j \geq 0}\xi^{(q)}_j - \sum_{q = 1}^{i - 1}\sum_{j \geq 0}\xi^{(q),m}_j\Big| > 
    \frac{x}{2}
    \bigg) & \leq \frac{\widetilde C_1\kappa^{m(2-\beta_+)}}{x^2 n^{r(2 - \beta_+) - 1} },
    \label{proof, goal 1, proposition: hat Y m n condition 1}
    \\ 
    \P\Big(\big(E_{2}^{(i)}(x)\big)^c\Big)
    =
    \P\bigg(
    \Big|\sum_{j = 1}^{m + t(n)}(\xi^{(i)}_j)^+ - \sum_{j = 1}^{m + t(n)}(\xi^{(i),m}_j)^+\Big| > \frac{x}{2}
    \bigg)
     & \leq \frac{\widetilde C_2\kappa^{m(2-\beta_+)}}{x^2 n^{r(2 - \beta_+) - 1} }, 
     \label{proof, goal 2, proposition: hat Y m n condition 1}
     \\ 
     \P\Big(\big(E_{3}^{(i)}(x)\big)^c\Big)
    =
    \P\bigg(
    \sum_{j \geq m + t(n) + 1}(\xi^{(i)}_j)^+ > x
    \bigg)
    & \leq \frac{\widetilde C_3}{x}\sqrt{\frac{1}{n^{d-1} \cdot 2^m}}.
    \label{proof, goal 3, proposition: hat Y m n condition 1}
\end{align}
This allows us to conclude the proof by setting $C_1 = \sum_{q = 1}^3 \widetilde C_q$.
Now, it remains to prove claims \eqref{proof, goal 1, proposition: hat Y m n condition 1}--\eqref{proof, goal 3, proposition: hat Y m n condition 1}.

\medskip
\noindent\textbf{Proof of Claim }\eqref{proof, goal 1, proposition: hat Y m n condition 1}

The claim is trivial if $i \leq 1$, so we only consider the case where $i \geq 2$. 
Due to the coupling between $\xi^{(i)}_j$ and $\xi^{(i),m}_j$ in \eqref{def xi i m j, ARA plus SBA}\eqref{coupling, ARA plus SBA}, we have
\begin{align*}
    \Big(\sum_{q = 1}^{i - 1}\sum_{j \geq 0}\xi^{(q)}_j,\sum_{q = 1}^{i - 1}\sum_{j \geq 0}\xi^{(q),m}_j\Big)
    \distequal
    \big(\Xi_n(u_{i-1}),\breve \Xi^m_n(u_{i-1})\big)
\end{align*}
where the laws of processes $\Xi_n,\breve \Xi^m_n$ are stated in \eqref{decomp, Xi n} and \eqref{def breve Xi n m, ARA}, respectively.
Applying Lemma \ref{lemma: algo, bound term 1}, we yield
\begin{align*}
    \P\bigg(\bigg|\sum_{q = 1}^{i - 1}\sum_{j \geq 0}\xi^{(q)}_j - \sum_{q = 1}^{i - 1}\sum_{j \geq 0}\xi^{(q),m}_j\bigg| > \frac{x}{2} \bigg)
    & 
    \leq \P\bigg( \sup_{t \in [0,n]}|\Xi_n(t) - \breve \Xi^m_n(t)  | >\frac{x}{2}\bigg)
    \leq 
    \frac{4C\kappa^{m(2-\beta_+)}}{x^2 n^{r(2 - \beta_+) - 1} },
\end{align*}
where $C < \infty$ is the constant characterized in Lemma~\ref{lemma: algo, bound term 1} that only depends on $\beta_+$ and the law of the L\'evy process $X$.
To conclude the proof of claim \eqref{proof, goal 1, proposition: hat Y m n condition 1}, we pick $\widetilde C_1 = 4C$.

\medskip
\noindent\textbf{Proof of Claim }\eqref{proof, goal 2, proposition: hat Y m n condition 1}

It follows directly from Lemma~\ref{lemma: algo, bound term 2} that
\begin{align*}
    & \P\bigg(
      \bigg|\sum_{j = 1}^{m + t(n)}(\xi^{(i)}_j)^+ - \sum_{j = 1}^{m + t(n)}(\xi^{(i),m}_j)^+\bigg|  > \frac{x}{2}
     \bigg)
     \leq 
     \frac{4C\kappa^{m(2-\beta_+)}}{x^2 n^{r(2 - \beta_+) - 1} },
\end{align*}
where $C < \infty$ is the constant characterized in Lemma~\ref{lemma: algo, bound term 2} that only depends on $\beta_+$ and the law of the L\'evy process $X$.
To conclude the proof of claim \eqref{proof, goal 2, proposition: hat Y m n condition 1}, we pick $\widetilde C_2 = 4C$.

\medskip
\noindent\textbf{Proof of Claim }\eqref{proof, goal 3, proposition: hat Y m n condition 1}

Using Lemma \ref{lemma: algo, bound term 3},
\begin{align*}
    \P\bigg(
     \sum_{j \geq m + t(n) + 1}(\xi^{(i)}_j)^+ > x \bigg)
     & 
     \leq 
     \frac{2C_X}{x}  \cdot\sqrt{\frac{1}{n^{d-1} \cdot 2^m}}
\end{align*}
where $C_X < \infty$ is the constant characterized in Lemma \ref{lemma: algo, bound term 3}  that only depends on the law of the L\'evy process $X$.
By setting $\widetilde C_3 = 2C_X$, we conclude the proof of claim \eqref{proof, goal 3, proposition: hat Y m n condition 1}.
\end{proof}

\begin{proof}[Proof of Proposition~\ref{proposition, intermediate 2, strong efficiency ARA}]
\linksinpf{proposition, intermediate 2, strong efficiency ARA}
In this proof, we fix some $k \in \mathbb N$.
Recall the representation 
$
\zeta_k(t) = \sum_{i = 1}^k z_i \mathbf{I}_{[u_i,n]}(t)
$
in \eqref{proof: def zeta k}
where $0 < u_1 < u_2 < \ldots < u_k < n$
are the order statistics of $k$ iid samples of Unif$(0,n)$.
Recall the definition of $\widetilde  W^{(i),m}_n(\zeta_k)$
in \eqref{proof, def: tilde W i m n}. 
See also \eqref{defStickLength1}--\eqref{def xi i m j, ARA plus SBA}
for the definitions $\xi^{(i)}_j$'s and $\xi^{(i),m}_j$'s.

We start with the following decomposition of events:
\begin{align*}
    & \P\bigg(\exists i \in [k+1]\ s.t.\
            \widetilde W^{(i),m}_n(\zeta_k) \in \bigg[ y - \frac{\delta^m}{\sqrt{n}},  y + \frac{\delta^m}{\sqrt{n}}  \bigg]
        \bigg)
        \\ 
        & \leq \P(u_1 < n\delta^{m\alpha_1})
        +
        \P\bigg(\exists i \in [k+1]\ s.t.\
            \widetilde W^{(i),m}_n(\zeta_k) \in \bigg[ y - \frac{\delta^m}{\sqrt{n}},  y + \frac{\delta^m}{\sqrt{n}}  \bigg],
            \ u_1 \geq n\delta^{m\alpha_1}
        \bigg).
\end{align*}
First, 
$
\P(u_1 < n\delta^{m\alpha_1}) \leq k \cdot \P(\text{Unif}(0,n) < n\delta^{m\alpha_1}) = k \cdot \delta^{m\alpha_1} < k \cdot \rho^m_0.
$
The last inequality follows from our choice of $\rho_1$ in \eqref{proofChooseRhoTimeBound} and $\rho_0 \in (\rho_1,1)$.
Furthermore, for each $i \in [k+1]$
\begin{align*}
     & \P\bigg(\widetilde W^{(i),m}_n(\zeta_k) \in [y - \frac{\delta^m}{\sqrt{n}},y+ \frac{\delta^m}{\sqrt{n}}],\ u_1 \geq n\delta^{m\alpha_1}\bigg)
\\
     & = \P\bigg(\sum_{q = 1}^{i - 1}\sum_{j \geq 0}\xi^{(q)}_j + \sum_{q = 1}^{i-1}z_q \in [ y - \delta^{m\alpha_2}, y + \delta^{m\alpha_2} ],\ \widetilde W^{(i),m}_n(\zeta_k) \in [y - \frac{\delta^m}{\sqrt{n}},y+ \frac{\delta^m}{\sqrt{n}}],\ u_1 \geq n\delta^{m\alpha_1}\bigg)
     \\ 
     &\qquad + 
     \P\bigg(\sum_{q = 1}^{i - 1}\sum_{j \geq 0}\xi^{(q)}_j + \sum_{q = 1}^{i-1}z_q \notin [ y - \delta^{m\alpha_2}, y + \delta^{m\alpha_2} ],\ \widetilde W^{(i),m}_n(\zeta_k) \in [y - \frac{\delta^m}{\sqrt{n}},y+ \frac{\delta^m}{\sqrt{n}}],\ u_1 \geq n\delta^{m\alpha_1}\bigg)
     \\ 
     & \leq 
     \P\bigg(\sum_{q = 1}^{i - 1}\sum_{j \geq 0}\xi^{(q)}_j + \sum_{q = 1}^{i-1}z_q \in [ y - \delta^{m\alpha_2}, y + \delta^{m\alpha_2} ],\ u_1 \geq n\delta^{m\alpha_1}\bigg)
     \\ 
     &\qquad + 
     \P\bigg(\sum_{q = 1}^{i - 1}\sum_{j \geq 0}\xi^{(q)}_j + \sum_{q = 1}^{i-1}z_q \notin [ y - \delta^{m\alpha_2}, y + \delta^{m\alpha_2} ],\ \widetilde W^{(i),m}_n \in [y - \frac{\delta^m}{\sqrt{n}},y+ \frac{\delta^m}{\sqrt{n}}]\bigg)
     \\ 
     & \leq \P\Big( \sum_{q = 1}^{i - 1}\sum_{j \geq 0}\xi^{(q)}_j + \sum_{q = 1}^{i-1}z_q \in [ y - \delta^{m\alpha_2}, y + \delta^{m\alpha_2} ],\ u_1 \geq n\delta^{m\alpha_1} \Big)
     \\ 
     &\qquad +
     \int_{ \R \symbol{92} [ y - \delta^{m\alpha_2}, y + \delta^{m\alpha_2} ]  }\P\bigg(\sum_{j = 1}^{m+t(n)}(\xi^{(i)}_j)^+ \in [y - x - \frac{\delta^m}{\sqrt{n}},y - x+ \frac{\delta^m}{\sqrt{n}}]\bigg)
     \P\Big( \sum_{q = 1}^{i - 1}\sum_{j \geq 0}\xi^{(q)}_j + \sum_{q = 1}^{i-1}z_q \in dx  \Big)
     \\ 
     & = \P\Big( \sum_{q = 1}^{i - 1}\sum_{j \geq 0}\xi^{(q)}_j + \sum_{q = 1}^{i-1}z_q \in [ y - \delta^{m\alpha_2}, y + \delta^{m\alpha_2} ],\ u_1 \geq n\delta^{m\alpha_1}\Big)
     \\ 
     &\qquad +
     \int_{ (-\infty,  y - \delta^{m\alpha_2} ]  }\P\bigg(\sum_{j = 1}^{m+t(n)}(\xi^{(i)}_j)^+ \in [y - x - \frac{\delta^m}{\sqrt{n}},y - x+ \frac{\delta^m}{\sqrt{n}}]\bigg)
     \P\Big( \sum_{q = 1}^{i - 1}\sum_{j \geq 0}\xi^{(q)}_j + \sum_{q = 1}^{i-1}z_q \in dx  \Big).
\end{align*}
The last equality follows from the simple fact that $\sum_{j \geq 1}(\xi^{(i)}_j)^+ \geq 0$.
Furthermore, we claim the existence of constants $\widetilde C_1$ and $\widetilde C_2$, the values of which do not vary with parameters $n,m,k,y,i$, such that for all $n \geq 1$ and $m \geq \bar m$,
\begin{align}
        \P\Big( \sum_{q = 1}^{i - 1}\sum_{j \geq 0}\xi^{(q)}_j + \sum_{q = 1}^{i-1}z_q \in [ y - \delta^{m\alpha_2}, y + \delta^{m\alpha_2} ],\ u_1 \geq n\delta^{m\alpha_1} \Big)
    & \leq \widetilde C_1 \rho_0^m\qquad \forall y > \delta^{m\alpha_2},
    \label{proof, goal 4, proposition: hat Y m n condition 1}
    \\ 
    \P\bigg(\sum_{j = 1}^{m+t(n)}(\xi^{(i)}_j)^+ \in [w,w +\frac{2\delta^m}{\sqrt{n}}]\bigg)
    & \leq \widetilde C_2 \rho_0^m\qquad \forall w \geq \delta^{m\alpha_2} - \frac{\delta^m}{\sqrt{n}}.
    \label{proof, goal 5, proposition: hat Y m n condition 1}
\end{align}
Then, we conclude the proof by setting $C_2 = 1 + \widetilde C_1 + \widetilde C_2$.
Now, we prove claims \eqref{proof, goal 4, proposition: hat Y m n condition 1} and \eqref{proof, goal 5, proposition: hat Y m n condition 1}

\medskip
\noindent\textbf{Proof of Claim }\eqref{proof, goal 4, proposition: hat Y m n condition 1}

If $i \leq 1$, the claim is trivial due to $y > \delta^{m\alpha}$ and hence $0 \notin [y-\delta^{m\alpha_2},y + \delta^{m\alpha_2}]$.
Now, we consider the case where $i \geq 2$.
Due to the independence between $z_i$ and $\xi^{(i)}_j$,
\begin{align*}
&  \P\Big( \sum_{q = 1}^{i - 1}\sum_{j \geq 0}\xi^{(q)}_j + \sum_{q = 1}^{i-1}z_q \in [ y - \delta^{m\alpha_2}, y + \delta^{m\alpha_2} ],\ u_1 \geq n\delta^{m\alpha_1} \Big)
\\ 
& =
\int_{\R} \P\Big(
 \sum_{q = 1}^{i - 1}\sum_{j \geq 0}\xi^{(q)}_j  \in [ y - x - \delta^{m\alpha_2}, y - x + \delta^{m\alpha_2} ],\ u_1 \geq n\delta^{m\alpha_1}
\Big) \P(\sum_{q = 1}^{i-1}z_q \in dx)
\\ 
& \leq 
\int_{\R} \P\Big(
 \sum_{q = 1}^{i - 1}\sum_{j \geq 0}\xi^{(q)}_j  \in [ y - x - \delta^{m\alpha_2}, y - x + \delta^{m\alpha_2} ]\ \Big|\ u_1 \geq n\delta^{m\alpha_1}
\Big) \P(\sum_{q = 1}^{i-1}z_q \in dx)
\\ 
& = 
\int_{\R} \P\Big( X^{<n\gamma}(u_{i-1}) \in [ y - x - \delta^{m\alpha_2}, y - x + \delta^{m\alpha_2} ]\ \Big|\ u_1 \geq n\delta^{m\alpha_1}
\Big) \P(\sum_{q = 1}^{i-1}z_q \in dx)
\end{align*}
where $(u_i)_{i = 1}^k$ are independent of the L\'evy process $X^{<n\gamma}$.
In particular, recall that $0 = u_0 < u_1 < u_2 < \ldots < u_k < n$ are order statistics.
Therefore, on event $\{u_1 \geq n\delta^{m\alpha_1}\}$ we must have $u_{i-1} \geq u_1 \geq n\delta^{m\alpha_1}$.
It then follows directly from Assumption~\ref{assumption: holder continuity strengthened on X < z t} that
\begin{align*}
    & \P\Big( X^{<n\gamma}(u_{i-1}) \in [ y - x - \delta^{m\alpha_2}, y - x + \delta^{m\alpha_2} ]\ \Big|\ u_1 \geq n\delta^{m\alpha_1}
\Big)
\\
& \leq \frac{C}{ (n^\lambda\delta^{m\alpha_1\lambda})\wedge 1 }\cdot 2\delta^{m\alpha_2}
\leq 
2 C \cdot \bigg( \frac{\delta^{ \alpha_2}}{ \delta^{\lambda \alpha_1} } \bigg)^m
\leq 
2 C \cdot \rho_0^m\qquad \text{due to \eqref{proofChooseRhoByAlpha_12} and }\rho_0 \in (\rho_1,1),
\end{align*}
where $C$ and $\lambda$ are the constants specified in Assumption~\ref{assumption: holder continuity strengthened on X < z t}.
To conclude, it suffices to set $\widetilde C_1 = 2C$.


\medskip
\noindent\textbf{Proof of Claim }\eqref{proof, goal 5, proposition: hat Y m n condition 1}

Applying Lemma \ref{lemma: algo, bound term 4} with $y_0 = \delta^{m\alpha_2} - \frac{\delta^m}{\sqrt{n}}$ and $c = \frac{2\delta^m}{\sqrt{n}}$,
we get (for all $n \geq 1, m \geq \bar m, y \geq y_0$)
\begin{align*}
    & \P\bigg(\sum_{j = 1}^{m+t(n)}(\xi^{(i)}_j)^+ \in [y,y +\frac{2\delta^m}{\sqrt{n}}]\bigg)
    \\ 
    & \leq 
    C\frac{ (m + (\ceil{\log_2(n^d)}) n^{\alpha_4 \lambda}}{ \delta^{m\alpha_3 \lambda}  }
    \cdot \frac{ 2 \delta^{m} }{\sqrt{n}}
    +
    4C_X\big(m^2 + (\ceil{\log_2(n^d)})^2\big)\frac{ \delta^{m\alpha_3/2} }{(\delta^{m\alpha_2} - \frac{\delta^m}{\sqrt{n}}) \cdot  n^{\alpha_4/2} }
    \\ 
    & \leq
    C\frac{ (m + (\ceil{\log_2(n^d)}) n^{\alpha_4 \lambda}}{ \delta^{m\alpha_3 \lambda}  }
    \cdot \frac{ 2 \delta^{m} }{\sqrt{n}}
    +
     8C_X\big(m^2 + (\ceil{\log_2(n^d)})^2\big)
     \frac{ \delta^{m\alpha_3/2} }{{\delta^{m\alpha_2}} \cdot  n^{\alpha_4/2} }
     \qquad \text{due to \eqref{proofChooseMbar}}
     \\ 
     & = 
     \underbrace{ 2 C \cdot \frac{ m }{n^{ \frac{1}{2} - \lambda\alpha_4}} \cdot \bigg( \frac{\delta}{\delta^{\lambda\alpha_3}} \bigg)^m }_{ \delequal p_{n,m,1} }
     + 
      \underbrace{ 2 C \cdot \frac{  \ceil{ \log_2(n^d) }  }{n^{ \frac{1}{2} - \lambda\alpha_4}} \cdot \bigg( \frac{\delta}{\delta^{\lambda\alpha_3}} \bigg)^m }_{ \delequal p_{n,m,2} }
      \\ 
      & \qquad 
      + 
      \underbrace{
      8C_X  \cdot \frac{m^2}{ n^{{\alpha_4}/{2} } } \cdot \bigg( \frac{ \delta^{\alpha_3/2} }{ \delta^{\alpha_2} }\bigg)^m
      }_{\delequal p_{n,m,3}}
      + 
      \underbrace{
      8C_X  \cdot \frac{ \big( \ceil{ \log_2(n^d) } \big)^2  }{ n^{{\alpha_4}/{2} } } \cdot \bigg( \frac{ \delta^{\alpha_3/2} }{ \delta^{\alpha_2} }\bigg)^m.
      }_{\delequal p_{n,m,4}}
\end{align*}
Here, $C_X < \infty$ is the constant in Lemma \ref{lemma: algo, bound supremum of truncated X} that only  depends on the law of L\'evy process $X$,
and 
$C \in (0,\infty),\lambda > 0$ are the constants in Assumption~\ref{assumption: holder continuity strengthened on X < z t}.
First, for any $n \geq 1$ and $m \geq \bar m$,
\begin{align*}
    p_{n,m,1} & \leq 2 C \cdot m \cdot \bigg( \frac{\delta}{\delta^{\lambda\alpha_3}} \bigg)^m
    \qquad \text{due to $\frac{1}{2} > \lambda\alpha_4$; see \eqref{proofChooseAlpha_34}}
    \\ 
    & \leq 2 C \cdot m \rho_1^m
    \qquad \text{due to \eqref{proofChooseRhoSBALongStick}}
    \\ 
    & \leq 2 C \cdot \rho_0^m
    \qquad \text{ due to \eqref{proof, choose bar m and rho 0}}.
\end{align*}
For term $p_{n,m,2}$, note that $\frac{  \ceil{ \log_2(n^d) }  }{n^{ \frac{1}{2} - \lambda\alpha_4}} \rightarrow 0$ as $n \rightarrow \infty$ due to $\frac{1}{2} > \lambda\alpha_4$.
This allows us to fix some $C_{d,1} < \infty$ such that $\sup_{n = 1,2,\cdots}\frac{  \ceil{ \log_2(n^d) }  }{n^{ \frac{1}{2} - \lambda\alpha_4}} \leq C_{d,1}$.
As a result, for any $n \geq 1, m \geq 0$,
\begin{align*}
    p_{n,m,2} & \leq 2 C C_{d,1} \cdot \bigg( \frac{\delta}{\delta^{\lambda\alpha_3}} \bigg)^m
    \leq  2 C C_{d,1} \cdot  \rho^m_0
    \qquad \text{ due to \eqref{proofChooseRhoSBALongStick} and }\rho_0 \in (\rho_1,1).
\end{align*}
Similarly, for all $n \geq 1$ and $m \geq \bar m$,
\begin{align*}
    p_{n,m,3} & \leq 8 C_X \cdot m^2 \cdot \bigg( \frac{ \delta^{\alpha_3/2} }{ \delta^{\alpha_2} }\bigg)^m
    \leq  8 C_X \cdot m^2\rho^m_1
    \qquad \text{due to \eqref{proofChooseRhoSBAShortStick}}
    \\ 
    & \leq 8 C_X \cdot \rho^m_0\qquad \text{ due to \eqref{proof, choose bar m and rho 0}}.
\end{align*}
Besides, due to $\frac{ ( \ceil{ \log_2(n^d) } )^2  }{ n^{{\alpha_4}/{2} } } \rightarrow 0$ as $n \rightarrow \infty$,
we can find $C_{d,2} < \infty$ such that $\sup_{n = 1,2,\cdots,}\frac{ ( \ceil{ \log_2(n^d) } )^2  }{ n^{{\alpha_4}/{2} } } \leq C_{d,2}$.
This leads to (for all $n \geq 1,m \geq 0$)
\begin{align*}
    p_{n,m,4} & \leq 8C_XC_{d,2} \cdot \bigg( \frac{ \delta^{\alpha_3/2} }{ \delta^{\alpha_2} }\bigg)^m
    \leq 
     8C_XC_{d,2} \cdot \rho^m_0
     \qquad \text{ due to \eqref{proofChooseRhoSBALongStick} and }\rho_0 \in (\rho_1,1).
\end{align*}
To conclude the proof, we can simply set
$
\widetilde C_2 =  2 C + 2 C C_{d,1} + 8C_X + 8C_XC_{d,2}.
$
\end{proof}

\subsection{Proof of Propositions~\ref{proposition: lipschitz cont, 1} and \ref{proposition: lipschitz cont, 2}}
\label{subsec: proof, lipschitz cont}

The proof of Proposition \ref{proposition: lipschitz cont, 1} is
based on the inversion formula of the characteristic functions (see, e.g., Theorem 3.3.14 of \cite{durrett2019probability}).
Specifically, we compare the characteristic function of $Y(t)$ with an $\alpha$-stable process to draw connections between their distributions.

\begin{proof}[Proof of Proposition \ref{proposition: lipschitz cont, 1}]
\linksinpf{proposition: lipschitz cont, 1}
The Lévy-Khintchine formula (see e.g.\ Theorem~8.1 of \cite{sato1999levy}) leads to the following expression for the characteristic function of $\varphi_{t}(z) = \E\exp(i z Y(t))$:
\begin{align*}
    \varphi_{t}(z) = \exp\Big( t\int_{(0,z_0)}\big[\underbrace{\exp(izx) -1 - izx\mathbf{I}_{(0,1]}(x)}_{ \delequal \phi(z,x) }\big]\mu(dx)\Big)\qquad \forall z\in\mathbb{R},\ t > 0.
\end{align*}
Note that 
\begin{align*}
    \phi(z,x) & = \cos(zx) - 1 + i\big( \sin(zx) - zx \mathbf{I}_{(0,1]}(x)  \big).
\end{align*}
Then from $|e^{x+iy}| = e^x$ for all $x,y \in \R$,
\begin{align}
    |\varphi_t(z)| = \exp\Big(-t\int_{(0,z_0)}\big(1 - \cos(zx)\big)\mu(dx) \Big)
    \qquad \forall z \in \R, \ t > 0.
    \label{normOfCF_lcont}
\end{align}
Furthermore, we claim the existence of some $\widetilde M,\widetilde C\in (0,\infty)$ such that
\begin{align}
    \int_{(0,z_0)}\big(1 - \cos(zx)\big)\mu(dx) \geq \widetilde C |z|^\alpha
    \qquad \forall z \in \R\text{ with }|z| \geq \widetilde M.
    \label{goal, proposition: lipschitz cont, 1}
\end{align}
Plugging \eqref{goal, proposition: lipschitz cont, 1} into \eqref{normOfCF_lcont},
we yield that for all $|z| \geq \widetilde M$ and $t > 0$,
$
|\varphi_t(z)| \leq \exp(-t\widetilde C |z|^\alpha).
$
It then follows directly from the inversion formula (see Theorem 3.3.14 of \cite{durrett2019probability}) that, for all $t > 0$, $Y(t)$ admits a continuous density function $f_{Y(t)}$ with a uniform bound
\begin{align*}
    \norm{f_{Y(t)}}_\infty & \leq \frac{1}{2\pi}\int |\varphi_t(z)|dz \\
    & \leq \frac{1}{2\pi}\Big( 2\widetilde{M} + \int_{|z| \geq \Tilde{M}}\exp\big( -t\widetilde C |z|^\alpha \big)dz \Big) \\
    & \leq \frac{1}{2\pi}\Big( 2\widetilde{M} + \frac{1}{t^{1/\alpha}}\int_\R\exp(-\widetilde C|x|^\alpha)dx \Big)
    \qquad\text{by letting $x = zt^{1/\alpha}$}\\
    & = \frac{\widetilde{M}}{\pi} + \frac{C_1}{t^{1/\alpha}}
    \qquad \text{where $C_1 = \frac{1}{2\pi}\int_\R \exp(-\widetilde C |x|^\alpha)dx < \infty$}.
\end{align*}
To conclude the proof, pick $C = \frac{\widetilde{M}}{\pi} + C_1$.
Now, it only remains to prove claim~\eqref{goal, proposition: lipschitz cont, 1}.

\medskip
\noindent\textbf{Proof of Claim }\eqref{goal, proposition: lipschitz cont, 1}.

We start by fixing some constants.
\begin{align}
    C_0 = \int_{0}^\infty (1 - \cos{y})\frac{dy}{y^{1 + \alpha}}.
\end{align}
For $y \in (0,1]$, note that $1 - \cos y \leq y^2/2$, and hence $\frac{|1 - \cos y|}{y^{1 + \alpha}} \leq \frac{1}{2y^{\alpha - 1}}$.
For $y \in (1,\infty)$, note that $1 - \cos y \in [0,1]$ and hence $\frac{|1 - \cos y|}{y^{1 + \alpha}} \leq 1/y^{\alpha + 1}$.
Due to $\alpha \in (0,2)$, we have $C_0 = \int_{0}^\infty (1 - \cos{y})\frac{dy}{y^{1 + \alpha}} \in (0,\infty)$.
Next, choose positive real numbers $\theta,\ \delta$ such that
\begin{align}
    \frac{\theta^{2 - \alpha}}{2(2 - \alpha)}\leq \frac{C_0}{4}, \label{chooseTheta_Lcont} \\
    \frac{\delta}{\alpha\theta^\alpha}\leq \frac{C_0}{4}. \label{chooseDelta_Lcont}
\end{align}
For any $M > 0$ and $z \neq 0$, observe that (by setting $y = |z|x$ in the last step)
\begin{align*}
    \frac{\int_{x \geq \frac{M}{|z|}}\big( 1 - \cos(zx)  \big)\frac{dx}{x^{1 + \alpha}}  }{|z|^\alpha} & = \frac{\int_{x \geq \frac{M}{|z|}}\big( 1 - \cos(|z|x)  \big)\frac{dx}{x^{1 + \alpha}}  }{|z|^\alpha} 
    = \int_{M}^\infty \big(  1 -\cos{y} \big)\frac{dy}{y^{1+\alpha}}.
\end{align*}
Therefore, by fixing some $M > \theta$ large enough, we have
\begin{align}
    \frac{1}{|z|^\alpha}\int_{x \geq M/|z|}\big( 1 - \cos(zx)  \big)\frac{dx}{x^{1 + \alpha}} \leq \frac{C_0}{4}
    \qquad \forall z \neq 0.
    \label{endNonRVterm_lCont}
\end{align}

To proceed, we compare $\int_{(0,z_0)}\big( 1 - \cos(zx) \big)\mu(dx)$ with $\int_0^{M/z}\big(1 - \cos(zx)\big)\frac{dx}{x^{1 + \alpha}}$.
Recall that $z_0$ is the constant prescribed in the statement of Proposition~\ref{proposition: lipschitz cont, 1}.
For any $z \in \R$ such that $|z|>M/z_0$,
\begin{equation}\label{proof, ineq 1, proposition: lipschitz cont, 1}
    \begin{split}
        & \frac{1}{|z|^\alpha}\bigg[ \int_{(0,z_0)}\Big( 1 - \cos(zx) \Big)\mu(dx) - \int_0^{\infty}\Big( 1 - \cos(zx)\Big)\frac{dx}{x^{1 + \alpha}} \bigg] \\
    & \geq 
    \frac{1}{|z|^\alpha}\bigg[ \int_{(\theta/|z|,M/|z|)}\Big( 1 - \cos(zx) \Big)\mu(dx) - \int_0^{\infty}\Big( 1 - \cos(zx)\Big)\frac{dx}{x^{1 + \alpha}} \bigg]
    \\
    &
     \qquad \text{due to our choice of $M > \theta$ and $|z| > M/z_0$}
    \\
    & \geq 
    -\underbrace{\frac{1}{|z|^\alpha}\int_{0}^{\theta/|z|}\Big( 1 - \cos(zx)\Big)\frac{dx}{x^{1 + \alpha}}}_{ \delequal I_1(z) } 
    - 
    \underbrace{
    \frac{1}{|z|^\alpha}\int_{M/|z|}^\infty 
    \Big( 1 - \cos(zx)\Big)\frac{dx}{x^{1 + \alpha}}
    }_{\delequal I_2(z)}
    \\
    &\qquad + \underbrace{\frac{1}{|z|^\alpha}\bigg[ \int_{[\theta/|z|,M/|z|)}\Big( 1 - \cos(zx) \Big)\mu(dx) - \int_{[\theta/|z|,M/|z|)}\Big( 1 - \cos(zx)\Big)\frac{dx}{x^{1 + \alpha}}   \bigg]}_{\delequal I_3(z) }.
    \end{split}
\end{equation}
We bound the terms $I_1(z)$, $I_2(z)$, and $I_3(z)$ separately.
First, for any $z \neq 0$,
\begin{align}
  I_1(z) & \leq \frac{1}{|z|^\alpha}\int_{0}^{\theta/|z|}\frac{z^2x^2}{2}\frac{dx}{x^{1+\alpha}}
  \qquad \text{ due to }1 - \cos w \leq \frac{w^2}{2} \ \forall w \in \R
  \nonumber  \\
   & = 
\frac{1}{2}\int_0^\theta y^{1 - \alpha} dy
   \qquad \text{by setting } y = |z|x
   \nonumber \\
   & = \frac{1}{2}\cdot \frac{ \theta^{2 - \alpha} }{2 - \alpha} 
   \leq  \frac{C_0}{4}\ \ \ \text{due to \eqref{chooseTheta_Lcont}}. \label{endTerm1RV_lCont}
\end{align}
For $I_2(z)$, it follows immediately from \eqref{endNonRVterm_lCont} that 
\begin{align}
    I_2(z) \leq \frac{C_0}{4}\qquad \forall z \neq 0.
    \label{proof: bound term I 2, proposition: lipschitz cont, 1}
\end{align}
Next, in order to bound $I_3(z)$,
we consider the function
$h(z) \delequal 1 - \cos{z}.$
Since $h(z)$ is uniformly continuous on $[\theta,M]$, we can find some $N \in \mathbb N,\ t_0 > 1$, and a sequence of real numbers $\theta = x_0 > x_1 > \cdots > x_N = M$ such that
\begin{equation}\label{uContOfG_lCont}
    \begin{split}
        &\frac{x_{j-1}}{x_j} = t_0\qquad\forall j = 1,2,\cdots,N, \\
    & |h(x) - h(y)| < \delta\qquad\forall j = 1,2,\cdots,N,\ x,y \in [x_j,x_{j-1}]. 
    \end{split}
\end{equation}
In other words, we use a geometric sequence $\{x_0,x_1,\cdots,x_{N}\}$ to partition $[\theta,M]$ into $N$ intervals.
On any of these intervals, the fluctuations of $h(z) = 1 - \cos{z}$ is bounded by the constant $\delta$ fixed in \eqref{chooseDelta_Lcont}.
Now fix some $\Delta > 0$ such that (recall that $\epsilon > 0$ is prescribed in the statement of this proposition)
\begin{align}
    (1 - \Delta)t_0^{\alpha+\epsilon} > 1. \label{chooseCapDelta_LCont}
\end{align}
Since
$\mu[x,\infty)$ is regularly varying as $x \rightarrow 0$ with index $-(\alpha + 2\epsilon)$,
for $g(y) = \mu[1/y,\infty)$ we have $g\in \RV_{\alpha + 2\epsilon}(y)$ as $y \to \infty$. 
By Potter's bound (see Proposition~2.6 in \cite{resnick2007heavy}),there exists $\bar y_1 > 0$ such that
\begin{align}
    \frac{g(ty)}{g(y)} \geq (1 - \Delta)t^{\alpha + \epsilon}\qquad\forall y \geq \bar y_1,\ t \geq 1. \label{potterBound_lCont}
\end{align}
Meanwhile, define
$$\widetilde{g}(y) = y^\alpha,\qquad \nu_\alpha(dx) = \mathbf{I}_{(0,
\infty)}(x)\frac{dx}{x^{1+\alpha}}$$
and note that $\widetilde{g}(y) = \nu_\alpha(1/y,\infty)$.
Due to $g \in \RV_{\alpha + 2\epsilon}$, we can find some $\bar y_2 > 0$ such that
\begin{align}
    g(y)\geq \frac{t_0^\alpha - 1}{ (1 - \Delta)t_0^{\alpha + \epsilon} - 1 }\cdot  \widetilde{g}(y)\ \ \ \forall y \geq \bar y_2. \label{gBound_lCont}
\end{align}
Let $\widetilde{M} = \max\{ M/z_0, M\bar y_1, M\bar y_2 \}$. 
For any $|z| \geq \widetilde M$, we have $|z| \geq M/z_0$ and
$
\frac{|z|}{x_j} \geq
\frac{|z|}{M} \geq \bar y_1 \vee \bar y_2
$
for any $j = 0,1,\cdots,N$.
As a result,
for $z \in \mathbb{R}$ with $|z| \geq \widetilde{M}$ and any $j = 1,2,\cdots,N$,
\begin{align*}
    \mu[x_j/|z|, x_{j-1}/|z|) & = g(|z|/x_{j}) - g(|z|/x_{j-1}) 
    \qquad \text{ by definition of }g(y) = \mu[1/y,\infty)
    \\
    & = g(t_0|z|/x_{j-1}) - g(|z|/x_{j-1}) 
    \qquad \text{ due to $x_{j - 1} = t_0 x_j$; see \eqref{uContOfG_lCont}}
    \\
    & \geq g(|z|/x_{j-1})\cdot\Big( (1 - \Delta)t_0^{\alpha+\epsilon} - 1 \Big)\qquad \text{due to $\frac{|z|}{x_j} \geq \bar y_1 \vee \bar y_2$ and \eqref{potterBound_lCont}} \\
    & \geq \widetilde{g}(|z|/x_{j - 1})\cdot( t_0^\alpha - 1 )\qquad \text{due to \eqref{gBound_lCont}}.
\end{align*}
On the other hand,
\begin{align*}
    \nu_\alpha[x_j/|z|,x_{j - 1}/|z|) & = \widetilde{g}(|z|/x_j) - \widetilde{g}(|z|/x_{j-1}) = \widetilde{g}(|z|/x_{j - 1})\cdot( t_0^\alpha - 1 ).
\end{align*}{}
Therefore, given any $z \in \R$ such that $|z|\geq \widetilde{M}$,
we have $\mu\big(E_j(z)\big) \geq \nu_\alpha\big(E_j(z)\big)$ for all $j \in [N]$
where $E_j(z) = [x_j/|z|,x_{j-1}/|z|)$.
This leads to
\begin{align}
    & I_3(z)
     \nonumber \\
    & = \frac{1}{|z|^\alpha}\sum_{j = 1}^N\bigg[  \int_{E_j(z)}\Big( 1 - \cos(zx) \Big)\mu(dx) - \int_{E_j(z)}\Big( 1 - \cos(zx)\Big)\frac{dx}{x^{1 + \alpha}}  \bigg]\nonumber \\
    & \geq 
    \frac{1}{|z|^\alpha}\sum_{j = 1}^N
    \Big[
    \underline m_j \cdot \mu\big(E_j(z)\big) - \bar m_j\cdot\nu_\alpha\big(E_j(z)\big)
    \Big]
    \nonumber \\
    &
    \qquad \text{with $\bar m_j = \max\{ h(z):\ z \in [x_j,x_{j-1}]\},\ \underline m_j= \min\{ h(z):\ z \in [x_j,x_{j-1}]\}$}
    \nonumber \\
    & = 
    \frac{1}{|z|^\alpha}\sum_{j = 1}^N
    \Big[
    \underline m_j \cdot \mu\big(E_j(z)\big) - \underline m_j\cdot\nu_\alpha\big(E_j(z)\big)
    \Big]
    +
    \frac{1}{|z|^\alpha}\sum_{j = 1}^N
    \Big[
    \underline m_j \cdot \nu_\alpha\big(E_j(z)\big) - \bar m_j\cdot\nu_\alpha\big(E_j(z)\big)
    \Big]
    \nonumber \\
    & \geq 
    0 + \frac{1}{|z|^\alpha}\sum_{j = 1}^N
    \Big[
    \underline m_j \cdot \nu_\alpha\big(E_j(z)\big) - \bar m_j\cdot\nu_\alpha\big(E_j(z)\big)
    \Big]
    \qquad \text{due to }\mu\big(E_j(z)\big) \geq \nu_\alpha\big(E_j(z)\big)
    \nonumber \\
& \geq 
-\frac{\delta}{|z|^\alpha} \sum_{j = 1}^N\nu_\alpha\big(E_j(z)\big)
=
-\frac{\delta}{|z|^\alpha} \nu_c[\theta/|z|,M/|z|)
\qquad \text{due to \eqref{uContOfG_lCont}}
    \nonumber \\
& = -\frac{\delta}{|z|^\alpha}
\int_{ \theta/|z| }^{M/|z|}\frac{dx}{x^{1 + \alpha}}
\nonumber \\
&
\geq 
-\frac{\delta}{|z|^\alpha}
\int_{ \theta/|z| }^{\infty}\frac{dx}{x^{1 + \alpha}} = - \frac{\delta}{\alpha \theta^\alpha}
    \nonumber \\
    & \geq  -\frac{C_0}{4}\qquad\text{due to \eqref{chooseDelta_Lcont}.} \label{endTerm2RV_lCont}
\end{align}{}

Plugging \eqref{endTerm1RV_lCont}, \eqref{proof: bound term I 2, proposition: lipschitz cont, 1}, and \eqref{endTerm2RV_lCont} 
back into \eqref{proof, ineq 1, proposition: lipschitz cont, 1},
we have shown that for all $|z| \geq \widetilde M$,
\begin{align*}
    & \frac{1}{|z|^\alpha} \int_{(0,z_0)}\big( 1 - \cos(zx) \big)\mu(dx)
    \\ 
    & \geq -\frac{3C_0}{4} + 
     \frac{1}{|z|^\alpha}\int_0^\infty\Big( 1 - \cos(zx)\Big)\frac{dx}{x^{1 + \alpha}}
     \\ 
& = 
-\frac{3C_0}{4} + 
    \int_0^\infty\Big( 1 - \cos y\Big)\frac{dy}{y^{1 + \alpha}}
    \qquad \text{by setting }y = |z|x
     \\
     & = 
     -\frac{3C_0}{4} + C_0 = \frac{C_0}{4}
     \qquad\text{ by definition of }C_0 = \int_{0}^\infty (1 - \cos{y})\frac{dy}{y^{1 + \alpha}}.
\end{align*}
To conclude the proof of claim \eqref{goal, proposition: lipschitz cont, 1}, we set $\widetilde C = C_0/4.$
\end{proof}

Again, the proof of Proposition \ref{proposition: lipschitz cont, 2} makes use of the inversion formula.

\begin{proof}[Proof of Proposition \ref{proposition: lipschitz cont, 2}]
\linksinpf{proposition: lipschitz cont, 2}
Let us denote the characteristic functions of $Y^\prime(t)$ and $Y(t)$ by ${\varphi}_t$ and $\widetilde\varphi_t$, respectively. 
Repeating the arguments using complex conjugates in \eqref{normOfCF_lcont}, we obtain
\begin{align*}
    |\widetilde{\varphi}_t(z)| = \exp\Big( -t\int_{|x|<b^N}\big(1 - \cos(zx)\big)\mu(dx)\Big).
\end{align*}{}
As for $\varphi_t$, using proposition 14.9 in \cite{sato1999levy}, we get
\begin{align}
    |\varphi_t(z)| = \exp\Big(-t|z|^\alpha\eta(z)\Big)
\end{align}{}
where $\eta(z)$ is a non-negative function continuous on $\mathbb{R}\symbol{92}\{0\}$ satisfying $\eta(bz)=\eta(z)$ and
\begin{align*}
    \eta(z) = \frac{ \int_\mathbb{R}\big(1 - \cos(zx)\big)\mu(dx) }{|z|^\alpha}\qquad \forall z\neq 0.
\end{align*}
This implies $\eta(z) = \eta(-z)$ for all $z \neq 0$.
Furthermore, we claim the existence of some $c > 0$ such that
\begin{align}
    \eta(z) \geq c\qquad \forall z \in [1,b].
    \label{goal, proposition: lipschitz cont, 2}
\end{align}
Then due to the self-similarity of $\mu$ (i.e., $\eta(bz)=\eta(z)$), we have $\eta(z) \geq c$ for all $z \neq 0$.
In the meantime, note that
\begin{align*}
    \frac{1}{|z|^\alpha}\int_{|x|\geq b^N}\big(1 - \cos(zx)\big)\mu(dx)\leq \frac{\mu\{x:|x|\geq b^N\}}{|z|^\alpha}.
\end{align*}{}
By picking $M > 0$ large enough, it holds for any $|z|\geq M$ that
\begin{align}
    \frac{1}{|z|^\alpha}\int_{|x|\geq b^N}\big(1 - \cos(zx)\big)\mu(dx) \leq \frac{c}{2}.
    \label{fact 2, proposition: lipschitz cont, 2}
\end{align}
Therefore, for any $|z| \geq M$,
\begin{align*}
    \int_{|x| < b^N}\big(1 - \cos(zx)\big)\mu(dx)
    & = \int_{x \in \R}\big(1 - \cos(zx)\big)\mu(dx) - \int_{|x|\geq b^N}\big(1 - \cos(zx)\big)\mu(dx)
    \\
    & =
    \eta(z) \cdot |z|^\alpha - \int_{|x|\geq b^N}\big(1 - \cos(zx)\big)\mu(dx)
    \\ 
    & \geq c|z|^\alpha - \frac{c}{2}|z|^\alpha  = \frac{c}{2}|z|^\alpha
    \qquad \text{using \eqref{goal, proposition: lipschitz cont, 2} and \eqref{fact 2, proposition: lipschitz cont, 2}},
\end{align*}
and hence
$
|\widetilde{\varphi}_t(z)|\leq \exp\big(-\frac{c}{2}t|z|^\alpha\big)
$
for all $|z| \geq M$.
Applying inversion formula, we get (for any $t > 0$)
\begin{align*}
    \norm{ f_{Y(t)} }_\infty & \leq \frac{1}{2\pi} \int |\widetilde \varphi_t(z)|dz
    \\ 
    & \leq \frac{1}{2\pi}\bigg[
    2M + \int_{|z| \geq M} |\widetilde \varphi_t(z)|dz
    \bigg]
    \\ 
    & \leq \frac{M}{\pi} + \frac{1}{2\pi}\int\exp\bigg(-\frac{c}{2}t|z|^\alpha\bigg) dz
    \\ 
    & = 
    \frac{M}{\pi} + \frac{1}{2\pi} \cdot \frac{1}{t^{1/\alpha}} \int \exp\bigg(-\frac{c}{2} |x|^\alpha\bigg)dx
    \qquad \text{using }x = t^{1/\alpha} \cdot z
    \\ 
    & \leq \frac{M}{\pi} + \frac{C_1}{t^{1/\alpha}}
    \qquad \text{ where }C_1 = \frac{1}{2\pi} \int \exp\bigg(-\frac{c}{2} |x|^\alpha\bigg)dx.
\end{align*}
To conclude the proof, we set $C = \frac{M}{\pi} + C_1$.
Now it only remains to prove claim \eqref{goal, proposition: lipschitz cont, 2}.

\medskip
\noindent\textbf{Proof of Claim }\eqref{goal, proposition: lipschitz cont, 2}

We proceed with a proof by contradiction. 
If $\inf_{z \in [1,b]}\eta(z) = 0$,
then by continuity of $\eta(z)$,
there exists some $z \in [1,b]$ such that
\begin{align*}
    \int_\mathbb{R}\big(1 - \cos(zx)\big)\mu(dx) = 0.
\end{align*}{}
Now for any $\epsilon > 0$, define the following sets:
\begin{align}
    S & = \{ x \in \mathbb{R}:\ 1 - \cos(zx) > 0  \} = \mathbb{R}\symbol{92}\{\frac{2\pi}{z}k:\ k \in \mathbb{Z}\};\nonumber \\
    S_\epsilon & = \{ x \in \mathbb{R}:\ 1 - \cos(zx) \geq \epsilon  \}.\nonumber
\end{align}{}
Observe that
\begin{itemize}
    \item For any $\epsilon > 0$, we have
    $
    \epsilon \cdot\mu(S_\epsilon) \leq \int_{S_\epsilon}\big(1 - \cos(zx)\big)\mu(dx)\leq \int_\mathbb{R}\big(1 - \cos(zx)\big)\mu(dx) = 0,
    $
    which implies $\mu(S_\epsilon) = 0$;
    
    \item Meanwhile, $\lim_{\epsilon \rightarrow 0}\mu(S_\epsilon) = \mu(S) = 0$.
\end{itemize}
Together with the fact that $\mu(\mathbb{R})>0$ (so that the process is non-trivial), there must be some $m \in \mathbb{Z},\ \delta > 0$ such that
$$\mu(\{ \frac{2\pi}{z}m \}) = \delta > 0.$$
Besides, from $\mu(S) = 0$ we know that $\mu\big(\{-\frac{2\pi}{z},\frac{2\pi}{z}\}\setminus\{0\}\big) = 0$.
However, by definition of semi-stable processes in \eqref{semiStable_scaleInvariance_lCont} we know that $\mu = b^{-\alpha}T_b\mu$ where the transformation $T_r$ ($\forall r > 0$) onto a Borel measure $\rho$ on $\mathbb{R}$ is defined as $(T_r \rho)(B) = \rho(r^{-1}B)$. This implies
\begin{align*}
    \mu(\{\frac{2\pi m}{z}b^{-k}\}) > 0 \ \ \forall k = 1,2,3,\cdots
\end{align*}
which would contradict $\mu\big(\{-\frac{2\pi}{z},\frac{2\pi}{z}\}\setminus\{0\}\big) = 0$ eventually for $k$ large enough.
This concludes the proof of $\eta(z) > 0$ for all $z \in [1,b]$.
\end{proof}

\bibliographystyle{abbrv} 
\bibliography{bib_appendix} 

\begin{thebibliography}{10}

\bibitem{10.1214/aoap/1177004597}
S.~Asmussen, P.~Glynn, and J.~Pitman.
\newblock {Discretization Error in Simulation of One-Dimensional Reflecting
  Brownian Motion}.
\newblock {\em The Annals of Applied Probability}, 5(4):875 -- 896, 1995.

\bibitem{asmussen_kroese_2006}
S.~Asmussen and D.~P. Kroese.
\newblock Improved algorithms for rare event simulation with heavy tails.
\newblock {\em Advances in Applied Probability}, 38(2):545–558, 2006.

\bibitem{asmussen2001approximations}
S.~Asmussen and J.~Rosi{\'n}ski.
\newblock Approximations of small jumps of l{\'e}vy processes with a view
  towards simulation.
\newblock {\em Journal of Applied Probability}, 38(2):482--493, 2001.

\bibitem{BASSAMBOO2007251}
A.~Bassamboo, S.~Juneja, and A.~Zeevi.
\newblock On the inefficiency of state-independent importance sampling in the
  presence of heavy tails.
\newblock {\em Operations Research Letters}, 35(2):251--260, 2007.

\bibitem{bianchi2011tempered}
M.~L. Bianchi, S.~T. Rachev, Y.~S. Kim, and F.~J. Fabozzi.
\newblock Tempered infinitely divisible distributions and processes.
\newblock {\em Theory of Probability \& Its Applications}, 55(1):2--26, 2011.

\bibitem{10.1214/07-AAP485}
J.~Blanchet and P.~Glynn.
\newblock {Efficient rare-event simulation for the maximum of heavy-tailed
  random walks}.
\newblock {\em The Annals of Applied Probability}, 18(4):1351 -- 1378, 2008.

\bibitem{blanchet2008efficient}
J.~Blanchet, P.~Glynn, and J.~Liu.
\newblock Efficient rare event simulation for heavy-tailed multiserver queues.
\newblock Technical report, Department of Statistics, Columbia University,
  2008.

\bibitem{10.1145/2517451}
J.~Blanchet, H.~Hult, and K.~Leder.
\newblock Rare-event simulation for stochastic recurrence equations with
  heavy-tailed innovations.
\newblock {\em ACM Trans. Model. Comput. Simul.}, 23(4), dec 2013.

\bibitem{blanchet_liu_2008}
J.~H. Blanchet and J.~Liu.
\newblock State-dependent importance sampling for regularly varying random
  walks.
\newblock {\em Advances in Applied Probability}, 40(4):1104–1128, 2008.

\bibitem{Borak2011}
S.~Borak, A.~Misiorek, and R.~Weron.
\newblock {\em Models for heavy-tailed asset returns}, pages 21--55.
\newblock Springer Berlin Heidelberg, Berlin, Heidelberg, 2011.

\bibitem{boxma2018linear}
O.~Boxma, E.~Cahen, D.~Koops, and M.~Mandjes.
\newblock Linear networks: rare-event simulation and markov modulation.
\newblock {\em Methodology and Computing in Applied Probability}, 2019.

\bibitem{boyarchenko2023efficientevaluationjointpdf}
S.~Boyarchenko and S.~Levendorskii.
\newblock Efficient evaluation of joint pdf of a l\'evy process, its extremum,
  and hitting time of the extremum, 2023.

\bibitem{boyarchenko2023simulationlevyprocessextremum}
S.~Boyarchenko and S.~Levendorskii.
\newblock Simulation of a l\'evy process, its extremum, and hitting time of the
  extremum via characteristic functions, 2023.

\bibitem{bucklew_ney_sadowsky_1990}
J.~A. Bucklew, P.~Ney, and J.~S. Sadowsky.
\newblock Monte carlo simulation and large deviations theory for uniformly
  recurrent markov chains.
\newblock {\em Journal of Applied Probability}, 27(1):44–59, 1990.

\bibitem{4fc30f5f-894d-3c6c-9b15-01f8f8d74820}
P.~Carr, H.~Geman, D.~Madan, and M.~Yor.
\newblock The fine structure of asset returns: An empirical investigation.
\newblock {\em The Journal of Business}, 75(2):305--332, 2002.

\bibitem{https://doi.org/10.1111/1467-9965.00020}
P.~Carr, H.~Geman, D.~B. Madan, and M.~Yor.
\newblock Stochastic volatility for lévy processes.
\newblock {\em Mathematical Finance}, 13(3):345--382, 2003.

\bibitem{10.3150/23-BEJ1590}
J.~I.~G. C{\'a}zares, A.~Kohatsu-Higa, and A.~Mijatović.
\newblock {Joint density of the stable process and its supremum: Regularity and
  upper bounds}.
\newblock {\em Bernoulli}, 29(4):3443 -- 3469, 2023.

\bibitem{10.1214/20-EJP503}
J.~I.~G. C{\'a}zares, A.~Mijatović, and G.~U. Bravo.
\newblock {$\varepsilon $-strong simulation of the convex minorants of stable
  processes and meanders}.
\newblock {\em Electronic Journal of Probability}, 25(none):1 -- 33, 2020.

\bibitem{10.1214/11-AOP708}
L.~Chaumont.
\newblock {On the law of the supremum of Lévy processes}.
\newblock {\em The Annals of Probability}, 41(3A):1191 -- 1217, 2013.

\bibitem{chen2019efficient}
B.~Chen, J.~Blanchet, C.-H. Rhee, and B.~Zwart.
\newblock Efficient rare-event simulation for multiple jump events in regularly
  varying random walks and compound poisson processes.
\newblock {\em Mathematics of Operations Research}, 44(3):919--942, 2019.

\bibitem{cohen2022covid}
J.~E. Cohen, R.~A. Davis, and G.~Samorodnitsky.
\newblock Covid-19 cases and deaths in the united states follow taylor’s law
  for heavy-tailed distributions with infinite variance.
\newblock {\em Proceedings of the National Academy of Sciences},
  119(38):e2209234119, 2022.

\bibitem{coutinpontierngom2018}
L.~Coutin, M.~Pontier, and W.~Ngom.
\newblock Joint distribution of a lévy process and its running supremum.
\newblock {\em Journal of Applied Probability}, 55(2):488–512, 2018.

\bibitem{cázares2023fast}
J.~I.~G. Cázares, F.~Lin, and A.~Mijatović.
\newblock Fast exact simulation of the first passage of a tempered stable
  subordinator across a non-increasing function, 2023.

\bibitem{10.1214/10-AAP695}
S.~Dereich.
\newblock {Multilevel Monte Carlo algorithms for Lévy-driven SDEs with
  Gaussian correction}.
\newblock {\em The Annals of Applied Probability}, 21(1):283 -- 311, 2011.

\bibitem{DEREICH20111565}
S.~Dereich and F.~Heidenreich.
\newblock A multilevel monte carlo algorithm for lévy-driven stochastic
  differential equations.
\newblock {\em Stochastic Processes and their Applications}, 121(7):1565--1587,
  2011.

\bibitem{Dia_Lamberton_2011}
E.~H.~A. Dia and D.~Lamberton.
\newblock Connecting discrete and continuous lookback or hindsight options in
  exponential lévy models.
\newblock {\em Advances in Applied Probability}, 43(4):1136–1165, 2011.

\bibitem{10.1145/1243991.1243995}
P.~Dupuis, K.~Leder, and H.~Wang.
\newblock Importance sampling for sums of random variables with regularly
  varying tails.
\newblock {\em ACM Trans. Model. Comput. Simul.}, 17(3):14–es, jul 2007.

\bibitem{dupuis2007dynamic}
P.~Dupuis, A.~D. Sezer, and H.~Wang.
\newblock {Dynamic importance sampling for queueing networks}.
\newblock {\em The Annals of Applied Probability}, 17(4):1306 -- 1346, 2007.

\bibitem{durrett2019probability}
R.~Durrett.
\newblock {\em Probability: Theory and Examples}.
\newblock Cambridge Series in Statistical and Probabilistic Mathematics.
  Cambridge University Press, 2019.

\bibitem{embrechts2013modelling}
P.~Embrechts, C.~Kl{\"u}ppelberg, and T.~Mikosch.
\newblock {\em Modelling extremal events: for insurance and finance},
  volume~33.
\newblock Springer Science \& Business Media, 2013.

\bibitem{FERREIROCASTILLA2014985}
A.~Ferreiro-Castilla, A.~Kyprianou, R.~Scheichl, and G.~Suryanarayana.
\newblock Multilevel monte carlo simulation for lévy processes based on the
  wiener–hopf factorisation.
\newblock {\em Stochastic Processes and their Applications}, 124(2):985--1010,
  2014.

\bibitem{gil08}
M.~B. Giles.
\newblock Multilevel monte carlo path simulation.
\newblock {\em Operations Research}, 56(3):607--617, 2008.

\bibitem{giles2017multilevel}
M.~B. Giles and Y.~Xia.
\newblock Multilevel monte carlo for exponential l{\'e}vy models.
\newblock {\em Finance and Stochastics}, 21(4):995--1026, 2017.

\bibitem{gonzalez2022simulation}
J.~Gonz{\'a}lez~C{\'a}zares and A.~Mijatovi{\'c}.
\newblock Simulation of the drawdown and its duration in l{\'e}vy models via
  stick-breaking gaussian approximation.
\newblock {\em Finance and Stochastics}, 26(4):671--732, 2022.

\bibitem{cazares2018geometrically}
J.~I. Gonz\'{a}lez~C\'{a}zares, A.~Mijatovi\'{c}, and G.~Uribe~Bravo.
\newblock Geometrically convergent simulation of the extrema of lévy
  processes.
\newblock {\em Mathematics of Operations Research}, 47(2):1141--1168, 2022.

\bibitem{gonzalez2019exact}
J.~I. González~Cázares, A.~Mijatović, and G.~U. Bravo.
\newblock Exact simulation of the extrema of stable processes.
\newblock {\em Advances in Applied Probability}, 51(4):967–993, 2019.

\bibitem{gudmundsson_hult_2014}
T.~Gudmundsson and H.~Hult.
\newblock Markov chain monte carlo for computing rare-event probabilities for a
  heavy-tailed random walk.
\newblock {\em Journal of Applied Probability}, 51(2):359–376, 2014.

\bibitem{pmlr-v139-gurbuzbalaban21a}
M.~Gurbuzbalaban, U.~Simsekli, and L.~Zhu.
\newblock The heavy-tail phenomenon in sgd.
\newblock In M.~Meila and T.~Zhang, editors, {\em Proceedings of the 38th
  International Conference on Machine Learning}, volume 139 of {\em Proceedings
  of Machine Learning Research}, pages 3964--3975. PMLR, 18--24 Jul 2021.

\bibitem{hei01}
S.~Heinrich.
\newblock Multilevel monte carlo methods.
\newblock In S.~Margenov, J.~Wa{\'{s}}niewski, and P.~Yalamov, editors, {\em
  Large-Scale Scientific Computing}, pages 58--67, Berlin, Heidelberg, 2001.
  Springer Berlin Heidelberg.

\bibitem{hesterberg1995}
T.~Hesterberg.
\newblock Weighted average importance sampling and defensive mixture
  distributions.
\newblock {\em Technometrics}, 37(2):185--194, 1995.

\bibitem{hodgkinson2021multiplicative}
L.~Hodgkinson and M.~Mahoney.
\newblock Multiplicative noise and heavy tails in stochastic optimization.
\newblock In {\em International Conference on Machine Learning}, pages
  4262--4274. PMLR, 2021.

\bibitem{hult2016exact}
H.~Hult, S.~Juneja, and K.~Murthy.
\newblock Exact and efficient simulation of tail probabilities of heavy-tailed
  infinite series.
\newblock 2016.

\bibitem{10.1214/10-AAP746}
A.~Kuznetsov, A.~E. Kyprianou, J.~C. Pardo, and K.~van Schaik.
\newblock {A Wiener–Hopf Monte Carlo simulation technique for Lévy
  processes}.
\newblock {\em The Annals of Applied Probability}, 21(6):2171 -- 2190, 2011.

\bibitem{10.1214/11-AOP719}
M.~Kwaśnicki, J.~Małecki, and M.~Ryznar.
\newblock {Suprema of Lévy processes}.
\newblock {\em The Annals of Probability}, 41(3B):2047 -- 2065, 2013.

\bibitem{li:hal-01891760}
Y.~Li.
\newblock {Queuing theory with heavy tails and network traffic modeling}.
\newblock working paper or preprint, Oct. 2018.

\bibitem{10.1214/18-EJS1456}
E.~Mariucci and M.~Rei{\ss}.
\newblock {Wasserstein and total variation distance between marginals of Lévy
  processes}.
\newblock {\em Electronic Journal of Statistics}, 12(2):2482 -- 2514, 2018.

\bibitem{michna2012formula}
Z.~Michna.
\newblock Formula for the supremum distribution of a spectrally positive l\'evy
  process, 2012.

\bibitem{10.1214/ECP.v18-2236}
Z.~Michna.
\newblock {Explicit formula for the supremum distribution of a spectrally
  negative stable process}.
\newblock {\em Electronic Communications in Probability}, 18(none):1 -- 6,
  2013.

\bibitem{michna2014distribution}
Z.~Michna, Z.~Palmowski, and M.~Pistorius.
\newblock The distribution of the supremum for spectrally asymmetric l\'evy
  processes, 2014.

\bibitem{https://doi.org/10.1111/mafi.12055}
A.~Mijatović and P.~Tankov.
\newblock A new look at short-term implied volatility in asset price models
  with jumps.
\newblock {\em Mathematical Finance}, 26(1):149--183, 2016.

\bibitem{doi:10.1287/13-SSY114}
K.~R.~A. Murthy, S.~Juneja, and J.~Blanchet.
\newblock State-independent importance sampling for random walks with regularly
  varying increments.
\newblock {\em Stochastic Systems}, 4(2):321--374, 2014.

\bibitem{pitman2012convex}
J.~Pitman and G.~U. Bravo.
\newblock {The convex minorant of a Lévy process}.
\newblock {\em The Annals of Probability}, 40(4):1636 -- 1674, 2012.

\bibitem{resnick2007heavy}
S.~I. Resnick.
\newblock {\em Heavy-tail phenomena: probabilistic and statistical modeling}.
\newblock Springer Science \& Business Media, 2007.

\bibitem{rhee2019sample}
C.-H. Rhee, J.~Blanchet, B.~Zwart, et~al.
\newblock Sample path large deviations for l{\'e}vy processes and random walks
  with regularly varying increments.
\newblock {\em The Annals of Probability}, 47(6):3551--3605, 2019.

\bibitem{rhee2015unbiased}
C.-H. Rhee and P.~W. Glynn.
\newblock Unbiased estimation with square root convergence for sde models.
\newblock {\em Operations Research}, 63(5):1026--1043, 2015.

\bibitem{ROSINSKI2007677}
J.~Rosiński.
\newblock Tempering stable processes.
\newblock {\em Stochastic Processes and their Applications}, 117(6):677--707,
  2007.

\bibitem{risks10080148}
P.~Sabino.
\newblock Pricing energy derivatives in markets driven by tempered stable and
  cgmy processes of ornstein–uhlenbeck type.
\newblock {\em Risks}, 10(8), 2022.

\bibitem{sato1999levy}
K.-i. Sato, S.~Ken-Iti, and A.~Katok.
\newblock {\em L{\'e}vy processes and infinitely divisible distributions}.
\newblock Cambridge university press, 1999.

\bibitem{TORRISI2004225}
G.~Torrisi.
\newblock Simulating the ruin probability of risk processes with delay in claim
  settlement.
\newblock {\em Stochastic Processes and their Applications}, 112(2):225--244,
  2004.

\bibitem{10.5555/3466184.3466229}
X.~Wang and C.-H. Rhee.
\newblock Rare-event simulation for multiple jump events in heavy-tailed
  l\'{e}vy processes with infinite activities.
\newblock In {\em Proceedings of the Winter Simulation Conference}, WSC '20,
  page 409–420. IEEE Press, 2021.

\bibitem{wang2023large}
X.~Wang and C.-H. Rhee.
\newblock Large deviations and metastability analysis for heavy-tailed
  dynamical systems, 2023.

\bibitem{10.5555/3643142.3643148}
X.~Wang and C.-H. Rhee.
\newblock Importance sampling strategy for heavy-tailed systems with
  catastrophe principle.
\newblock In {\em Proceedings of the Winter Simulation Conference}, WSC '23,
  page 76–90. IEEE Press, 2024.

\end{thebibliography}

\newpage
\appendix

\section{Barrier Option Pricing}
\label{sec: barrier option pricing}

\subsection{Problem Setting}

This section considers the estimation of probabilities $P(A_n)$ with $A_n = \{\bar X_n \in A\}$ and
$$
    A \delequal 
    \{
        \xi \in \D:\ 
        \xi(1) \leq -b,\ \sup_{t \leq 1}\xi(t) + ct \geq a
    \},
$$
which corresponds to rare-event simulation in the context of down-and-in option.
Here, we assume that $a,\ b > 0$ and $c < a$.
We consider the two-sided case in Assumption~\ref{assumption: heavy tailed process}.
That is, $X(t)$ is a centered L\'evy process with L\'evy measures $\nu$,
and there exists some $\alpha,\ \alpha^\prime > 1$ such that $\nu[x,\infty) \in \RV_{-\alpha}(x)$ and $\nu(-\infty,-x] \in \RV_{-\alpha^\prime}(x)$ as $x \to \infty$.
Also, we impose an alternative version of Assumption~\ref{assumption: holder continuity strengthened on X < z t} throughout.
Let $X^{(-z,z)}(t)$ be the L\'evy process with 
with generating triplet $(c_X,\sigma,\nu|_{(-z,z)})$.
That is, $X^{(-z,z)}(t)$ is a modulated version of $X$ where all jumps with size larger than $z$ are removed.
\begin{assumption}
\label{assumption: holder continuity strengthened on X < z t, alternative }
    There exist ${z_0},\ {C},\ {\lambda} > 0$ such that
    \begin{align*}
        \P\big(X^{(-z,z)}(t) \in [x, x + \delta]\big) \leq \frac{C\delta}{t^\lambda\wedge 1}\qquad \forall z \geq z_0,\ t > 0,\ x \in \R,\ \delta > 0.
    \end{align*}
\end{assumption}

\subsection{Importance Sampling Algorithm}

Below, we present the design of the importance sampling algorithm.
For any $\xi \in \D$ and $t \in (0,1]$, let $\Delta \xi(t) = \xi(t) - \xi(t-)$ be the discontinuity in $\xi$ at time $t$,
and we set $\Delta \xi(0) \equiv 0$.
Let
\begin{align*}
    B^\gamma = 
    \Big\{\xi \in \D:\ 
        \#\{ t \in [0,1]:\ \Delta\xi(t) \geq \gamma \} \geq 1,
        \ 
        \#\{ t \in [0,1]:\ \Delta\xi(t) \leq -\gamma \} \geq 1
    \Big\}
\end{align*}
and let $B^\gamma_n = \{\bar X_n \in B^\gamma_n\}$.
Intuitively speaking, on event $B^\gamma_n$ there is at least one upward and one downward ``large'' jump in $\bar X_n$,
where $\gamma > 0$ is understood as the threshold for jump sizes to be considered ``large''.

Fix some $w \in (0,1)$, and let
$$
    \Q_n(\cdot)
    =
    w\P(\cdot) + (1-w)\P(\ \cdot\ |B^\gamma_n).
$$
The algorithm samples
$$
    {L_n} = Z_n \frac{d\P}{d\Q_n} 
    =
    \frac{Z_n}{ w + \frac{1 - w}{ \P(B^\gamma_n) }\mathbf{I}_{ B^\gamma_n } }
$$
under $\Q_n$. Now, we discuss the design of $Z_n$ to ensure the strong efficiency of $L_n$.
Analogous to the decomposition in \eqref{def: process J n Xi n}, let
\begin{align*}
     {J_n(t)} & = \sum_{s \in [0,t]}\Delta X(s) \mathbf{I}\big( |\Delta X(s)| \geq n\gamma \big),
        \\
    {\Xi_n(t)} & = X(t) - J_n(t) = X(t) - \sum_{s \in [0,t]}\Delta X(s) \mathbf{I}\big( |\Delta X(s)| \geq n\gamma \big).
\end{align*}
Let ${\bar J_n(t)} = \frac{1}{n}J_n(nt)$, $\bar J_n = \{\bar J_n(t):\ t \in [0,1]\}$,
${\bar \Xi_n(t)} = \frac{1}{n}\Xi_n(nt)$, and $\bar \Xi_n = \{\bar \Xi_n(t):\ t \in [0,1]\}$.
Meanwhile, set 
\begin{align*}
    {M_c(t)} \delequal \sup_{s \leq t}X(s) + cs,
    \qquad 
    {Y^*_{n;c}} \delequal \mathbf{I}\big( M_c(n) \geq na,\ X(n) \leq -nb \big),
\end{align*}
We have $\mathbf{I}_{A_n} = Y^*_{n;c}$.
Under the convention $\hat Y^{-1}_n \equiv 0$, consider estimators $Z_n$ of form
\begin{align}
    {Z_n} = \sum_{m = 0}^{ \tau}\frac{ \hat Y^m_{n;c} - \hat Y^{m-1}_{n;c} }{ \P(\tau \geq m) }
    \label{def: Z_n, option pricing}
\end{align}
where ${\tau}$ is $\text{Geom}(\rho)$ for some ${\rho} \in (0,1)$ and is independent of everything else.
Analogous to Proposition~\ref{proposition: design of Zn},
the following result provides sufficient conditions on $\hat Y^{m}_{n;c}$ for $L_n$ to attain strong efficiency.

\begin{proposition}
\label{proposition: design of Zn, barrier option pricing}
Let $C_0 > 0$, $\rho_0 \in (0,1)$, $\mu > \alpha + \alpha^\prime - 2$, and $\bar m \in \mathbb N$. Suppose that
\begin{align}
    \P\Big(Y^*_{n;c} \neq \hat Y^m_{n;c}\ \Big|\ \mathcal{D}^+(\bar J_n) = k,\ \mathcal{D}^-(\bar J_n) = k^\prime\Big) \leq C_0 \rho^m_0 \cdot (k+k^\prime+1)\qquad \forall k,\ k^\prime \geq 0,\ n\geq 1,\  m \geq \bar m
    \label{condition 1, proposition: design of Zn, option pricing}
\end{align}
where $\mathcal{D}^+(\xi)$ and $\mathcal{D}^-(\xi)$ count the number of discontinuities of positive and negative sizes in $\xi$, respectively.
Besides, suppose that for all $\Delta \in (0,1)$,
\begin{align}
    \P\Big(Y^*_{n;c} \neq \hat Y^m_{n;c},\ \bar X_n \notin A^\Delta\ \Big|\ \mathcal{D}^+(\bar J_n) = 0\text{ or }\mathcal{D}^-(\bar J_n) = 0\Big)
    \leq \frac{C_0 \rho^m_0}{ \Delta^2 n^{\mu} }\qquad \forall n\geq 1,\ m \geq 0
    \label{condition 2, proposition: design of Zn, option pricing}
\end{align}
where 
${A^\Delta} = \big\{\xi \in \mathbb{D}: \sup_{t \in [0,1]}\xi(t) + ct\geq a - \Delta,\ \xi(1) \leq -b \big\}$.
Then given $\rho \in (\rho_0,1)$,
there exists some $\bar \gamma = \bar \gamma(\rho) \in (0,b)$
such that
for all $\gamma \in (0,\bar \gamma)$,
the estimators $(L_n)_{n \geq 1}$ are \textbf{unbiased and strongly efficient} for $\P(A_n) = \P(\bar X_n \in A)$ under the importance sampling distribution $\Q_n$.
\end{proposition}

The proof is almost identical to that of Proposition~\ref{proposition: design of Zn}.
In particular, the proof requires that
$$
    \P(A_n) = \bm{O}\big(n\nu[n,\infty)\cdot n \nu(-\infty,-n]\big)
$$
and that, for any $\beta > 0$, it holds for all $\gamma$ small enough that
$$
    \P(A^\Delta_n \setminus B^\gamma_n)  = \bm{o}(n^\beta)
$$
where $A^\Delta_n = \{\bar X_n \in A^\Delta\}$.
These can be obtained directly using sample path large deviations for heavy-tailed L\'evy processes in Result~\ref{result: LD of Levy, two-sided}.
The Proposition~\ref{proposition: design of Zn, barrier option pricing} is then established by repeating the arguments in 
Proposition~\ref{proposition: design of Zn} using Result~\ref{resultDebias} for randomized debiasing technique.

\subsection{Construction of $\hat Y^m_{n;c}$}

Next, we describe the construction of $\hat Y^m_{n;c}$ that can satisfy the conditions in Proposition~\ref{proposition: design of Zn, barrier option pricing}.
Specifically, we consider the case where ARA is involved. 
Let
$$
    \Xi_{n;c}(t) \delequal \Xi_n(t) + ct.
$$
Under both $\P$ and $\Q_n$,  $\Xi_{n;c}(t)$ admits the law of a L\'evy process with generating triplet  $(c_X + c,\sigma,\nu|_{ (-n\gamma,n\gamma) })$.
This leads to the L\'evy-Ito decomposition
\begin{align*}
    \Xi_{n,c}(t) & \distequal 
    (c_X+c)t + \sigma B(t) + 
    \underbrace{\sum_{s \leq t}\Delta X(s) \mathbf{I}\Big(\Delta X(s) \in (-n\gamma,-1]\cup [1,n\gamma)\Big)}_{ \delequal J_{n,-1}(t)}
    \\
    &
    + \sum_{m \geq 0}\Bigg[
    \underbrace{\sum_{s \leq t}\Delta X(s) \mathbf{I}\Big(|\Delta X(s)| \in[\kappa_{n,m},\kappa_{n,m-1})\Big)
    -
    t \cdot \nu\Big( (-\kappa_{n,m-1},-\kappa_{n,m}] \cup [\kappa_{n,m},\kappa_{n,m-1})\Big)
    }_{ \delequal J_{n,m}(t) }
    \Bigg]
\end{align*}
with $\kappa_{n,m}$ defined in \eqref{def: kappa n m}.
Besides, let $\bar \sigma^2(\cdot)$ be defined as in \eqref{def bar sigma}.
For each $n \geq 1$ and $m \geq 0$, consider the approximation 
\begin{align*}
    {\Breve{\Xi}^{m}_{n;c}(t)} & \delequal (c_X+c)t + \sigma B(t) + \sum_{q = -1}^{m}J_{n,q}(t)
        +
        \sum_{q \geq m + 1}\sqrt{\bar \sigma^2(\kappa_{n,q-1}) - \bar \sigma^2(\kappa_{n,q})} \cdot W^{q}(t)
\end{align*}
where $(W^m)_{m \geq 1}$ is a sequence of iid copies of standard Brownian motions independent of everything else.

Next, we discuss how to apply SBA and construct approximators $\hat Y^m_{n;c}$'s in \eqref{def: Z_n, option pricing}.
Let $\zeta_k(t) = \sum_{i = 1}^k z_i \mathbf{I}_{[u_i,n]}(t)$ be a piece-wise step function with $k$ jumps over $(0,n]$,
where $0 < u_1 < u_2 < \ldots < u_k \leq n$,
and $z_i \neq 0$ for each $i \in [k]$.
Recall that the jump times in $\zeta_k$ leads to a partition of $[0,n]$ of $(I_i)_{i \in [k+1]}$ defined in \eqref{def, partition of time line, I_i}.
For any $I_i$, let the sequence $l^{(i)}_j$'s be defined as in \eqref{defStickLength1}--\eqref{defStickLength2}.
Conditioning on $(l^{(i)}_j)_{j \geq 1}$, one can then sample ${\xi^{(i),m}_{j;c},\xi^{(i)}_{j;c}}$ using
\begin{align}
    \big(\xi^{(i)}_{j;c},\xi^{(i),0}_{j;c},\xi^{(i),1}_{j;c},\xi^{(i),2}_{j;c},\ldots) \distequal \Big(\Xi_{n;c}(l^{(i)}_{j}),\ \Breve{\Xi}^0_{n;c}(l^{(i)}_{j}),\ \Breve{\Xi}^1_{n;c}(l^{(i)}_{j}),\ \Breve{\Xi}^2_{n;c}(l^{(i)}_{j}),\ldots\Big).
    \nonumber
\end{align}
The coupling in \eqref{prelim: SBA, coupling} then implies
\begin{align*}
    & \Big( 
         \Xi_{n;c}(u_i) - \Xi_{n;c}(u_{i-1}),\ 
         \sup_{t \in I_i}\Xi_{n;c}(t) - \Xi_{n;c}(u_{i-1}),\ 
         \breve \Xi^{0}_{n;c}(u_i) - \breve \Xi^{0}_{n;c}(u_{i-1}),\ 
         \sup_{t \in I_i}\breve \Xi^{0}_{n;c}(t) - \breve \Xi^{0}_{n;c}(u_{i-1}),
         \\
         &\qquad\qquad\qquad\qquad \qquad\qquad\qquad\qquad 
         \breve \Xi^{1}_{n;c}(u_i) - \breve \Xi^{1}_{n;c}(u_{i-1}),\ 
         \sup_{t \in I_i}\breve \Xi^{1}_{n;c}(t) - \breve \Xi^{1}_{n;c}(u_{i-1}),
    \ldots
    \Big)
    \\
    &\qquad \distequal 
    \Big(
    \sum_{j \geq 1}\xi^{(i)}_{j;c},\ \sum_{j \geq 1}(\xi^{(i)}_{j;c})^+,\
    \sum_{j \geq 1}\xi^{(i),0}_{j;c},\ \sum_{j \geq 1}(\xi^{(i),0}_{j;c})^+,\
    \sum_{j \geq 1}\xi^{(i),1}_{j;c},\ \sum_{j \geq 1}(\xi^{(i),1}_{j;c})^+,\ldots  
    \Big).
\end{align*}
Now, we define
\begin{align}
    {\hat M^{(i),m}_{n;c}(\zeta_k)} = \sum_{j = 1}^{m + \ceil{\log_2(n^d)}  }( \xi^{(i),m}_{j;c})^+
    \nonumber
\end{align}
as an approximation to
$
{M_{n;c}^{(i),*}(\zeta_k)} = \sup_{ t \in I_i}\Xi_{n;c}(t) - \Xi_{n;c}(u_{i-1}) = \sum_{j \geq 1}(\xi^{(i)}_{j;c})^+.
$
Now, set 
\begin{align}
    {\hat Y^m_{n;c}(\zeta_k)} & = 
    \bigg[
        \max_{i \in [k+1]}\mathbf{I}\Big( \sum_{q = 1}^{i-1}\sum_{j \geq 0}\xi^{(q),m}_{j;c} + \sum_{q = 1}^{i-1}z_q + \hat M^{(i),m}_{n;c}(\zeta_k) \geq na \Big)
    \bigg]
    \cdot 
    \mathbf{I}\Big(
    \sum_{q = 1}^{k+1}\sum_{j \geq 0}\xi^{(q),m}_{j;c} + \sum_{q = 1}^{k}z_q - cn \leq -nb
    \Big).
    \nonumber
\end{align}
In \eqref{def: Z_n, option pricing}, we plug in $\hat Y^m_{n;c} = \hat Y^m_{n;c}(J_n)$.

The proof of the strong efficiency is almost identical to that of Theorem~\ref{theorem: strong efficiency}.
The only major difference is that in Lemma~\ref{lemma: algo, bound term 4}, we apply Assumption~\ref{assumption: holder continuity strengthened on X < z t, alternative } instead of Assumption~\ref{assumption: holder continuity strengthened on X < z t}.

\ifshowreminders
\newpage
\newgeometry{left=1cm,right=1cm,top=0.5cm,bottom=1.5cm}
\footnotesize
\section*{\linkdest{location of reminders}Some Reminders}
    
\fi

\ifshownotationindex
\newpage
\section*{\linkdest{location of notation index}Notation Index}
\begin{itemize}[leftmargin=*]

    
\item
    \notationdef{def-strong-efficiency}{Strong efficiency}:
     $(L_n)_{n \geq 1}$ are {unbiased and strongly efficient} estimators of $(A_n)_{n \geq 1}$ if
     $
     \E L_n = \P(A_n) \ \forall n \geq 1,\
    \E L^2_n = \bm{O}\big(\P^2(A_n)\big)
     $
     as $n \to \infty$

\item 
    \notationidx{set-for-integers-below-n}{[k]}: $[k] = \{1,2,\cdots,k\}$ for any $k \in \mathbb{Z}^+$.

\item 
    $\notationidx{wedge-min-operator}{x \wedge y}:
    {x \wedge y}\delequal \min\{x,y\}$

\item 
    $\notationdef{vee-max-operator}{x\vee y}:
    {x\vee y}\delequal \max\{x,y\}$

\item 
    \notationidx{floor-operator}{$\floor{x}$}:
    $\floor{x}\delequal \max\{n \in \mathbb{Z}:\ n \leq x\}$.

\item 
    \notationidx{ceil-operator}{$\ceil{x}$}:
    $
    {\ceil{x}} \delequal  \min\{n \in \mathbb{Z}:\ n \geq x\}
    $

\item
    $\notationidx{notation-density-norm}{\norm{ f }_\infty}:
    {\norm{ f }_\infty} = \sup_{x \in \R}|f(x)|$ for the densify $f$ of some random variable $X$.

\item 
    $\notationidx{measure-restricted-on-A}{\mu|_A( \cdot )}:
    {\mu|_A( \cdot )}\delequal \mu( A \cap \cdot )$

\item
    $\notationidx{notation-def-measure-majorization}{(\mu_1 - \mu_2)|_A \geq 0}$:
    We say that Borel measure $\mu_1$ majorizes Borel measure $\mu_2$ when restricted on $A$ (denoted as ${(\mu_1 - \mu_2)|_A \geq 0}$)
    if $\mu(B \cap A) = \mu_1(B \cap A) - \mu_2(B \cap A) \geq 0$ for any Borel set $B \subset \R$.

\linkdest{location, notation index A}

\item
    $\notationidx{def-parameter-a-in-set-A}{a}$: parameter $a > 0$ in set $A$ defined in \eqref{def set A}

\item
    $\notationidx{def-target-set-A}{A}:
    A\delequal \{\xi \in \mathbb{D}: \sup_{t \in [0,1]}\xi(t)\geq a; \sup_{t \in (0,1] }\xi(t) - \xi(t-) < b \}$

\item 
    $\notationidx{def-rare-event-A-n}{A_n}:
    {A_n}\delequal \{\bar X_n \in A\}$

\item 
    $\notationidx{notation-set-A-Delta}{A^\Delta}:
    {A^\Delta} = \big\{\xi \in \mathbb{D}: \sup_{t \in [0,1]}\xi(t)\geq a - \Delta \big\}$

\item 
    $\notationidx{notation-heavy-tailed-index-alpha}{\alpha,\alpha^\prime}:
    {\alpha,\alpha^\prime}> 1$, heavy-tailed indices in Assumption \ref{assumption: heavy tailed process}

\linkdest{location, notation index B}

\item
    $\notationidx{def-parameter-b-in-set-A}{b}$: parameter $b > 0$ in set $A$ defined in \eqref{def set A}

\item
    $\notationidx{notation-set-B-gamma}{B^\gamma}:
    {B^\gamma}\delequal \{ \xi \in \mathbb{D}: \#\{ t \in [0,1]: \xi(t) - \xi(t-) \geq \gamma \} \geq l^* \}.$

\item 
    $\notationidx{notaiton-set-B-gamma-n}{B^\gamma_n}:
    {B^\gamma_n}\delequal \{\bar{X}_n \in B^\gamma\}$

\item
     $\notationidx{def-beta-blumenthal-getoor-index}{\beta}:
     \beta \delequal \inf\{p > 0: \int_{(-1,1)}|x|^p\nu(dx) < \infty\}$,
     Blumenthal-Getoor index.

\linkdest{location, notation index C}

\item 
    $\notationidx{constant-C-A2}{C}$: constant $C > 0$ in Assumption \ref{assumption: holder continuity strengthened on X < z t}

\item
    $
    \notationidx{measure-C-l-beta}{\mathbf C_{\beta}^l(\cdot)}:
    {\mathbf C_{\beta}^l(\cdot)}\delequal \E\Bigg[ \nu_\beta^l\big\{ y \in (0,\infty)^l:\ \sum_{j = 1}^l y_j \mathbf{I}_{[U_j,1]}\in\cdot\big\} \Bigg]
    $
    where $(U_j)_{j \geq 1}$ is an iid sequence of $\text{Unif}(0,1)$

\item 
    $\notationidx{measure-C-j,k}{\mathbf C_{j,k}(\cdot)}
    \delequal 
    \E\Bigg[
    \nu^j_\alpha \times \nu^k_{\alpha^\prime}
    \bigg\{
    (x,y) \in (0,\infty)^j \times (0,\infty)^k:\ 
    \sum_{l = 1}^j x_l \mathbf{I}_{ [U_l,1] } - \sum_{m = 1}^k y_m \mathbf{I}_{[V_m,1]} \in \cdot  
    \bigg\}
    \Bigg]$
    where $U_l,V_m$ are two independent sequences of iid Unif$(0,1)$ RVs.

\linkdest{location, notation index D}

\item 
    $\notationidx{algo-param-d}{d}$:
    Parameter of the algorithm $d > 0$, determining the $\ceil{\log_2(n^d)}$ term in ${\hat M^{(i),m}_n(\zeta)}$ defined in \eqref{def hat M i m}

\item 
    $\notationidx{cadlag-space-D-01}{\mathbb{D}}:
    {\mathbb{D}}= \mathbb{D}_{[0,1],\R}$, the space of all real-valued càdlàg functions with domain $[0,1]$

\item 
    $\notationidx{set-D-l}{\mathbb{D}_l}$:
     subset of $\mathbb D$ containing all the non-decreasing step functions that has $l$ jumps and vanishes at the origin, with convention $\D_0 =\{\bm 0\}$ where $\bm{0}(t) \equiv 0$ is the zero function.

\item 
    $\notationidx{set-D-<-l}{\mathbb D_{<l}}:
    {\mathbb D_{<l}} = \cup_{j = 0,1,\cdots,l-1}\mathbb{D}_l$

\item
    $\notationidx{def-set-D-j,k}{\mathbb D_{j,k}}$:
    the set containing all step functions in $\mathbb D$ vanishing at the origin that has exactly $j$ upward jumps and $k$ downward jumps, with convention $\mathbb D_{0,0} = \{\bm 0\}$

\item 
    $\notationidx{def-set-D-<-j,k}{\mathbb D_{<j,k}}:
    {\mathbb D_{<j,k}} \delequal \bigcup_{ (l,m) \in \mathbb I_{<j,k} }\mathbb D_{l,m}$

\item
    $\notationidx{skorokhod-j1-metric-d}{\bm{d}(x,y)}:
    {\bm{d}(x,y)} \delequal \inf_{\lambda \in \Lambda} \sup_{t \in [0,1]}|\lambda(t) - t| \vee |x(\lambda(t)) - y(t)|$ 
    with $\Lambda$ being the set of all increasing homeomorphisms from $[0,1]$ to itself.

\linkdest{location, notation index E}

\item
    $\notationidx{notation-set-E-in-estimator}{E}:
    E \delequal \{ \xi \in \mathbb{D}:\ \sup_{t \in (0,1]}\xi(t) - \xi(t-) < b  \}$

\item 
    $\notationidx{notation-event-bar-En}{E_n}:
    E_n \delequal \{\bar J_n \in E\}$

\linkdest{location, notation index F}

\linkdest{location, notation index G}

\item 
    $\notationidx{algo-param-gamma}{\gamma}$:
    Parameter of the algorithm $\gamma \in (0,b)$ characterizing the set $B^\gamma$ defined in \eqref{def set B gamma}

\linkdest{location, notation index H}

\linkdest{location, notation index I}

\item 
    $\notationidx{index-set-I-<-j,k}{\mathbb I_{<j,k}}:
    {\mathbb I_{<j,k}} \delequal \big\{ (l,m) \in \mathbb N^2 \symbol{92} \{(j,k)\}:\ l(\alpha-1) + m(\alpha^\prime - 1) \leq j(\alpha - 1) + k(\alpha^\prime - 1) \big\}.$

\item 
    $\notationidx{notation-partition-I-i}{I_i}$: Given $\zeta_k = \sum_{i = 1}^k z_i \mathbf{I}_{[u_i,n]}$, ${I_i} = [u_{i-1},u_i)\ \forall i \in [k]$ and $I_{k+1} = [u_k,1].$

\linkdest{location, notation index J}

\item
    $\notationidx{notation-process-Jn}{J_n(t)}:
    {J_n(t)}\delequal \sum_{s \leq t}\Delta X(t) \mathbf{I}\big( \Delta X(t) \geq n\gamma \big)$

\item 
    $\notationidx{notation-process-bar-Jn}{\bar J_n}:
    \bar J_n(t) = \frac{1}{n}J_n(nt), \bar J_n = \{\bar J_n(t):\ t \in [0,1]\}$

 \linkdest{location, notation index K}

 \item 
    $\notationidx{algo-param-kappa}{\kappa}$:
    Parameter of the algorithm, $\kappa \in (0,1)$ decides the truncation level of $\kappa_{n,m}$ in \eqref{def: kappa n m}

\item
    $\notationidx{notation-kappa-n,m}{\kappa_{n,m}}:
    {\kappa_{n,m}} = \frac{\kappa^m}{n^r}$;
    the truncation level for ARA in $\breve \Xi^m_n$ in \eqref{def breve Xi n m, ARA}

\linkdest{location, notation index L}
\item
    $\notationidx{notation-l*-jump-number}{l^*}:
    {l^*}\delequal \ceil{a/b}$

\item 
    $\notationidx{notation-stick-length-l-i-j}{l^{(i)}_j}$:
    $
    l^{(i)}_1 = V^{(i)}_1(u_{i+1} - u_i),\ 
    {l^{(i)}_j} = V^{(i)}_j(u_{i+1} - u_i - l^{(i)}_1 - l^{(i)}_2 -\cdots - l^{(i)}_{j-1})
    $
    where $V^{(i)}_j$ is an iid sequence of Unif$(0,1)$.

\item 
    $\notationidx{notation-stick-length-l-j-l}{l_j(l)}$:
    Given $l \geq 0$, sample
    $l_1(l) = V_1 \cdot l$ and ${l_j(l)} = V_j \cdot (l - l_1 - l_2 - \cdots - l_{j-1})\ \forall j \geq 2$
    where $V_j$ is an iid sequence of Unif$(0,1)$ RVs.

\item
    $\notationidx{notation-IS-esitmator-Ln}{L_n} \delequal Z_n \frac{d\P}{d\Q_n} 
    =
    \frac{Z_n}{ w + \frac{1 - w}{ \P(B^\gamma_n) }\mathbf{I}_{ B^\gamma_n } }$

\item 
    $\notationidx{law-of-X}{\mathscr{L}(X)}$: Law of the random variable $X$

\item 
      $\notationidx{law-of-X-on-A}{\mathscr{L}(X|A)}$: Law of random variable $X$ when conditioning on $\{X \in A\}$.

\item   
    $\notationidx{constant-lambda-A2}{\lambda}$: constant $\lambda > 0$ in Assumption \ref{assumption: holder continuity strengthened on X < z t}

\linkdest{location, notation index M}
    
\item 
    $\notationidx{notaiton-running-supremum-Mt}{M(t)}:
    {M(t)}\delequal \sup_{s \leq t}X(s)$

\item
    $\notationidx{notation-M-n-i}{M_{n}^{(i),*}(\zeta_k)}:
    {M_{n}^{(i),*}(\zeta_k)} \delequal \sup_{ t \in I_i}\Xi_n(t) - \Xi_n(u_{i-1})$
    given $\zeta_k(t) = \sum_{i = 1}^k z_i \mathbf{I}_{[u_i,1]}(t)$

\item 
    $\notationidx{notation-hat-M-n-i-m}{\hat M^{(i),m}_n(\zeta_k)}$:
    Approximators to ${M_{n}^{(i),*}(\zeta_k)}$.
    In Algorithm~\ref{algoISnoARA} (i.e., without ARA), we set
    ${\hat M^{(i),m}_n(\zeta_k)}= \sum_{j = 1}^{m + \ceil{\log_2(n^d)}  }( \xi^{(i)}_j)^+$.
    In Algorithm~\ref{algoIS} (i.e., with ARA), we set
    ${\hat M^{(i),m}_n(\zeta_k)}= \sum_{j = 1}^{m + \ceil{\log_2(n^d)}  }( \xi^{(i),m}_j)^+$.

\linkdest{location, notation index N}

\item 
    $\notationidx{def-measure-nu-beta}{\nu_\beta}$: 
    the measure concentrated on $(0,\infty)$ with $\nu_\beta(x,\infty) = x^{-\beta}$

\item 
    $\notationidx{def-nu-beta-l-fold}{\nu^l_\beta}$: 
    the $l-$fold product measure of $\nu_\beta$ restricted on $\{ y \in (0,\infty)^l:\ y_1 \geq y_2 \geq \cdots \geq y_l \}$





    


\linkdest{location, notation index Q}

\item
    $\notationidx{notation-ISDM-measure-Qn}{\mathbf{Q}_n(\cdot)}:
    {\mathbf{Q}_n(\cdot)}\delequal w\P(\cdot) + (1 - w) \P(\ \cdot\ | B^\gamma_n).$

\linkdest{location, notation index R}

\item 
    $\notationidx{algo-param-r}{r}$:
    Parameter of the algorithm, $r > 0$ decides the truncation level of $\kappa_{n,m}$ in \eqref{def: kappa n m}

\item 
    $\notationidx{algo-param-rho}{\rho}$:
    Parameter of the algorithm, $\rho \in (0,1)$ is the rate of decay in the law of $\tau \sim \text{Geom}(\rho)$

\item 
    \notationidx{notation-RV-LDP}{$\RV_\beta$}:
    $\phi \in \RV_\beta$ (as $x \rightarrow \infty$) if $\lim_{x \rightarrow \infty}\phi(tx)/\phi(x) = t^\beta$ for any $t>0$;
    $\phi \in \RV_\beta(\eta)$ (as $\eta \downarrow 0$) if $\lim_{\eta \downarrow 0}\phi(t\eta)/\phi(\eta) = t^\beta$ for any $t>0$

    
\linkdest{location, notation index S}

\item 
    $\notationidx{notation-bar-sigma-function}{\bar\sigma^2(\cdot)}:
    {\bar\sigma^2}(c) \delequal \int_{(-c,c)} x^2 \nu(dx)$

\linkdest{location, notation index T}

\item 
    $\notationidx{notation-truncation-time-tau}{\tau}$: $\text{Geom}(\rho)$ that is independent of everything else; the truncation index in $Z_n$

\linkdest{location, notation index W}

\item 
    $\notationidx{algo-param-w}{w}$:
    Parameter of the algorithm $w \in (0,1)$, determining the relative weights of the defensive mixture in the importance sampling distribution $\Q_n$ in \eqref{def: IS distribution Qn}

\linkdest{location, notation index X}

\item   
    $\notationidx{Levy-process-X}{X(t)}$: L\'evy process with generating triplet $(c_X,\sigma,\nu)$
    where the L\'evy measure $\nu$ satisfies Assumption \ref{assumption: heavy tailed process}

\item
    $\notationidx{process-X-<z}{X^{<z}(t)}$: the L\'evy process with generating triplet $(c_X,\sigma,\nu|_{(-\infty,z)})$.
    That is, $X^{<z}$ is the modulated version of $X$ where all the upward jumps with size larger than $z$ are removed.

\item
    $\notationidx{notation-X-jump-bounded-by-c}{X^{(-c,c)}(t)}$: Given any $c \in (0,1)$, $X^{(-c,c)}(t)$ is the L\'evy process with generating triplet $(0,0,\nu|_{(-c,c)})$. That is, it is jump martingale with all jumps bounded by $c$.

\item 
    $\notationidx{scaled-process-bar-X-n}{\bar X_n}:
    \bar X_n = \{\bar X_n(t):\ t \in [0,1]\},\ {\bar X_n(t)} = \frac{1}{n}X(nt)
    $

\item
    $\notationidx{notation-increment-on-sticks-xi-i-j-m}{\xi^{(i),m}_j,\xi^{(i)}_j}$:
    Conditioning on the values of $l^{(i)}_j$, sample
    $\big(\xi^{(i)}_j,\xi^{(i),1}_j,\xi^{(i),2}_j,\xi^{(i),3}_j,\cdots) \distequal \Big(\Xi_n(l^{(i)}_j),\Breve{\Xi}^1_n(l^{(i)}_j),\Breve{\Xi}^2_n(l^{(i)}_j),\Breve{\Xi}^3_n(l^{(i)}_j),\cdots\Big). $

\item
    $\notationidx{notation-increments-on-sticks-xi-n-j-l}{\xi^{[n]}_j(l),\xi^{[n];m}_j(l)}$:
    Conditioning on the values of $l_j(l)$, sample 
    $
    \big(\xi^{[n]}_j(l),\xi^{[n],1}_j(l),\xi^{[n],2}_j(l),\xi^{[n],3}_j(l),\cdots \big)
    =
    \Big( \Xi_n\big(l_j(l)\big), \breve \Xi^1_n\big(l_j(l)\big),\breve \Xi^2_n\big(l_j(l)\big),\breve \Xi^3_n\big(l_j(l)\big),\cdots \Big)
    $

\item
    $\notationidx{notation-process-Xi-n}{\Xi_n(t)}:
    {\Xi_n(t)}\delequal X(t) - J_n(t) = X(t) - \sum_{s \leq t}\Delta X(t) \mathbf{I}\big( \Delta X(t) \geq n\gamma \big).$
    That is, it is a L\'evy process with the law of $X^{<n\gamma}$

\item 
    $\notationidx{notation-process-bar-Xi-n}{\bar \Xi_n}:
    {\bar \Xi_n(t)} = \frac{1}{n}\Xi_n(nt), \bar \Xi_n = \{\bar \Xi_n(t):\ t \in [0,1]\}$

\item
    $\notationidx{process-breve-Xi-m-n}{\Breve{\Xi}^{m}_n(t)}$:
    A modulated version of L\'evy process $\Xi_n(t)$ where the jump martingale of jumps bounded by $\kappa_{n,m}$ is substituted by a standard Brownian motion with the same variance; see \eqref{def breve Xi n m, ARA} for the definition.

\linkdest{location, notation index Y}

\item
    $\notationidx{notation-Y-*-n}{Y^*_n}:
    {Y^*_n}\delequal \mathbf{I}\big( M(n) \geq na \big).$
    Note that $Y^*_n(J_n) = \mathbf{I}\big( M(n) \geq na \big)$ for $Y^*_n(\cdot)$ defined right below.

\item 
    $\notationidx{notation-Y-*-n-zeta}{Y^*_n(\zeta_k)}:
    {Y^*_n(\zeta_k)}
    \delequal
    \max_{i \in [k+1]}\mathbf{I}\Big( \Xi_n(u_{i-1}) + \zeta(u_{i-1}) + M^{(i),*}_n \geq na \Big)
    =
    \max_{i \in [k+1]}\mathbf{I}\Big( \sum_{q = 1}^{i-1}\sum_{j \geq 0}\xi^{(q)}_j + \sum_{q = 1}^{i-1}z_q + M^{(i),*}_n \geq na \Big)
    $
    given 
    $\zeta_k(t) = \sum_{i = 1}^k z_i \mathbf{I}_{[u_i,1]}(t)$.

\item
    $\notationidx{notation-approximation-hat-Y-m-n}{\hat Y^m_n(\zeta_k)}:$
    Approximators to $Y^*_n(\zeta_k)$.
    In Algorithm~\ref{algoISnoARA} (i.e., without ARA), we set
    $
    {\hat Y^m_n(\zeta_k)} = \max_{i \in [k+1]}\mathbf{I}\Big( \sum_{q = 1}^{i-1}\sum_{j \geq 0}\xi^{(q)}_j + \sum_{q = 1}^{i-1}z_q + \hat M^{(i),m}_n(\zeta) \geq na \Big)$.
    In Algorithm~\ref{algoIS} (i.e., with ARA), we set
    $
    {\hat Y^m_n(\zeta_k)} = \max_{i \in [k+1]}\mathbf{I}\Big( \sum_{q = 1}^{i-1}\sum_{j \geq 0}\xi^{(q),m}_j + \sum_{q = 1}^{i-1}z_q + \hat M^{(i),m}_n(\zeta) \geq na \Big)$.

\linkdest{location, notation index Z}

\item 
    $\notationidx{constant-z0-A2}{z_0}$: constant $z_0 > 0$ in Assumption \ref{assumption: holder continuity strengthened on X < z t}

\item
    $ \notationidx{notation-estimator-Zn}{Z_n}:
    {Z_n}= \sum_{m = 0}^{ \tau}\frac{ \hat Y^m_n - \hat Y^{m-1}_n }{ \P(\tau \geq m) }\mathbf{I}_{E_n}$

\end{itemize}
\fi

\ifshowtheoremtree
\newpage
\section*{\linkdest{location of theorem tree}Theorem Tree}
\begin{thmdependence}[leftmargin=*]

\thmtreenode{-}
    {Proposition}{proposition: design of Zn}{0.8}{}
    \begin{thmdependence}
        \thmtreeref{Result}{resultDebias}
        \thmtreenode{-}
            {Lemma}{lemma: LD, events A n}{0.8}{}
            \begin{thmdependence}
                \thmtreeref{Result}{result: LD of Levy, two-sided}
            \end{thmdependence}

        \thmtreenode{-}
            {Lemma}{lemma: LD, event A Delta}{0.8}{}
            \begin{thmdependence}
                \thmtreeref{Result}{result: LD of Levy, two-sided}
            \end{thmdependence}
    \end{thmdependence}

\bigskip 

\thmtreenode{-}
    {Theorem}{theorem: strong efficiency}{0.8}{}
    \begin{thmdependence}
        \thmtreeref{Proposition}{proposition: design of Zn}
        \thmtreeref{Proposition}{proposition, intermediate 1, strong efficiency ARA}
        \thmtreeref{Proposition}{proposition, intermediate 2, strong efficiency ARA}
    \end{thmdependence}

\thmtreenode{-}
    {Theorem}{theorem: strong efficiency without ARA}{0.8}{}
    \begin{thmdependence}
        \thmtreeref{Theorem}{theorem: strong efficiency}
    \end{thmdependence}

\bigskip 

\thmtreenode{-}
    {Proposition}{proposition, intermediate 1, strong efficiency ARA}{0.8}{}
    \begin{thmdependence}
        \thmtreenode{-}
            {Lemma}{lemma: algo, bound term 1}{0.8}{}
            \begin{thmdependence}
                \thmtreeref{Result}{result: bound bar sigma}
            \end{thmdependence}

        \thmtreenode{-}
            {Lemma}{lemma: algo, bound term 2}{0.8}{}
            \begin{thmdependence}
                \thmtreeref{Result}{result: bound bar sigma}
            \end{thmdependence}

        \thmtreenode{-}
            {Lemma}{lemma: algo, bound term 3}{0.8}{}
            \begin{thmdependence}
                \thmtreeref{Result}{result: concave majorant of Levy}

                \thmtreenode{-}
                    {Lemma}{lemma: algo, bound supremum of truncated X}{0.8}{}
            \end{thmdependence}
    \end{thmdependence}

\bigskip

\thmtreenode{-}
    {Proposition}{proposition, intermediate 2, strong efficiency ARA}{0.8}{}
    \begin{thmdependence}
        \thmtreenode{-}
            {Lemma}{lemma: algo, bound term 4}{0.8}{}
            \begin{thmdependence}
                \thmtreeref{Lemma}{lemma: algo, bound supremum of truncated X}
            \end{thmdependence}
    \end{thmdependence}

\bigskip 
\thmtreenode{-}
    {Theorem}{ CorollaryRVlevyMeasureAtOrigin }{0.8}{}
    \begin{thmdependence}
        \thmtreenode{-}
            {Proposition}{proposition: lipschitz cont, 1}{0.8}{}
    \end{thmdependence}

\thmtreenode{-}
    {Theorem}{ CorollarySemiStableLevyMeasureAtOrigin }{0.8}{}
    \begin{thmdependence}
        \thmtreenode{-}
            {Proposition}{proposition: lipschitz cont, 2}{0.8}{}
    \end{thmdependence}

\end{thmdependence}
\fi

\ifshowtheoremlist
\newpage
\linkdest{location of theorem list}
\listoftheorems
\fi

\ifshowequationlist
\newpage
\linkdest{location of equation number list}
\section*{Numbered Equations}
\fi

\ifshownavigationpage
\newpage
\normalsize
\noindent
\hyperlink{location of notation index}{Notation Index}
\begin{itemize}
\item[] 
    \hyperlink{location, notation index A}{A},
    \hyperlink{location, notation index B}{B},
    \hyperlink{location, notation index C}{C},
    \hyperlink{location, notation index D}{D},
    \hyperlink{location, notation index E}{E},
    \hyperlink{location, notation index F}{F},
    \hyperlink{location, notation index G}{G},
    \hyperlink{location, notation index H}{H},
    \hyperlink{location, notation index I}{I},
    \hyperlink{location, notation index J}{J},
    \hyperlink{location, notation index K}{K},
    \hyperlink{location, notation index L}{L},
    \hyperlink{location, notation index M}{M},
    \hyperlink{location, notation index N}{N},
    \hyperlink{location, notation index O}{O},
    \hyperlink{location, notation index P}{P},
    \hyperlink{location, notation index Q}{Q},
    \hyperlink{location, notation index R}{R},
    \hyperlink{location, notation index S}{S},
    \hyperlink{location, notation index T}{T},
    \hyperlink{location, notation index U}{U},
    \hyperlink{location, notation index V}{V},
    \hyperlink{location, notation index W}{W},
    \hyperlink{location, notation index X}{X},
    \hyperlink{location, notation index Y}{Y},
    \hyperlink{location, notation index Z}{Z}
\end{itemize}

\tableofcontents

\section*{Navigation Links}
\fi

\end{document}